\documentclass{article}
\usepackage{fancyhdr}
\usepackage{amscd}
\usepackage{graphicx}
\usepackage{amsmath}
\usepackage{amssymb}
\usepackage{amsthm}
\usepackage{mathrsfs}
\usepackage{bm}
\usepackage{url}

\usepackage{amsfonts}
\usepackage{mathtools}

\usepackage{tikz}
\usepackage{dsfont}

\newtheorem{dfn}{Definition}[section]
\newtheorem{thm}[dfn]{Theorem}
\newtheorem{prop}[dfn]{Proposition}
\newtheorem{lem}[dfn]{Lemma}

\newtheorem{rem}[dfn]{Remark}

\newtheorem{prob}[dfn]{Problem}

\newcommand{\R}{\mathbb{R}}

\newcommand{\N}{\mathbb{N}}

\newcommand{\tx}{\tilde{x}}

\newcommand{\ba}{\bar{a}}
\newcommand{\br}{\bar{r}}
\newcommand{\ta}{\tilde{a}}
\newcommand{\bydef}{\,\stackrel{\mbox{\tiny\textnormal{\raisebox{0ex}[0ex][0ex]{def}}}}{=}\,} 

\newcommand{\dagA}{A^\dagger}

\newcommand{\inc}{\bm{\iota}^{(N)}}

\numberwithin{equation}{section}

\definecolor{burgundy}{rgb}{0.5, 0.0, 0.13}

\usepackage{color}
\def\corcommstyle{\bf\small\tt}

\def\corrl #1<<#2||#3>>{
\if\visiblecomments y
  \begin{quote} {\corcommstyle $<<$COMMENT$>>$ {\color{red}#1\marginpar{!!}}\\$<<$OLD$<<$} \end{quote}

{\color{red} 
 #2
 }

  \begin{quote} {\corcommstyle ==NEW== } \end{quote}
   \noindent\hrulefill
 
\vspace{-10pt} 
 
 \noindent\hrulefill
 
 \vspace{-10pt} 
 
 \noindent\dotfill
 
  #3
  
   \noindent\dotfill 

\vspace{-10pt} 
 
 \noindent\hrulefill
 
 \vspace{-10pt} 
 
 \noindent\hrulefill
  \begin{quote} {\corcommstyle $>>$END$>>$ } \end{quote}
 \else
  #3
 \fi
}

\long\def\longcorrl #1<<#2||#3>>{
\if\visiblecomments y
  \begin{quote} {\corcommstyle $<<$COMMENT$>>$ {\color{red}#1\marginpar{!!}}\\$<<$OLD$<<$} \end{quote}
 
 {\color{red}

  #2
  
  }
  
  \begin{quote} {\corcommstyle ==NEW== } \end{quote}
  
    \noindent\hrulefill
 
\vspace{-10pt} 
 
 \noindent\hrulefill
 
 \vspace{-10pt} 
 
 \noindent\dotfill
 
  #3
  
   \noindent\dotfill 

\vspace{-10pt} 
 
 \noindent\hrulefill
 
 \vspace{-10pt} 
 
 \noindent\hrulefill
  \begin{quote} {\corcommstyle $>>$END$>>$ } \end{quote}
 \else
  #3
 \fi
}

\def\mlabel #1
{
  \if\visiblecomments y
     \marginpar[\flushright \bf \footnotesize #1]{\bf \footnotesize #1}
  \fi
}

\def\flabel #1
{
  \if\visiblecomments y
       \hbox{\bf\footnotesize #1}
  \fi
}

\def\corrq #1<<#2>>{
\if\visiblecomments y
  \begin{quote} {\corcommstyle $<<$COMMENT$>>$ #1\marginpar{!!}\\$<<$BEG$<<$} \end{quote}
  \noindent\hrulefill
 
\vspace{-10pt} 
 
 \noindent\hrulefill
 
 \vspace{-10pt} 
 
 \noindent\dotfill
 
 {\color{red}
  
  #2
  
  }
   
  \noindent\dotfill 

\vspace{-10pt} 
 
 \noindent\hrulefill
 
 \vspace{-10pt} 
 
 \noindent\hrulefill 
  \begin{quote} {\corcommstyle $>>$END$>>$ } \end{quote}
 \else
  #2
 \fi
}

\long\def\longcorrq #1<<#2>>{
\if\visiblecomments y
  \begin{quote} {\corcommstyle $<<$COMMENT$>>$ #1\marginpar{!!}\\$<<$BEG$<<$} \end{quote}
  \noindent\hrulefill
 
\vspace{-10pt} 
 
 \noindent\hrulefill
 
 \vspace{-10pt} 
 
 \noindent\dotfill

  #2

  \noindent\dotfill 

\vspace{-10pt} 
 
 \noindent\hrulefill
 
 \vspace{-10pt} 
 
 \noindent\hrulefill 
  \begin{quote} {\corcommstyle $>>$END$>>$ } \end{quote}
 \else
  #2
 \fi
}

\def\corrc #1<<>>{
\if\visiblecomments y
  \begin{quote} {\corcommstyle $<<$COMMENT$>>$ \color{red} #1\marginpar{!!}} \end{quote}
\fi
}

\def\corre #1<<#2||#3>>{
\if\visiblecomments y
  #3\marginpar{\corcommstyle #1}
 \else
  #3
 \fi
}

\long\def\longcorre #1<<#2||#3>>{
\if\visiblecomments y
  #3\marginpar{\corcommstyle #1}
 \else
  #3
 \fi
}

\def\corrs #1<<#2||#3>>{
\if\visiblecomments y
  #3\marginpar{\corcommstyle #2 $\rightarrow$ #3\\ #1}
 \else
  #3
 \fi
}

\def\corro #1<<#2||#3>>{
#2}

\def\corrn #1<<#2||#3>>{
#3}

\long\def\longcorro #1<<#2||#3>>{
#2}

\long\def\longcorrn #1<<#2||#3>>{
#3}

\long\def\underconstruction #1<<<#2>>>{
\if\visiblecomments y
  \begin{quote} {\corcommstyle $<<$UNDER CONSTRUCTION - BEGIN$>>$ #1\marginpar{!!}} \end{quote}
   {\color{red}
  #2
  }
  \begin{quote} {\corcommstyle $>>$UNDER CONSTRUCTION - END$>>$ } \end{quote}
 \else
 \fi
}

\def\showcomments{
  \let\visiblecomments y
}

\def\hidecomments{
  \let\visiblecomments n
}

\showcomments

\newcommand{\KMa}[1]{{\color{black} #1}}
\newcommand{\JP}[1]{{\color{black} #1}}

\begin{document}

\title{\KMa{Saddle-Type Blow-Up Solutions with Computer-Assisted Proofs: Validation and Extraction of Global Nature}}

\author{
Jean-Philippe Lessard
\thanks{McGill University, Department of Mathematics and
Statistics, 805 Sherbrooke Street West, Montreal, QC, H3A 0B9, Canada ({\tt jp.lessard@mcgill.ca})}
\and
Kaname Matsue
\thanks{(Corresponding author)} $^{,}$ \footnote{Institute of Mathematics for Industry, Kyushu University, Fukuoka 819-0395, Japan ({\tt kmatsue@imi.kyushu-u.ac.jp})} $^{,}$ \footnote{International Institute for Carbon-Neutral Energy Research (WPI-I$^2$CNER), Kyushu University, Fukuoka 819-0395, Japan} 
\and 
Akitoshi Takayasu
\thanks{Faculty of Engineering, Information and Systems, University of Tsukuba, 1-1-1 Tennodai, Tsukuba, Ibaraki 305-8573, Japan (\texttt{takitoshi@risk.tsukuba.ac.jp})}
}
\maketitle

\begin{abstract}
In this paper, blow-up solutions of autonomous ordinary differential equations (ODEs) which are unstable under perturbations of initial \KMa{points}, \KMa{referred to as {\em saddle-type blow-up solutions}}, are studied. 
Combining dynamical systems machinery (e.g., compactifications, time-scale desingularizations of vector fields) with tools from computer-assisted proofs (e.g., rigorous integrators, the parameterization method for invariant manifolds), \KMa{these blow-up solutions are obtained as trajectories on local stable manifolds of hyperbolic saddle equilibria at infinity}. 
\KMa{With the help of computer-assisted proofs, global trajectories on stable manifolds, inducing blow-up solutions, \KMa{provide} a global picture organized by global-in-time solutions and blow-up solutions simultaneously.}
\KMa{Using the proposed methodology, intrinsic features of saddle-type blow-ups are observed}: \KMa{locally} smooth dependence of blow-up times on \KMa{initial points}, level set distribution of blow-up times, and \KMa{decomposition} of the phase space \KMa{playing a role as separatrixes among solutions, where the magnitude of initial points near those blow-ups does not matter for asymptotic behavior. 
Finally, singular behavior of blow-up times on initial points belonging to different family of blow-up solutions is addressed.}
\end{abstract} 



\section{Introduction}

\label{sec:intro}

Our concern in the present paper is blow-up solutions of the following initial value problem of an autonomous system of ordinary differential equations (ODEs) in $\mathbb{R}^n$:
\begin{equation}
\label{eqn:ODE}
\frac{dy(t)}{dt}=f (y(t)),\quad y(0)=y_0,
\end{equation}
where $t\in[0,T)$ with $0<T\le\infty$, $f:\mathbb{R}^n\to\mathbb{R}^n$ is a $C^1$ function and $y_0\in\mathbb{R}^n$.
We call a solution $y(t)$ of the initial value problem (\ref{eqn:ODE}) a {\em blow-up solution} if
\[
	t_{\max}\bydef\sup\left\{\bar t\mid \mbox{a solution $y\in C^1([0,\bar t))$ of \eqref{eqn:ODE} exists}\right\} < \infty.
\]
The maximal existence time $t_{\max}$ is then called the {\em blow-up time} of \eqref{eqn:ODE}.
Blow-up solutions can be seen in many dynamical systems generated by ODEs, or partial differential equations (PDEs) like nonlinear heat equations or Keller-Segel systems. These dynamical systems are categorized as exhibiting finite-time singularities, and have been the center of attention of many researchers, who have studied these phenomena from mathematical, physical, numerical viewpoints and so on  (e.g. \cite{FM2002, HV1997, M2016, W2013} from theoretical viewpoints and e.g. \cite{AIU2017, BK1988, C2016, CHO2007, ZS2017} from numerical viewpoints).
Fundamental questions for blow-up \KMa{solutions} are {\em whether or not a solution blows up} and, if it does, {\em when, where, and how} it blows up. In general, blow-up phenomena depend on initial \KMa{points}, and rigorously characterizing them as functions of initial \KMa{points} remains nontrivial.
\par
A typical approach for studying and proving existence of blow-up solutions is via {\em energy estimates} (see e.g. \cite{FM2002}), namely inequalities (involving energy functionals associated with the systems) giving sufficient conditions for existence of blow-up. In such cases, relatively large initial data induce finite-time blow-up.
However, in general, these criteria do not provide an answer on how large initial points should be to exhibit blow-up and how solutions behave when these criteria are violated.
There are several cases where initial points are divided such that solutions through them either exist globally in time or blow-up by means of \KMa{{\em bounded}} stationary solutions (e.g. \cite{F1969}).
A stationary solution with the above property is referred to as {\em the separatrix}, which plays a key role in describing asymptotic behavior of solutions. Despite their importance, results about the existence and explicit description of \KMa{separatrixes} are limited. 
On the other hand, there are also results about \KMa{the existence of} blow-ups in which the magnitude of initial points does not matter. 
Alternative approaches to the energy estimates have been introduced to prove such blow-ups, but their dependence on initial points remain unknown in many cases, while arguments based on energy estimates easily yield the continuous dependence of blow-up behavior on initial points by continuity of energy functionals.
Furthermore, there are also blow-up solutions whose asymptotic behavior is described not only by divergence, but also by complex behavior like oscillations, some of which are mentioned in Section \ref{section-remark-intro} (Concluding Remarks).
Mathematical and physical importance for studying blow-up behavior follow from such rich nature, but their comprehensive understanding are limited to well-known systems like PDEs mentioned above at present.
See e.g. \cite{FM2002, GV2002} for more detailed summaries of blow-up problems including another well-known characterization of blow-up solutions by means of (backward) {\em self-similarity}.
\par
Meanwhile, the second author has recently proposed a description of blow-up solutions from the viewpoint of dynamical systems \KMa{(\cite{Mat2018})}.
More precisely, {\em compactifications} of the phase space $\mathbb{R}^n$ is applied to mapping the infinity onto points on the boundary $\mathcal{E}$ of a compact manifold or their tangent spaces denoted by $\overline{\mathcal{D}}$ with $\partial \mathcal{D} = \mathcal{E}$.
\KMa{The boundary $\mathcal{E}$ shall be called {\em the horizon} in this context.}
Accordingly the vector field (\ref{eqn:ODE}) is transformed to one on the corresponding manifolds, but the behavior of solutions near the boundary \KMa{$\mathcal{E}$} is still singular reflecting the behavior of the original vector field \KMa{at infinity}.
The time-scale transformation, which shall be called the {\em time-scale desingularization}, is then introduced to desingularize the singularity of the vector field around $\mathcal{E}$.
Consequently, {\em dynamics at infinity} can be characterized through the time-transformed vector field, called {\em the desingularized vector field}, on $\overline{\mathcal{D}}$. 
Standard arguments in the theory of dynamical systems through compactifications show that divergent solutions of (\ref{eqn:ODE}) correspond to \KMa{global-in-time} solutions of the desingularized vector field converging to invariant sets on $\mathcal{E}$\footnote{
The above ideas themselves are applied to describe dynamics around bounded invariant sets in several preceding works (e.g. \cite{DH1999}).
}.
A significant consequence of the preceding studies is that, a solution of (\ref{eqn:ODE}) with bounded initial point is a blow-up solution, namely $t_{\max} < \infty$, if the image of the solution through a compactification mentioned above is on the \KMa{local} stable manifold of a {\em hyperbolic} equilibrium on $\mathcal{E}$ for the desingularized vector field\footnote{
The same conclusion holds for hyperbolic periodic orbits on $\mathcal{E}$.
A brief comment about the statement is mentioned in Remark \ref{rem-intro-per-blowup}.
Several theoretical generalizations are discussed in \cite{Mat2019}.
}.
\par
Simultaneously, the second and the third authors have developed a {\em computer-assisted} methodology for proving the existence of blow-up solutions for concretely given dynamical systems with rigorous \KMa{bounds} of their blow-up times $t_{\max}$ \cite{MT2020_1, MT2020_2, TMSTMO2017}.
The basic idea is the combination of compactifications as well as time-scale desingularizations mentioned above with rigorous integrator of ODEs based on interval (and affine) arithmetic and topological characterizations of asymptotic behavior such as locally defined Lyapunov functions.
\KMa{Evaluation of $t_{\max}$ is one of the most important issues in blow-up studies to estimate upper bounds of the existence of solutions, or the onset of finite-time singularities such as ignition in combustion studies (e.g., \cite{D1985}), while the study is limited even in numerical studies (e.g., \cite{C2016}).
The proposed methodology provides a rigorous and standard way to obtain both lower and upper bounds of $t_{\max}$ through dynamics at infinity.}
\par
The methodology works successfully for validating profiles and blow-up times of blow-up solutions generated by hyperbolic {\em stable} equilibria at infinity, while blow-up generated by unstable equilibria at infinity is not reported yet due to several technical difficulties.
Note that there is another work for characterizing blow-up solutions with computer assistance by the first author and his collaborators based on analytic approach \cite{DALP2015} whose detail is briefly mentioned in Section \ref{section-remark-intro}.
On the other hand, from the viewpoint of dynamics at infinity itself, namely when the viewpoint of blow-up characterizations is not considered, asymptotic behavior of {\rm unstable invariant sets at infinity} is quite natural to study towards description of global bounded dynamics (e.g. \cite{D2006, DH1999, DLA2006, GKO2020, KR2004}).
We then believe that blow-up solutions generated by unstable invariant sets at infinity contribute towards the comprehensive understanding of global dynamics, including characteristics such as criteria for the existence, dependence on initial points and analytic information of blow-up times.
Despite many mathematical and numerical studies of blow-ups, characterizations and computations of blow-up solutions which are unstable under perturbations of initial points in a standard way are not realistic, because we have to treat two numerical difficulties simultaneously:

\begin{itemize}
\item instability of \KMa{trajectories exhibiting blow-up solutions} under perturbations of initial points, and
\item treatment of infinity.
\end{itemize}
We shall call blow-up solutions \KMa{exhibiting instability under perturbation of initial points {\em saddle-type blow-up solutions} in the present paper, respecting the structure of equilibria at infinity.}
This fuzzy nature is difficult to characterize clearly in general, while such behavior can be partially observed in several practical problems as mentioned in Section \ref{section-remark-intro}.
\par
\bigskip

\JP{The main aim of the present paper is to reveal \KMa{a global} nature of \KMa{saddle-type} blow-up solutions through mathematically rigorous blow-up characterizations with the dynamical systems computational machineries mentioned above\KMa{, both qualitatively and quantitatively}. As any computational method inevitably suffers from numerical errors, due both to rounding and discretizing, one must question the validity of its output. This is especially through when solutions are sensitive to initial conditions, as it is the case for instance for dynamical systems possessing blow-up solutions or exhibiting chaos. In order to address the fundamental issue of reliability of computations, the recent field of computer-assisted proofs in nonlinear analysis emerged at the intersection of scientific computing, functional analysis, approximation theory, numerical analysis and topology. In essence, a computer-assisted proof is the process by which the hypotheses of a theorem are verified rigorously with the help of the computer. In the context of dynamical systems, early pioneering works include the proof of the universality of the Feigenbaum constant \cite{feigenbaum} and the proof of existence of the strange attractor in the Lorenz system \cite{Tucker2002}. We refer the interested reader to the survey papers \cite{KOCH_ComputeAssisted, NAKAO_VerifiedPDE, TUCKER_ValidatedIntroduction, VANDENBERG_Dynamics,GOMEZ_PDESurvey}, as well as the recent book \cite{MR3971222}. Computer-assisted proofs are one way to both characterize and visualize mathematical objects in a mathematically rigorous way. Keeping the success of computer-assisted proofs for various applications to dynamical systems (e.g. \cite{CLM2018, DALP2015, MT2020_1, MT2020_2, TMSTMO2017}) in mind, we believe that studying blow-up solutions with computer-assisted proofs provides rich insights into asymptotic behavior of solutions to differential equations as well as new research directions of global dynamics and finite-time singularities.}

\par
\KMa{
To validate saddle-type blow-up solutions, we combine the machinery applied in preceding works, compactifications and time-scale desingularizations, with {\em the parameterization method} (e.g. see \cite{MR1976079,MR1976080,MR2177465}).
The latter notion is now understood as one of universal machineries in dynamical systems, which aims at characterizing and constructing invariant manifolds, including local (un)stable manifolds of invariant sets such as equilibria and periodic orbits.
Moreover, the parameterization method with rigorous ODE integrations has a great compatibility with computer-assisted proofs to capture global nature of invariant manifolds in dynamical systems with their explicit enclosures. 
In particular, globally extended saddle-type blow-up solutions and the corresponding curves of blow-up times can be validated as easily as preceding works (\cite{MT2020_1, MT2020_2, TMSTMO2017}).
\par
We shall also unravel non-trivial and global nature of saddle-type blow-up solutions with the applicability of our proposed methodology through several examples. 
}
The main features of blow-up solutions we shall extract in the present paper are summarized as follows, \KMa{which are not observed in preceding works or theoretical characterizations of blow-ups:}
\begin{itemize}
\item The blow-up time $t_{\max}$ is described by a locally real-analytic function of \KMa{initial points} (Section \ref{sec:tmax-smoothness}).
\item Local foliation structure in level sets of blow-up times which is {\em independent of dynamics at infinity} is observed (Section \ref{sec:distribution_tmax}).
\item Chain of connecting orbits including those corresponding to saddle-type blow-up solutions can separate initial points into several regions possessing significantly different properties, where solutions through these points either exist global-in-time or blow up in finite time, {\em no matter how large \KMa{the magnitude of} initial points is} (Section \ref{sec:dep_blowup_time}).
\item The above chain of connecting orbits induces discontinuity of blow-up times (Section \ref{sec:dep_blowup_time}).
\end{itemize}
The first feature is one of the biggest benefits of the parameterized method in blow-up studies. 
In preceding works, no explicit expression of local stable manifolds is \KMa{obtained}, yielding at most upper and lower bounds of $t_{\max}$ (e.g., \cite{TMSTMO2017}).
In the present methodology, the explicit expressions of local stable manifolds as the graphs of locally analytic functions can be applied and hence, combined with formulae of $t_{\max}$ by means of integrals through trajectories, we obtain the explicit formulae of $t_{\max}$ as functions of initial points.
\par
Through computer-assisted proofs, we obtain explicit distributions of local stable manifolds with their visualizations.
We then see an interesting relationship between asymptotic behavior of blow-up solutions and the corresponding $t_{\max}$.
As the second feature, we see that the asymptotic dynamics near blow-up do {\em not} essentially contribute to determine blow-up times.
In other words, only the magnitude of solutions can determine $t_{\max}$.
The remaining features are also important and completely different from blow-up solutions possessing  persistence of structure under perturbations of initial points.
We see that, in the presence of saddle-type blow-up solutions, there is {\em no} relationship between the magnitude of initial points and blow-up behavior of solutions through these points.
All these features rely on computer-assisted proofs, implying that all results are mathematically rigorous and the methodology towards these results are available to a large class of \KMa{ODEs} {\em without any knowledge of blow-up behavior}.

\par
\bigskip

The rest of the present paper is organized as follows.
In Section \ref{sec:pre_blow_up}, we review a methodology for characterizing blow-up solutions from the viewpoint of dynamical systems, which is based on compactifications and time-scale desingularizations studied in e.g. \cite{Mat2018}.
Three types of compactifications are shown there: {\em directional}, {\em Poincar\'{e}-type} and {\em parabolic-type} ones.
\KMa{
The concrete process for characterizing blow-up solutions is explained for each compactification for readers' accessibility, while the fundamental idea is identical.
Both advantages and disadvantages of each compactification depending the situation are finally mentioned.
}
In Section \ref{sec:parameterization_method}, the parameterization method for calculating invariant manifolds is summarized.
In the present paper, we restrict our attention to stable manifolds of equilibria.
Under an essential assumption called the non-resonance condition of eigenvalues, local stable manifolds can be characterized as zeros of a countable family of nonlinear equations on Banach spaces.
Combining with the method of {\em radii polynomials}, which is one of standard functional-analytic and algebraic machineries for finding zeros of (infinite-dimensional) nonlinear maps, computer-assisted proofs of the existence and characterization of local stable manifolds are provided.
Note that the non-resonance condition yields that validated stable manifolds can be given as locally {\em real-analytic} functions.
In Section \ref{sec:methodology}, we provide a methodology of computer-assisted proofs of the existence of blow-up solutions. 
\KMa{Because} the detailed implementations such as the choice of compactifications and time-scale transformations is \KMa{problem-dependent}, only the basic idea for validating blow-up solutions are presented therein.
We also show that the present methodology enables us to provide an exact and explicit formula of the maximal existence time, equivalently the blow-up time, of solutions as a \KMa{locally smooth or real-analytic function of initial points}, provided all our implementations work successfully.
The present characterization of the \KMa{blow-up} time provides us with a quantitative feature of blow-up solutions such as distributions of blow-up times depending on initial \KMa{points} which are not provided in preceding works \cite{MT2020_1, MT2020_2, TMSTMO2017}.
\KMa{As we shall see, the combination of compactifications with the parameterization method provide a universal concept of blow-up validations and characterizations both qualitatively and quantitatively, no matter how stable equilibria on the horizon (for desingularized vector fields) are.}
\par
The applicability of the present methodology \KMa{and global nature of saddle-type blow-up solutions} are shown in successive sections.
In Section \ref{sec:demo1}, a two-dimensional ODE possessing saddle-type blow-up solutions is considered.
A locally defined (i.e. directional) compactification is applied, and a saddle-type blow-up \KMa{solution, as well as the blow-up time as a function of initial points,} is validated to check the applicability of our methodology to locally distributed blow-up solutions.
In particular, the blow-up profile as well as its blow-up time as a function of initial points \KMa{is} successfully validated\KMa{, extended} and visualized.
In Section \ref{sec:demo2}, we consider a three-dimensional system.
The Poincar\'{e}-type compactification is applied, and one- and two-dimensional stable manifolds of saddle equilibria \KMa{on the horizon} are validated.
The aim is to show the applicability of our methodology to saddle-type blow-up solutions distributed on multi-dimensional stable manifolds of unstable invariant sets \KMa{on the horizon}.
Furthermore, distribution of blow-up times as \KMa{functions} of initial points on two-dimensional stable manifolds are validated, which shows a relationship of blow-up times \KMa{to} the structure of stable manifolds \KMa{around the horizon}.
\KMa{Finally, global extension of local stable manifolds is demonstrated to visualize the distribution of blow-up nature.}
In Section \ref{sec:demo3}, a two-dimensional ODE which is quasi-homogeneous in an asymptotic sense is considered.
The system possesses both stable and unstable equilibria \KMa{on the horizon}.
The parabolic-type compactification is applied and a saddle-type blow-up solution is firstly validated, while validations of blow-up solutions asymptotic to stable equilibria on the horizon are already demonstrated in a preceding work \cite{MT2020_1}.
The main aim of this section is to study \KMa{global} nature of solution families \KMa{near} saddle-type blow-up solutions.
We see that saddle-type blow-up solutions can play the role of the {\em separatrix} decomposing initial points into collections of blow-up solutions and \KMa{global-in-time} solutions.
\KMa{
In other words, saddle-type blow-up solutions can divide initial points into those with globally bounded nature and blow-up nature, {\em no matter how large magnitudes of initial points are}.
This separation cannot be seen in blow-up solutions induced by solutions asymptotic to stable equilibria on the horizon for desingularized vector fields.
Moreover, it is also seen that {\em blow-up times can behave in a singular manner across the saddle-type blow-up solutions}.
Remark that such a singular nature has not been provided only by the local theory, because the {\em global} dynamical information requires to unravel it, while many theoretical characterizations of solution structures are stated only in the local sense. 
We emphasize that computer-assisted proofs enable us to clarify the global nature, even in dynamically singular one, with appropriately chosen machineries.}
All the codes for generating results with computer-assisted proofs in Sections \ref{sec:demo1}, \ref{sec:demo2} and \ref{sec:demo3} are available at \cite{code}.

\section{Preliminary 1: Characterization of blow-up solutions}

\label{sec:pre_blow_up}

In this section, we briefly review a characterization of blow-up solutions for autonomous, finite dimensional \KMa{systems of} ODEs from the viewpoint of dynamical systems.
\KMa{In particular, we pay attention to several concrete cases which are applied in examples later, while details of the present methodology are already provided in \cite{Mat2018, MT2020_2}.
} 
\par
\bigskip
Consider the initial value problem of an autonomous system of ODEs
\begin{equation}
\label{ODE-original}
y' = \frac{dy(t)}{dt}=f (y(t)),\quad y(0)=y_0,
\end{equation}
where $t\in[0,T)$ with $0<T\le\infty$, $f:\mathbb{R}^n\to\mathbb{R}^n$ is a $C^1$ function and $y_0\in\mathbb{R}^n$.

%
%
\subsection{Asymptotically quasi-homogeneous vector fields}
\label{sec:QH}

First of all, we review a class of vector fields in our present discussions.
\begin{dfn}[Asymptotically quasi-homogeneous vector fields, cf. \cite{D1993, Mat2018}]\rm
Let $\KMa{f_0}: \mathbb{R}^n \to \mathbb{R}$ be a smooth \KMa{(i.e. $C^r$ with $r\geq 1$)} function.
Let $\alpha_1,\ldots, \alpha_n, k \geq 1$ be natural numbers.
We say that $\KMa{f_0}$ is a {\em quasi-homogeneous function of type $\KMa{\alpha = } (\alpha_1,\ldots, \alpha_n)$ and order $k$} if
\begin{equation*}
\KMa{f_0}(s^{\alpha_1}x_1,\ldots, s^{\alpha_n}x_n) = s^k \KMa{f_0}(x_1,\ldots, x_n),\quad \forall x\in \mathbb{R}^n,\quad s\in \mathbb{R}.
\end{equation*}
\par
Next, let $X = \sum_{j=1}^n f_j(x)\frac{\partial }{\partial x_j}$ be a smooth vector field on $\mathbb{R}^n$.
We say that $X$, or simply $f = (f_1,\ldots, f_n)$ is a {\em quasi-homogeneous vector field of type $\KMa{\alpha = }(\alpha_1,\ldots, \alpha_n)$ and order $k+1$} if each component $f_j$ is a quasi-homogeneous function of type $\KMa{\alpha}$ and order $k + \alpha_j$.
\par
Finally, we say that $X = \sum_{j=1}^n f_j(x)\frac{\partial }{\partial x_j}$, or simply $f$ is an {\em asymptotically quasi-homogeneous vector field of type $\KMa{\alpha =}(\alpha_1,\ldots, \alpha_n)$ and order $k+1$ at infinity} if \KMa{there is a quasi-homogeneous vector field  $f_{\alpha,k} = (f_{j; \alpha,k})_{j=1}^n$ of type $\KMa{\alpha}$ and order $k+1$ such that}
\begin{equation*}
\lim_{s\to +\infty} s^{-(k+\alpha_j)}\left\{ f_j(s^{\alpha_1}x_1, \ldots, s^{\alpha_n}x_n) - s^{k+\alpha_j} \KMa{f_{j;\alpha,k}}(x_1, \ldots, x_n) \right\}
 = 0
 \end{equation*}
holds uniformly for $(x_1,\ldots, x_n)\in S^{n-1} \equiv \{x = (x_1,\ldots, x_n) \in \mathbb{R}^n \mid \sum_{i=1}^n x_i^2 = 1\}$.
\end{dfn}
Throughout successive sections, consider the (autonomous) vector field (\ref{ODE-original}),
where $f: \mathbb{R}^n \to \mathbb{R}^n$ is an asymptotically quasi-homogeneous smooth vector field of type $\alpha = (\alpha_1,\ldots, \alpha_n)$ and order $k+1$ at infinity.

%
%

\subsection{\KMa{Compactifications, Dynamics at Infinity and Blow-Up Criteria}}
\label{sec:compactification}
\KMa{
Here we summarize the basic strategy used throughout the successive sections.
The main idea is application of {\em compactifications}; the embedding of the original phase space into compact manifolds or their tangent spaces with boundaries.
The boundaries then correspond to the infinity.
There are mainly two different types of compactifications: the locally defined one and globally defined one.
The local one is simple and applied to many preceding works involving dynamics at infinity, while the global one enables us to treat dynamics including infinity in one chart.
After introducing compactifications, we derive vector fields which we mainly concern, and provide the characterization of blow-up solutions by means of dynamical systems.
The concrete process for the characterization of blow-up solutions is provided for each compactification which we introduce.
}

\subsubsection{\KMa{A basic strategy}}
\label{sec:basic}

\KMa{
The basic strategy for characterizing blow-up solutions is summarized as follows, which is independent of the choice of compactifications introduced below.
\begin{enumerate}
\item For given vector field $f$ provided by (\ref{ODE-original}), determine its type $\alpha$ and order $k+1$.
\item Choose an appropriate compactification of the same type $\alpha$ (mentioned below) as $f$.
\item Transform (\ref{ODE-original}) into the corresponding one through the compactification.
\item Introduce a time-scale transformation to desingularize the vector field determined by the order $k+1$ of $f$. 
The resulting vector field shall be called {\em the desingularized vector field}.
{\em Dynamics at infinity} then makes sense through the desingularized vector field.
\item Validate hyperbolic invariant sets on the special geometric object corresponding to infinity, {\em the horizon}, and their local stable manifolds for the desingularized vector field.
\end{enumerate}
Once invariant sets, such as equilibria and periodic orbits, {\em on the horizon} with their hyperbolicity are validated, their local stable manifolds characterize the collection of blow-up solutions of (\ref{ODE-original}) near blow-ups, which is the essence of our proposing methodology.
In the successive parts, the blow-up characterization is shown for each compactification.
\par
An important point here is a suitable choice of \lq\lq appropriate" compactifications so that our blow-up problem can be reduced to standard issues in dynamical systems.
Below are examples of such suitable compactifications, which possess both advantages and disadvantages and hence these compactifications have to be used according to our needs.
Several characteristics of compactifications are summarized in Section \ref{rem-choice-compactification}.
}

\subsubsection{Directional compactifications}
\label{sec:dir}
\KMa{First a locally defined compactification is introduced, which shall be called a {\em directional compactification}.}
\begin{dfn}[Directional compactification, cf. \cite{DLA2006, Mat2018}]\rm
A {\em directional compactification}\footnote
{
Although $T_d$ is not a compactification in the topological sense, we shall use this terminology for $T_d$ from its geometric interpretation shown below.
} of type $\alpha = (\alpha_1, \ldots, \alpha_n)$ is defined as 
\begin{align}
\notag
&y = (y_1,\ldots, y_n) \mapsto 
T_d(y) = (s,\hat x) \equiv (s, \hat x_1,\ldots, \hat x_{i_0 -1}, \hat x_{i_0+1}, \ldots, \hat x_n),\\
\label{dir-cpt}
&y_i := \frac{\hat x_i}{s^{\alpha_i}}\quad (i\not = i_0),\quad y_{i_0} := \pm \frac{1}{s^{\alpha_{i_0}}}
\end{align}
with given direction $i_0\in \{1,\ldots, n\}$ and the signature $\pm$.
This compactification is bijective in $\mathbb{R}^n\cap \{\pm y_{i_0}  > 0\}$, in which sense directional compactifications are {\em local} ones.
In particular, this compactification is available when we are interested in trajectories of (\ref{ODE-original}) such that the $i_0$-th component has the identical sign during time evolution.
The image of $T_d$ is
\begin{equation}
\label{D-directional}
\mathcal{D} = \{(s, \hat x_1,\ldots, \hat x_{i_0-1}, \hat x_{i_0+1}, \ldots, \hat x_n) \mid s >0,\quad \hat x_i \in \mathbb{R}\quad (i\not = i_0)\}.
\end{equation}
The set $\mathcal{E} = \{s=0\}$ corresponds to the infinity in the original coordinate, which shall be called {\em the horizon}.
\end{dfn}
\KMa{
Other geometric interpretations are mentioned in Section \ref{rem-choice-compactification}.
For simplicity, fix $i_0=1$ in (\ref{dir-cpt}) in the following arguments.
}
Next transform (\ref{ODE-original}) via (\ref{dir-cpt}), which is straightforward:
\begin{align*}
\frac{ds}{dt} &= -\frac{1}{\alpha_1}s^{-k+1}\hat f_1(s, \hat x_2, \ldots, \hat x_n),\\
\frac{d\hat x_i}{dt} &= s^{-k} \left\{ \hat f_i(s, \hat x_2, \ldots, \hat x_n) -\frac{\alpha_i}{\alpha_1}x_i \hat f_1(s, \hat x_2, \ldots, \hat x_n) \right\}\quad (i=2,\ldots, n),
\end{align*}
where
\begin{equation}
\label{f-tilde-directional}
\hat f_i(s, \hat x_2, \ldots, \hat x_n) \bydef s^{k+\alpha_i} f_i(s^{-\alpha_1}, s^{-\alpha_2}\hat x_2, \ldots, s^{-\alpha_n}\hat x_n),\quad i=1,\ldots, n.
\end{equation}
The resulting vector field is still singular near \KMa{the horizon}, but it turns out that the order of divergence of vector field as $s\to +0$ is $O(s^{-k})$, and hence the following time-scale transformation is available.

\begin{dfn}[Time-variable desingularization: \KMa{the directional} version]\rm 
Define the new time variable $\tau_d$ by
\begin{equation}
\label{time-desing-directional}
d\tau_d = s(t)^{-k} dt
\end{equation}
equivalently,
\begin{equation}
\label{time-desing-directional-integral}
t =  t_0 + \int_{\tau_0}^\tau s(\tau_d)^k d\tau_d,
\end{equation}
where $\tau_0$ and $t_0$ denote the correspondence of initial times, and $s(\tau_d)$ is the solution trajectory $s(t)$ under the parameter $\tau_d$.
We shall call (\ref{time-desing-directional}) {\em the time-variable desingularization (of order $k+1$)}.
\end{dfn}

The vector field $g=g_d$ in $\tau_d$-time-scale is 
\begin{equation}
\label{ODE-desing-directional}
\begin{pmatrix}
\frac{ds}{d\tau_d} \\ \frac{dx_2}{d\tau_d} \\ \vdots \\ \frac{dx_{n}}{d\tau_d}
\end{pmatrix}
=
g_d(s,\hat x_2,\ldots, \hat x_n)
\bydef 
\begin{pmatrix}
-s & 0 & \cdots & 0 \\
0 & 1 & \cdots & 0\\
\vdots & \vdots & \ddots & \vdots \\
0 & 0 & \cdots & 1
\end{pmatrix}
B
\begin{pmatrix}
\hat f_1 \\ \hat f_2 \\ \vdots \\ \hat f_n
\end{pmatrix}
\end{equation}
where $B$ is the inverse\footnote{
The existence of $B$ immediately follows by cyclic permutations and the fact that $\alpha_n > 0$.
}
of the matrix
\begin{equation*}
\begin{pmatrix}
\alpha_1 & 0  & \cdots & 0& 0\\
\alpha_2 \hat x_2 & 1 & \cdots & 0 & 0 \\
\vdots & \vdots & \ddots & \vdots & \vdots \\
\alpha_{n-1} \hat x_{n-1} & 0 & \cdots & 1 & 0\\
\alpha_n \hat x_n & 0  & \cdots & 0 & 1
\end{pmatrix}.
\end{equation*}
The componentwise expression is
\begin{align*}
\frac{ds}{d\tau_d} &= g_{d,1}(s,\hat x_2,\ldots, \hat x_n) \equiv -\frac{1}{\alpha_1}s^{-k+1}\hat f_1(s, \hat x_2, \ldots, \hat x_n),\\
\frac{d\hat x_i}{d\tau_d} &= g_{d,i}(s,\hat x_2,\ldots, \hat x_n) \equiv \hat f_i(s, \hat x_2, \ldots, \hat x_n) -\frac{\alpha_i}{\alpha_1}x_i \hat f_1(s, \hat x_2, \ldots, \hat x_n)\quad (i=2,\ldots, n).
\end{align*}
This vector field is as smooth as $f$ {\em including $s=0$} and hence {\em dynamics at infinity} makes sense through dynamics generated by (\ref{ODE-desing-directional}) around the horizon $\mathcal{E} = \{s=0\}$.
Once the desingularized vector field (\ref{ODE-desing-directional}) is provided, blow-up solutions can be characterized as follows.

\begin{thm}[Stationary blow-up: the directional version, \cite{Mat2018}]
\label{thm:blowup-dir}
Assume that \KMa{the desingularized vector field (\ref{ODE-desing-directional}) associated with (\ref{ODE-original}) has an equilibrium on the horizon ${\bf x}_\ast = (0,x_\ast)\in \mathcal{E}$}.
Also suppose that ${\bf x}_\ast$ is hyperbolic with $n_s > 0$ (resp. $n_u = n-n_s$) eigenvalues of the Jacobian matrix $Dg_d({\bf x}_\ast)$ with negative (resp. positive) real parts.
If there is a solution $y(t)$ of (\ref{ODE-original}) with a bounded initial \KMa{point} $y(0)$ whose image ${\bf x} = T_d(y)$ is on the local stable manifold $W^s_{\rm loc}({\bf x}_\ast; g_d)$, then $t_{\max} < \infty$ holds; namely, $y(t)$ is a blow-up solution.
Moreover,
\begin{equation*}
s(t)^{-1} \sim c(t_{\max} - t)^{-1/k}\quad \text{ as }\quad t\to t_{\max}
\end{equation*}
where $c>0$ is a constant.
Finally, if the $i$-th component of ${\bf x}_\ast$ ($i\in \{2,\ldots, n\}$) is not zero, then we also have
\begin{equation*}
y_i(t) \sim c_i(t_{\max} - t)^{-\alpha_i /k}\quad \text{ as }\quad t\to t_{\max},
\end{equation*}
where $c_i$ is a constant with the same sign as $y_i(t)$ as $t\to t_{\max}$.
\end{thm}

\begin{rem}
Note that there are other locally defined compactifications, such as a {\em quasi-polar} one known as {\em Poincar\'{e}-Lyapunov disk} (e.g. \cite{DH1999, DLA2006, Mat2018}).
\end{rem}

%
%
\subsubsection{Poincar\'{e}-type compactifications}

\KMa{
The remaining compactifications we introduce here are {\em global} ones in the sense that they are embeddings of the whole phase space $\mathbb{R}^n$ into compact manifolds with boundaries.
A suitable class of global type compactifications for characterizing dynamics at infinity for asymptotically quasi-homogeneous vector fields is discuss in \cite{MT2020_1}, where such a class of compactifications are called {\em admissible global compactifications}.
Among such compactifications, two representative compactifications are reviewed.
}
\par
\KMa
{
As a general setting, for given $n$-tuple of natural numbers $\alpha = (\alpha_1,\ldots, \alpha_n)$, 
let $\beta_1,\ldots, \beta_n$ be natural numbers\footnote{
The simplest choice of the natural number $c$ is the least common multiple of $\alpha_1,\ldots, \alpha_n$.
Once we choose such $c$, we can determine the $n$-tuples of natural numbers $\beta_1,\ldots, \beta_n$ uniquely.
The choice of natural numbers in (\ref{LCM}) is essential to desingularize vector fields at infinity, as shown below.
} such that 
\begin{equation}
\label{LCM}
\alpha_1 \beta_1 = \alpha_2 \beta_2 = \cdots = \alpha_n \beta_n \equiv c \in \mathbb{N}.
\end{equation}
Then define a functional $p(y)$ as
\begin{equation}
\label{func-p}
p(y) \bydef \left( y_1^{2\beta_1} +  y_2^{2\beta_2} + \cdots +  y_n^{2\beta_n} \right)^{1/2c}.
\end{equation}
}
The prototype of admissible global compactifications is the {\em Poincar\'{e}-type}.
\KMa{
\begin{dfn}[Poincar\'{e}-type compactification. cf. \cite{Mat2018}]\rm
The {\em Poincar\'{e}-type compactification (of type $\alpha = (\alpha_1, \ldots, \alpha_n)$)} is defined as the mapping $T_{qP}: \mathbb{R}^n \to \mathbb{R}^n$ as
\begin{equation}
\label{cpt-global}
\quad T_{qP}(y) = x,\quad x_i \bydef \frac{y_i}{\kappa(y)^{\alpha_i}},
\end{equation}
with $\kappa(y) = \kappa_{qP}(y) \bydef (1+p(y)^{2c})^{1/2c}$.
The map $T_{qP}$ maps $\mathbb{R}^n$ onto 
\begin{equation}
\label{D-global}
\mathcal{D} = \{x\in \mathbb{R}^n \mid p(x) < 1\}.
\end{equation}
The boundary $\mathcal{E} \equiv \partial \mathcal{D} = \{x\in \mathbb{R}^n \mid p(x) = 1\}$ is called the {\em horizon}.
\end{dfn}
}
\KMa{Its geometric interpretation is mentioned in Section \ref{rem-choice-compactification}.}
Note from \cite{Mat2018} that $\kappa = \kappa_{qP}(y)$ has an equivalent expression by means of $x$:
\begin{equation*}
\kappa = \kappa_{qP}(T_{qP}^{-1}(x)) = \left(1 - \sum_{j= 1}^n x_j^{2\beta_j}\right)^{-1/2c}.
\end{equation*}
\KMa{
Similar to the directional ones, for given vector field $f$ of the same type $\alpha$, we apply the Poincar\'{e}-type compactification of the same type $\alpha$.}
Then we have
\begin{align*}
\frac{dx_i}{dt} = 
 \tilde f_i(x) - \alpha_i x_i \sum_{j=1}^n (\nabla \kappa)_j \kappa^{\alpha_j - 1}\tilde f_j(x),
\end{align*}
where
\begin{equation}
\label{f-tilde-Poincare}
\tilde f_j(x_1,\ldots, x_n) := \kappa^{-(k+\alpha_j)} f_j(\kappa^{\alpha_1}x_1, \ldots, \kappa^{\alpha_n}x_n),\quad j=1,\ldots, n,
\end{equation}
which is the alternate object of $\hat f_j$'s in (\ref{f-tilde-directional}), $\kappa = \kappa_{qP}(y)$, and
\begin{equation}
\label{nabla_kappa}
(\nabla \kappa)_j \equiv (\nabla_y \kappa(y))_j  = \frac{\beta_j y_j^{2\beta_j-1}}{c\kappa^{2c-1} } = \frac{\beta_j \kappa^{2c-\alpha_j} x_j^{2\beta_j-1}}{c\kappa^{2c-1} } = \frac{x_j^{2\beta_j-1}}{\alpha_j \kappa^{\alpha_j-1} }.
\end{equation}

It is shown in \cite{Mat2018} that the above vector field is still singular on the horizon $\mathcal{E}$, but the order of divergence is $O(\kappa^k)$ as \KMa{$p(y)\to +\infty$, equivalently $p(x)\to 1$}, which is independent of components.
Therefore a common time-scale transformation can be introduced.

\begin{dfn}[Time-variable desingularization: the Poincar\'{e}-type version]\rm 
Define the new time variable $\tau_d$ by
\begin{equation}
\label{time-desing-Poin}
d\tau_{qP} = \kappa_{qP}(y(t))^k dt
\end{equation}
equivalently,
\begin{equation}
\label{time-desing-Poin-integral}
t = t_0 + \int_{\tau_0}^\tau \kappa_{qP}(y(\tau_{qP}))^{-k} d\tau_{qP},
\end{equation}
where $\tau_0$ and $t_0$ denote the corresponding initial times, and $y(\tau_{qP})$ is the solution  $y(t)$ under the time-scale $\tau_{qP}$.
We shall call (\ref{time-desing-Poin}) {\em the time-variable desingularization (of order $k+1$)}.
\end{dfn}

Using this time-scale, we obtain 
\begin{align}
\label{vectorfield-cw-Poincare}
\dot x_i = \frac{dx_i}{d\tau} &= g_{qP, i}(x) \bydef
 \tilde f_i(x) - \alpha_i x_i \sum_{j=1}^n \frac{x_j^{2\beta_j-1}}{\alpha_j} \tilde f_j(x).
\end{align}
This vector field is continuous including the horizon $\mathcal{E}$, and hence {\em dynamics at infinity} makes sense through (\ref{vectorfield-cw-Poincare}).
It should be noted, however, that the desingularized vector field (\ref{vectorfield-cw-Poincare}) is not always smooth on $\mathcal{E}$. 
Details are mentioned in Section \ref{sec:choice-compactification}.
Similar to Theorem \ref{thm:blowup-dir}, blow-up characterization is provided as follows.

\begin{thm}[Stationary blow-up: the Poincar\'{e}-type version, \cite{Mat2018}]
\label{thm:blowup-Poincare}
Consider the desingularized vector field $g_{qP}$ associated with (\ref{ODE-original}) given by (\ref{vectorfield-cw-Poincare}).
Assume that \KMa{$g_{qP}$ is $C^1$ in a neighborhood of the horizon $\mathcal{E}$, and that $g_{qP}$ has an equilibrium on the horizon ${\bf x}_\ast\in \mathcal{E}$}.
Suppose that ${\bf x}_\ast$ is hyperbolic with $n_s > 0$ (resp. $n_u = n-n_s$) eigenvalues of $Dg_{qP}({\bf x}_\ast)$ with negative (resp. positive) real parts.
If there is a solution $y(t)$ of (\ref{ODE-original}) with a bounded initial \KMa{point} $y(0)$ whose image $x = T_{qP}(y)$ is on the local stable manifold $W^s_{\rm loc}({\bf x}_\ast; g_{qP})$, then $t_{\max} < \infty$ holds; namely, $y(t)$ is a blow-up solution.
Moreover,
\begin{equation*}
\KMa{p(y(t))} \sim c(t_{\max} - t)^{-1/k}\quad \text{ as }\quad t\to t_{\max},
\end{equation*}
where $c>0$ is a constant.
Finally, if the $j$-th component ${\bf x}_\ast$ is not zero, then we also have
\begin{equation*}
y_i(t) \sim c_i(t_{\max} - t)^{-\alpha_i /k}\quad \text{ as }\quad t\to t_{\max},
\end{equation*}
where $c_i$ is a constant with the same sign as $y_i(t)$ as $t\to t_{\max}$.
\end{thm}

\subsubsection{Parabolic-type compactifications}
\KMa{
An alternative admissible global compactification, which shall be called the {\em parabolic-type} compactification, is introduced here.
Compactifications of the present type were originally introduced in \cite{G2004} and generalized in \cite{MT2020_2}.
}
\par
Similar to the Poincar\'{e}-type compactifications, define a set $\mathcal{D}\subset \mathbb{R}^n$ by (\ref{D-global}).
For any $x\in \mathcal{D}$, \KMa{correspond $y\in \mathbb{R}^n$ to $x\in \mathcal{D}$} by 
\begin{equation*}
S(x) = y,\quad y_j = \frac{x_j}{(1- p(x)^{2c})^{\alpha_j}},\quad j=1,\ldots, n.
\end{equation*}
Let $\tilde \kappa_\alpha(x) \bydef (1-p(x)^{2c})^{-1}$, which satisfies $\tilde \kappa_\alpha(x) \geq 1$ for all $x\in \mathcal{D}$.
Moreover, $y\not = 0$ implies $\tilde \kappa_\alpha(x) > 1$.
We also have 
\begin{equation}
\label{py-para}
p(y)^{2c} = \tilde \kappa_\alpha(x)^{2c}p(x)^{2c} = \tilde \kappa_\alpha(x)^{2c}\left(1-\frac{1}{\tilde \kappa_\alpha(x)}\right).
\end{equation}
This equality indicates that $p(y) = p(S(x)) < \tilde \kappa_\alpha(x)$ holds for all $x\in \mathcal{D}$.

\begin{lem}[\cite{MT2020_2}]
\label{lem-zero-compactification}
Let $F(\kappa; R) \bydef \kappa^{2c} - \kappa^{2c-1} - R^{2c}$ for $R\geq 0$.
Then, for any $R\geq 0$, there is a unique $\kappa = q(R)$ satisfying $q(0) = 1$ such that $F(q(R); R) \equiv 0$.
Moreover, $q(R) > 1$ holds for all $R>0$ and $q(R)$ is smooth with respect to $R\geq 0$.
\end{lem}

Now we have $\tilde \kappa_\alpha(x)$ satisfies $F(\tilde \kappa_\alpha(x);p(y)) = 0$.
By the uniqueness of $\kappa(y) = q(R)$ with respect to $R=p(y)$, for any $y\in \mathbb{R}^n\setminus \{{\bf 0}\}$, $\kappa(y) = \kappa_{para}(y) \equiv \kappa(S(x)) \bydef \tilde \kappa_\alpha(x)$ is well-defined.
\KMa{As a consequence, the mapping $S$ admits the inverse $S^{-1} = T \equiv T_{para}$, which yields the following definition.}

\begin{dfn}[\KMa{Parabolic-type} compactification, \cite{MT2020_2}]\rm
\label{dfn-quasi-para}
Let the type $\alpha = (\alpha_1,\ldots, \alpha_n)\in \mathbb{Z}_{>0}^n$ fixed.
Let $\{\beta_i\}_{i=1}^n$ and $c$ be a collection of natural numbers satisfying (\ref{LCM}).
Define $T_{para}:\mathbb{R}^n\to \mathcal{D}$ as
\begin{equation*}
T_{para}(y) \bydef x,\quad x_i = \frac{y_i}{\kappa_{para}(y)^{\alpha_i}},
\end{equation*}
where $\kappa = \kappa_{para}(y) = \tilde \kappa_\alpha(x)$ is the unique zero of $F(\kappa; p(y))=0$ obtained in Lemma \ref{lem-zero-compactification}.
We say the map $T_{para}$ the {\em \KMa{parabolic-type} compactification (of type $\alpha = (\alpha_1,\ldots, \alpha_n)$)}.
The map $T_{para}$ maps $\mathbb{R}^n$ onto $\mathcal{D}$.
The boundary $\mathcal{E} \equiv \partial \mathcal{D} = \{x\in \mathbb{R}^n \mid p(x) = 1\}$ is called the {\em horizon}.
\end{dfn}

Similar to directional and \KMa{the} Poincar\'{e}-type ones, we apply the parabolic\KMa{-type} compactification of the type $\alpha$ which is the same as that of $f$ to transforming (\ref{ODE-original}).
The resulting vector field is 
\begin{equation*}
\frac{dx_i}{dt} =  \tilde f_i(x) - \alpha_i x_i \sum_{j=1}^n (\nabla \kappa)_j \kappa^{\alpha_j - 1}\tilde f_j(x),
\end{equation*}
where $\tilde f = (\tilde f_1,\ldots, \tilde f_n)$ is (\ref{f-tilde-Poincare}) replacing $\kappa$ by $\kappa_{para}$, in which case
\begin{equation*}
(\nabla_y \kappa(y))_j  = \frac{y_j^{2\beta_j-1}}{\alpha_j\kappa(y)^{2c-1} \left(1- \frac{2c-1}{2c} \kappa(y)^{-1}\right)}.
\end{equation*}

\KMa{
Similar to the Poincar\'{e}-type case, all components of the transformed vector field are $O(\kappa^k)$ as $p(y) \to \infty$, equivalently as $x$ approaches to $\mathcal{E}$, and hence the uniform time-scale transformation can be introduced to desingularize the vector field on $\mathcal{E}$.
}
\begin{dfn}[Time-variable desingularization: the parabolic-type version]\rm 
Define the new time variable \KMa{$\tau_{para}$} by
\begin{equation}
\label{time-desing-para}
d\tau_{para} = (1-p(x)^{2c})^{-k}\left\{1-\frac{2c-1}{2c}(1-p(x)^{2c}) \right\}^{-1}dt,
\end{equation}
equivalently,
\begin{equation}
\label{time-desing-para-integral}
t = t_0 + \int_{\tau_0}^\tau \left\{1-\frac{2c-1}{2c}(1-p(x(\tau_{para}))^{2c}) \right\}(1-p(x(\tau_{para}))^{2c})^k d\tau_{para},
\end{equation}
where $\tau_0$ and $t_0$ denote the correspondence of initial times\KMa{.}
We shall call (\ref{time-desing-para}) {\em the time-variable desingularization (of order $k+1$)}.
\end{dfn}
The change of coordinate and the above desingularization yield the following vector field \KMa{$g_{para}$}, which is continuous on $\overline{\mathcal{D}} = \{p(x) \leq 1\}$:
\begin{align}
\label{desing-para}
\dot x_i = g_{para,i}(x) \bydef \left(1-\frac{2c-1}{2c}(1-p(x)^{2c}) \right)\tilde f_i(x) - \alpha_i x_i \sum_{j=1}^n \frac{x_j^{2\beta_j - 1}}{\alpha_j}\tilde f_j(x) ,
\end{align}
\KMa{The desingularized vector field $g_{para}$ has the very similar form to $g_{qP}$. 
On the other hand, the algebraic structure of $\kappa$ is quite different from each other. 
In particular, \KMa{$\kappa = \kappa_{para}$} does not include radicals in $x$, and hence the smoothness of $f$ and the asymptotic quasi-homogeneity guarantee the smoothness of the right-hand side $g_{para}$ of (\ref{desing-para}) including the horizon \KMa{$\mathcal{E}$}.
See \cite{MT2020_1} for details.
This property yields a relaxation of conditions for characterizing blow-ups.
}

\begin{thm}[Stationary blow-up: the parabolic-type version, cf. \cite{Mat2018}, \cite{MT2020_1}]
\label{thm:blowup-parabolic}
Consider the desingularized vector field $g_{para}$ associated with (\ref{ODE-original}) given by (\ref{desing-para}).
Assume that $g_{para}$ has an equilibrium on the horizon ${\bf x}_\ast\in \mathcal{E}$.
Also, suppose that ${\bf x}_\ast$ is hyperbolic with $n_s > 0$ (resp. $n_u = n-n_s$) eigenvalues of $Dg_{para}({\bf x}_\ast)$ with negative (resp. positive) real parts.
If there is a solution $y(t)$ of (\ref{ODE-original}) with a bounded initial \KMa{point} $y(0)$ whose image $x = T_{para}(y)$ is on the local stable manifold $W^s_{\rm loc}({\bf x}_\ast; g_{para})$, then $t_{\max} < \infty$ holds; namely, $y(t)$ is a blow-up solution.
Moreover,
\begin{equation*}
\KMa{p(y(t))} \sim c(t_{\max} - t)^{-1/k}\quad \text{ as }\quad t\to t_{\max}
\end{equation*}
where $c>0$ is a constant.
Finally, if the $j$-th component ${\bf x}_\ast$ is not zero, then we also have
\begin{equation*}
y_i(t) \sim c_i(t_{\max} - t)^{-\alpha_i /k}\quad \text{ as }\quad t\to t_{\max},
\end{equation*}
where $c_i$ is a constant with the same sign as $y_i(t)$ as $t\to t_{\max}$.
\end{thm}
\KMa{
The proof is essentially the same as Theorem \ref{thm:blowup-Poincare}.
Indeed, only the admissible nature (discussed in \cite{MT2020_1}) of $T_{para}$ is used to prove $t_{\max}<\infty$, which is the same as $T_{qP}$.
}

\par
\bigskip
The key point of \KMa{our characterization of blow-ups (Theorems \ref{thm:blowup-dir}, \ref{thm:blowup-Poincare} and \ref{thm:blowup-parabolic})} is that {\em blow-up solutions for (\ref{ODE-original}) are characterized as trajectories on (local) stable manifolds of invariant sets\footnote{
Hyperbolicity ensures not only blow-up behavior of solutions but their asymptotic behavior with the specific form.
Several case studies of blow-up solutions beyond hyperbolicity are shown in \cite{Mat2019}.
} on the horizon $\mathcal{E}$ for desingularized vector fields}.
Computations of blow-up solutions are therefore reduced to \KMa{those of local stable manifolds} of invariant sets, such as (hyperbolic) equilibria, for the associated vector field.
\KMa{Although the above theorems only characterizes the existence and {\em local} dynamical nature of blow-up solutions, combinations of our characterization with numerical computations and computer-assisted proofs provide {\em global} nature of blow-up solutions in the phase space.}


\subsection{Remark on appropriate choice of compactifications}
\label{rem-choice-compactification}

We have introduced three compactifications in this section.
Each compactification has its own set of advantages and disadvantages, which depend on our requirements.
Here we remark the choice of compactifications in case that the original vector field $f$ is polynomial\footnote{
This assumption is not essential but just for simplifications to show advantages and disadvantages of each compactification.
}.
In our examples (Sections \ref{sec:demo1}, \ref{sec:demo2} and \ref{sec:demo3}), all these compactifications are applied. 
It is worth mentioning several features of each compactification towards effective choice and applications of our machineries to practical and advanced problems.

\subsubsection{Geometric interpretations of compactifications}
\KMa{
First the geometric interpretation of each compactification is briefly summarized.
Directional compactifications are not actually compactifications in the topological sense, while these are still called \lq\lq compactifications" because these are inclusively discussed in the context of compactifications for applications.
In fact, images of directional compactifications are interpreted as the tangent space of the Poincar\'{e}'s hemisphere considered in the Poincar\'{e}-type compactifications at points on the horizon, as shown in Figure \ref{fig:compactification}-(a).
\par
Global (Poincar\'{e}-type and parabolic type) compactifications are geometrically simple in the homogeneous case $\alpha = (1,\ldots, 1)$, in which case $p(y) = \|y\|$ and we can choose $\beta = (1,\ldots, 1)$ and $c=1$. 
Therefore $\kappa_{qP}(y) = (1+ \|y\|^2)^{1/2}$, which is a well-known (global, but homogeneous) compactification\footnote{
In many references, this compactification is called the {\em Poincar\'{e} compactification}.
The quasi-homogeneous counterpart is introduced in \cite{Mat2018} where the corresponding mapping $T = T_{qP}$ is called the {\em quasi-Poincar\'{e} compactification}.
}, and the resulting mapping $T_{qR}$ is the embedding of $\mathbb{R}^n$ into the Poincar\'{e} hemisphere
\begin{equation*}
\mathcal{H} = \{(x_1, \ldots, x_n, z)\in \mathbb{R}^{n+1} \mid \|x\|^2 + z^2 = 1,\, z > 0\}.
\end{equation*}
A homogeneous compactification of this kind is shown in Figure \ref{fig:compactification}-(b).
\par
The geometric nature of the parabolic-type compactification \KMa{with} $\alpha = (1,\ldots, 1)$ \KMa{is also understood in a simple way}, in which case $T_{para}$ is defined as
\begin{equation*}
x_i = \frac{2y_i}{1+\sqrt{1+4\|y\|^2}}\quad \Leftrightarrow \quad y_i = \frac{x_i}{1-\|x\|^2},\quad i=1,\ldots, n.
\end{equation*}
See \cite{EG2006, G2004} for the homogeneous case, which is called the {\em parabolic compactification}.
In particular, the parabolic compactification is the embedding of $\mathbb{R}^n$ onto the bounded parabola 
\begin{equation*}
\left\{(x_1,\ldots, x_{n+1})\in \mathbb{R}^{n+1} \mid \sum_{i=1}^n x_i^2 = x_{n+1},\, x_{n+1} < 1\right\}
\end{equation*}
in $\mathbb{R}^{n+1}$ with the focus point $(x_1,\ldots, x_n, x_{n+1}) = (0,\ldots, 0,1)$.
The map $T_{para}$ is actually defined as the composition of this embedding and the projection onto the first $n$ components, which is shown in Figure \ref{fig:compactification}-(c).
\par
\bigskip
Our compactificaitons introduced here are quasi-homogeneous counterparts of the above homogeneous compactifications.
Geometric pictures of quasi-homogeneous Poincar\'{e}-type and parabolic-type compactifications are shown in \cite{Mat2018} and \cite{MT2020_1}, respectively.
}

\begin{rem}
\label{rem-comp-anal}
\KMa{Any compactifications we have introduced are} analytic at any point in \KMa{$\mathcal{D}$} by using the binomial theorem in the standard calculus and the inverse function theorem for analytic mappings (e.g. \cite{D1960}).
\end{rem}

\begin{figure}[h!]\em
\begin{minipage}{0.32\hsize}
\centering
\includegraphics[width=4.5cm]{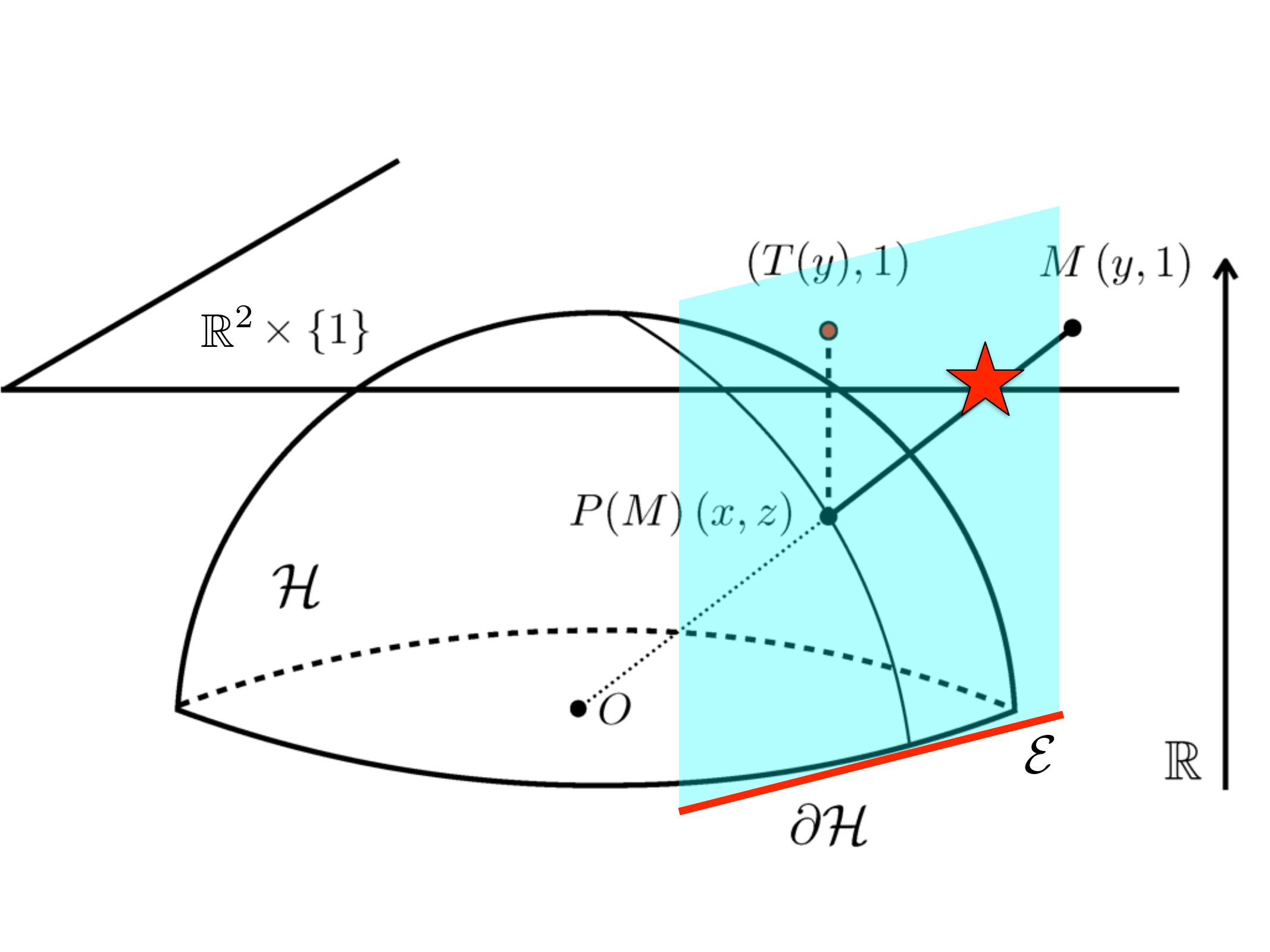}
(a)
\end{minipage}
\begin{minipage}{0.32\hsize}
\centering
\includegraphics[width=4.5cm]{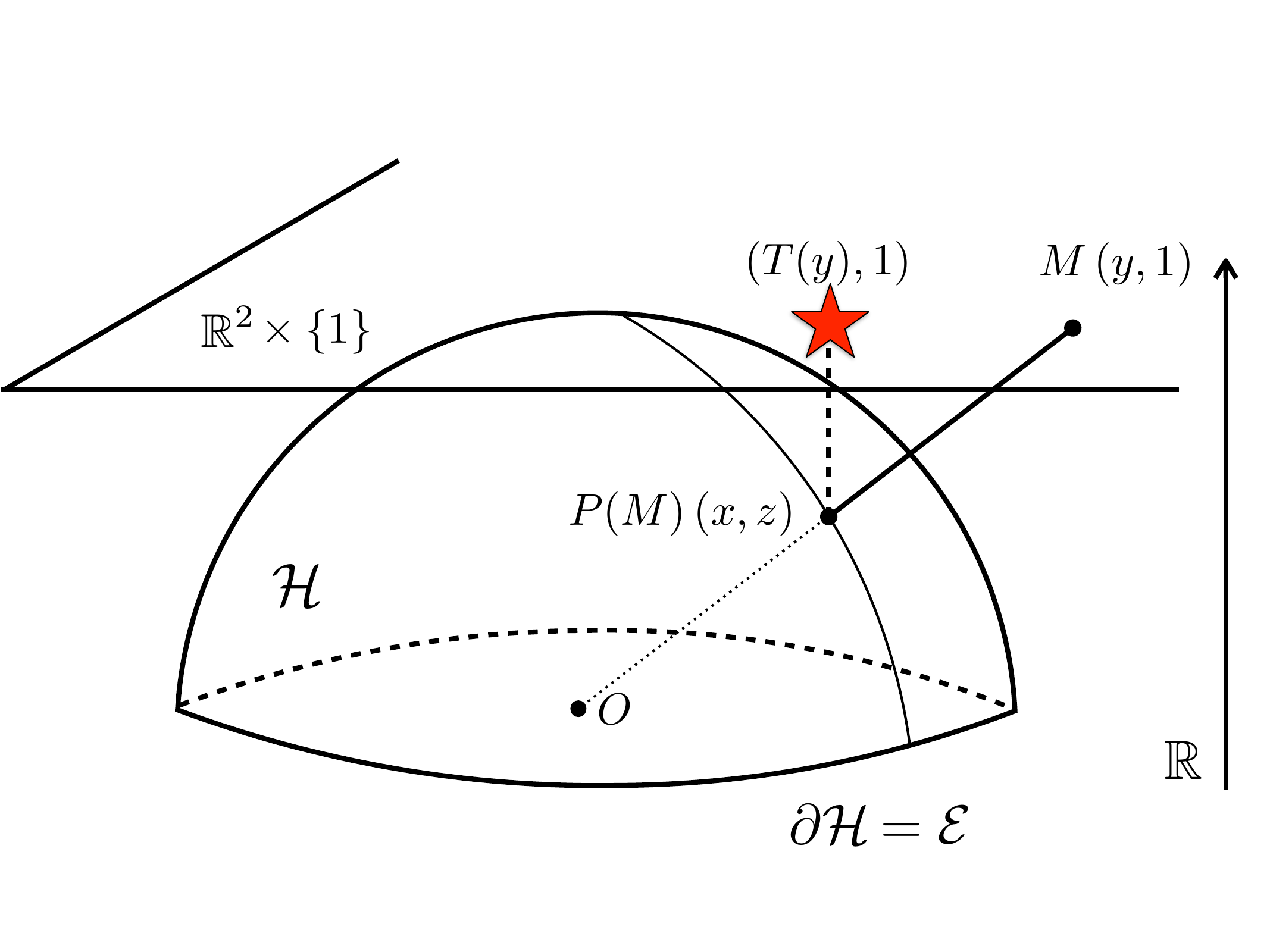}
(b)
\end{minipage}
\begin{minipage}{0.32\hsize}
\centering
\includegraphics[width=4.5cm]{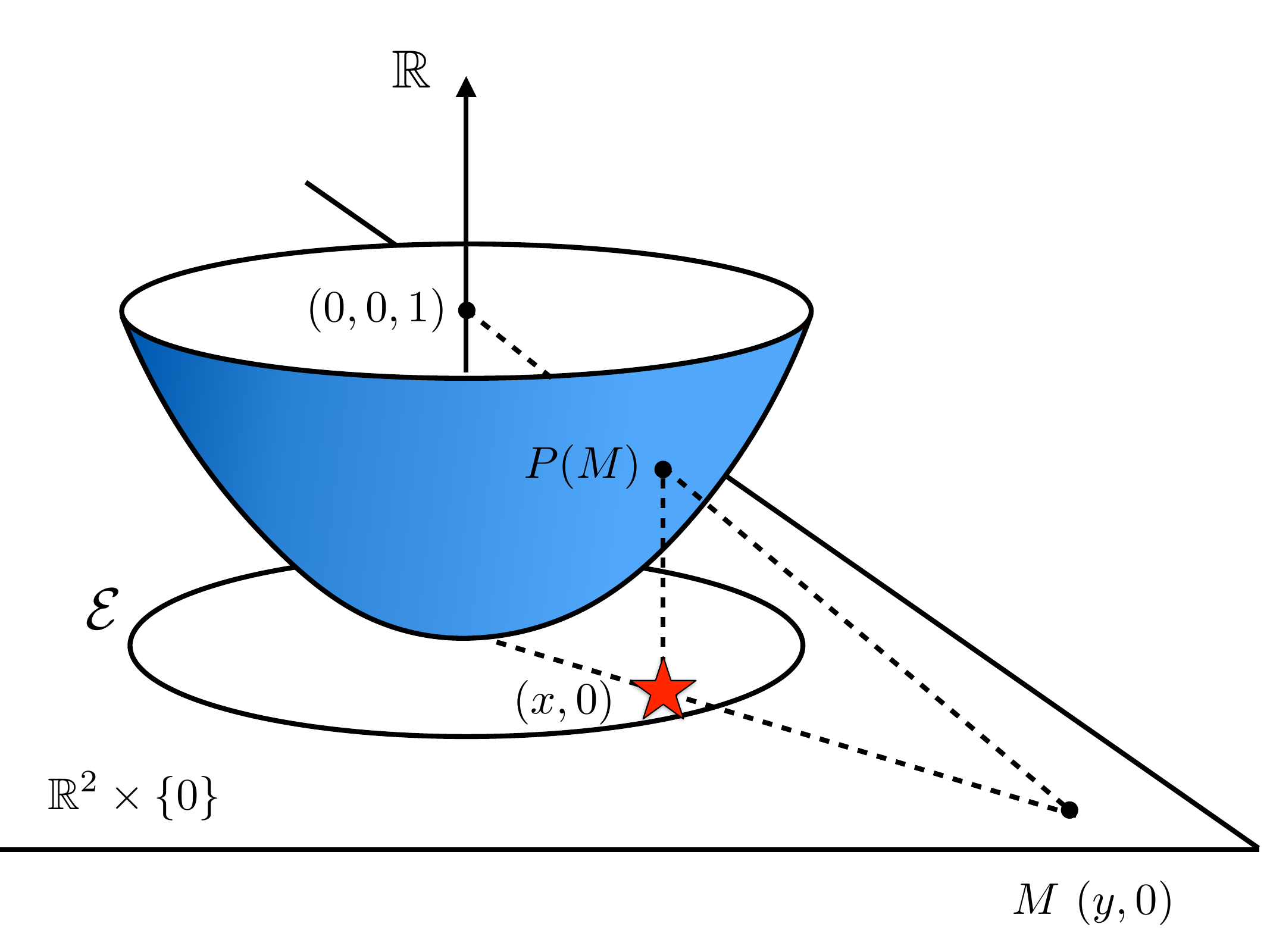}
(c)
\end{minipage}
\caption{\KMa{Schematic illustrations of} homogeneous compactifications of $\mathbb{R}^2$}
\label{fig:compactification}
(a): Directional compactification with type $\alpha = (1,1)$.
The original point $M = (y, 1) \in \mathbb{R}^2\times \{1\}$ is mapped into the point (drawn as the red star) on the upper-half tangent space (colored by skyblue) of a point on $\partial \mathcal{H}$, where $\mathcal{H}$ is the Poincar\'{e}'s hemisphere determining the Poincar\'{e} compactification.
The boundary $\mathcal{E}$ of the upper-half tangent space (the red line) is the horizon.
\par
(b): Poincar\'{e} compactification with type $\alpha = (1,1)$.
The image $T(y)$ of the original point $y\in \mathbb{R}^2$ is defined as the projection of the intersection point $P(M)\in \mathcal{H}$, given by the line segment connecting $M = (y,1)\in \mathbb{R}^3$ and the origin $O\in \mathbb{R}^3$, onto the original phase space $\mathbb{R}^2$.
The horizon is identified with $\partial \mathcal{H}$. 
The precise definition is its projection onto $\mathbb{R}^2 \times \{1\}$.
\par
(c): Parabolic compactification with type $\alpha = (1,1)$.
The image $x$ of the original point $y\in \mathbb{R}^2$ is defined as the projection of the intersection point $P(M)\in \mathcal{H}$ determined by the paraboloid $x_1^2 + x_2^2 = x_3$ in $\mathbb{R}^3$ and the line segment connecting $M = (y,0)\in \mathbb{R}^3$ and the focus point $(0,0,1)\in \mathbb{R}^3$, onto the original phase space $\mathbb{R}^2$.
The horizon is identified with the circle $\{x_1^2 + x_2^2 = 1\}$ on the parabola. 
The precise definition is its projection onto $\mathbb{R}^2\times \{0\}$.
\end{figure}

\subsubsection{Directional compactifications: advantages and disadvantages}
A typical way to study dynamics at infinity is the application of directional compactifications introduced in Section \ref{sec:dir}, which is simple in the sense that the magnitude of points in the original coordinate can be measured by an independent variable $s$.
Heuristically, associated desingularized vector fields are as complex as the original vector fields \KMa{because} the new variable $\hat x_i$ in (\ref{dir-cpt}) depends only on the original variable $y_i$ and the scaling variable $s$. 
Moreover, $\hat x_i$ is proportional to $y_i$.
Characterization of blow-up times is also simple, \KMa{because} they are characterized only by the asymptotic behavior of $s = s(\tau)$.
On the other hand, directional compactifications are defined only {\em locally}.
If our interested blow-up solutions have sign-changing structure, multiple charts of compactifications can be  necessary for complete descriptions of blow-up solutions.
From the numerical viewpoint, change of coordinates may cause additional computation costs and errors.
If {\em one already knows from preceding mathematical or numerical arguments that targeting blow-up solutions have identical signs during time evolutions for a certain component}, directional compactifications 
with appropriate choice of the constant-sign components are efficient.

\subsubsection{Global compactifications: advantages and disadvantages}
If we study blow-up solutions with sign-changing structure, or one does not have sufficient knowledge of \KMa{solutions} near infinity, globally defined compactifications like the Poincar\'{e}-type and the parabolic-type are more appropriate than directional ones, \KMa{because} one does not suffer from violation of integrations of differential equations due to the change of signs, or change of local charts.
\KMa{Because} the horizon, topologically sphere-shaped boundary of the compactified space, is invariant under associated desingularized vector fields \KMa{(cf. \cite{Mat2018})}, {\em \KMa{computed trajectories through points in $\mathcal{D}$ for desingularized vector fields are always inside $\overline{\mathcal{D}}$}, unless unrealistic or mathematically inappropriate choice of numerical parameters}.
On the other hand, application of such global compactifications generally increases the \KMa{degree} of associated desingularized vector fields as polynomial ones, which cause complication of arguments.
For example, in the case of the vector field shown in Section \ref{sec:demo3}, we have to study (desingularized) polynomial vector fields with \KMa{degree} over $10$, while the original one before compactification \KMa{has degree} at most $2$ or $3$.
Without systematic implementations of vector fields or their derivatives like automatic differentiations, applications to concrete systems require lengthy calculations.

\subsubsection{Poincar\'{e}-type or parabolic-type ?}
\label{sec:choice-compactification}

Among globally defined compactifications, more than one compactifications are introduced here, the Poincar\'{e}-type and the parabolic-type.
The simplest one in the class of {\em admissible} compactifications \KMa{(e.g., \cite{EG2006, MT2020_1})} is the Poincar\'{e}-type, which is easy to understand from geometric viewpoints and widely applied in many fields of mathematics.
However, the Poincar\'{e}-type compactification has an unavoidable defect, the presence of {\em radicals} in the definition.
Radicals generally lose the smoothness of desingularized vector fields on the horizon.
In other words, desingularized vector fields under the Poincar\'{e}-type compactification are $C^0$ but not $C^1$ in general around the horizon.
Therefore typical \lq\lq linear stability analysis" in the theory of dynamical systems does not always make sense on the horizon.
Nevertheless, it should be noted that there is an exception where the Poincar\'{e}-type compactifications can be applied without losing the smoothness of resulting vector fields, which is the case if $f$ is {\em quasi-homogeneous (not only in the asymptotic sense), or the residual term $f - f_{\alpha, k}$ has sufficiently low degree}. 
In this case, the associated desingularized vector field is also smooth and hence no obstruction of $C^1$ smoothness on the horizon arises.
Details are discussed in \cite{Mat2018}.
\par
Although the \KMa{degree} of polynomials significantly increases when we apply the parabolic-type compactifications, we do not worry about the lack of smoothness of desingularized vector fields.
Indeed, {\em parabolic-type transformations of the present type originally transforms rational functions into rational ones}, unlike the Poincar\'{e}-type ones (cf. \cite{G2004}).
We thus do not \KMa{suffer from obstructions} to consider dynamics at infinity when we apply parabolic-type compactifications.

\subsubsection{\KMa{The other choice ?}}
The geometrically simplest compactification would be the one-point compactifications such as embedding of $\mathbb{R}^n$ into $S^n$, which is known as the {\em Bendixson's compactification}.
One can use the Bendixson's compactification to map the infinity to a bounded point, where the corresponding dynamics possess the high degeneracy in general (e.g. \cite{H2010}).
In order to avoid the degeneracy at infinity, we have to apply an additional desingularization (blowing-up) of the infinity.
The Poincar\'{e}-type and the parabolic-type compactifications avoid such extra tasks for obtaining desingularized dynamics at infinity.

\section{Preliminary 2: Parameterization method}

\label{sec:parameterization_method}

In this section, we introduce the theory of the parameterization method \cite{MR1976079,MR1976080,MR2177465} to compute rigorous charts of local stable and unstable manifolds of fixed points of ODEs of the form $\dot x = g(x)$, where $g$ is a desingularized vector field. We begin by making some assumptions, which will be sufficient for the purpose of the present paper. 

\begin{itemize}
\item[A1.] Assume $g:\R^n \to \R^n$ is a polynomial vector field with a steady state $\tx \in \R^n$ (i.e. $g(\tx)=0$).
\item[A2.] Assume that the eigenvalues of the Jacobian matrix $Dg(\tx)$ are real, nonzero and distinct (hence the Jacobian matrix $Dg(\tx)$ is diagonalizable over the real and $\tx$ is hyperbolic). 
\end{itemize} 

Denote by $\lambda_1,\dots,\lambda_m<0$ the {\em stable eigenvalues} of $Dg(\tx)$ with $\xi_1,\dots,\xi_m \in \R^n$ some associated {\em stable eigenvectors}.
From now on, we focus on the computation of a local stable manifold, which we denote by $W^s_{\text{loc}}(\tx)$, and note that $ \dim W^s_{\text{loc}}(\tx) = m \le n$. The computation of the unstable manifold is similar (e.g. see \cite{BLM2016}). The idea of the computational approach is to represent the chart of the local stable manifold using a Taylor series representation $P\colon B^m \to \R^n$ of the form 
\begin{equation} \label{eq:parameterization_multi-indices_taylor}
P(\theta) = \sum_{|\alpha|=0}^\infty a_\alpha \theta^\alpha, \qquad a_\alpha \in \R^n,
\end{equation}
where $B^m \subset \R^m$ is a domain (usually chosen to be a ball) on which the Taylor series converges, and where $\alpha = (\alpha_1,\dots,\alpha_m) \in \mathbb{N}^m$, $|\alpha| = \alpha_1 + \dots + \alpha_m$, $\theta=(\theta_1,\dots,\theta_m) \in \R^m$ and $\theta^\alpha = \theta_1^{\alpha_1} \cdots \theta_m^{\alpha_m}$. This requires making an extra assumption, which involves the notion of a resonance. 
\begin{dfn}
{\em
The eigenvalues $\lambda_1,\dots,\lambda_m$ are said to have a {\em resonance} of order $\alpha = (\alpha_1,\dots,\alpha_m) \in \mathbb{N}^m$ if
\begin{equation} \label{eq:resonance}
\alpha_1 \lambda_1 + \dots + \alpha_m \lambda_m - \lambda_j = 0,
\end{equation}
for some $j \in \{ 1, \dots,m\}$ with $|\alpha| \ge 2$. If there are no resonances at any order $|\alpha| \ge 2$, then the eigenvalues $\lambda_1,\dots,\lambda_m$ are said to be {\em non-resonant}.
}
\end{dfn}

We are ready to state our third hypothesis. 

\begin{itemize}
\item[A3.] Assume that the eigenvalues $\lambda_1,\dots,\lambda_m$ are {\em non-resonant}
\end{itemize} 

%
Construct the following real-valued matrices: an $m \times m$ diagonal matrix with the diagonal entries made up of the stable eigenvalues
\begin{equation}
\label{eq:diagonalL}
\Lambda = \begin{pmatrix}
 \lambda_1 &  \ldots & 0  \\
 \vdots & \ddots  & \vdots  \\
 0  & \ldots  & \lambda_m  
\end{pmatrix}
\end{equation}
and an $n \times m$ matrix whose columns are the associated eigenvectors 
\[
 A_0 = [\xi_1 | \ldots | \xi_m].
\]

Using the  basis defined by the stable eigenvectors, the linearized equation for $\dot x = g(x)$ restricted to the stable subspace takes the form
\[
\dot{y} = \Lambda y,\quad y\in \R^m.
\]
The associated flow is given by $e^{\Lambda t}$.
As indicated above our goal is to construct an analytic function $P\colon B^m \to \R^n$ such that $P(B^m) = \KMa{W_{\rm loc}^s}(\tx)$.
To obtain constraints, so that we can solve for $P$, we begin by insisting that $P$ be a conjugacy between the flow $\varphi$ of $\dot x = g(x)$ restricted to $W_{\text{loc}}^s(\tx)$ and the flow $e^{\Lambda t}$ of the linear equation.
The most obvious restriction  is that $P$ must map fixed points to fixed points and hence
\[
P(0) = \tx.
\]
To obtain the conjugacy we assume that
\[
DP(0) = A_0
\]
and
\begin{equation} \label{eq:conjRelation}
\varphi \left( t, P(\theta) \right) = P\left( e^{\Lambda t} \theta \right),
\end{equation}
for all $\theta \in B^m$. 
The geometric meaning
of this conjugacy is illustrated in Figure \ref{fig:parmConj}.
To see that $P \left( B^m\right)\subset W^s_{\text{loc}}(\tx)$ observe that
\[
\lim_{t \to \infty} \varphi (t, P(\theta) ) = \lim_{t \to \infty} 
	P\left( e^{\Lambda t} \theta \right)
			= P\left(\lim_{t \to \infty} e^{\Lambda t} \theta \right) 
		    = P(0) 
		    = \tx,
\]
\KMa{because} the entries of $\Lambda$ are negative.

\begin{figure}[h!] 
\begin{center}
\includegraphics[width=10cm]{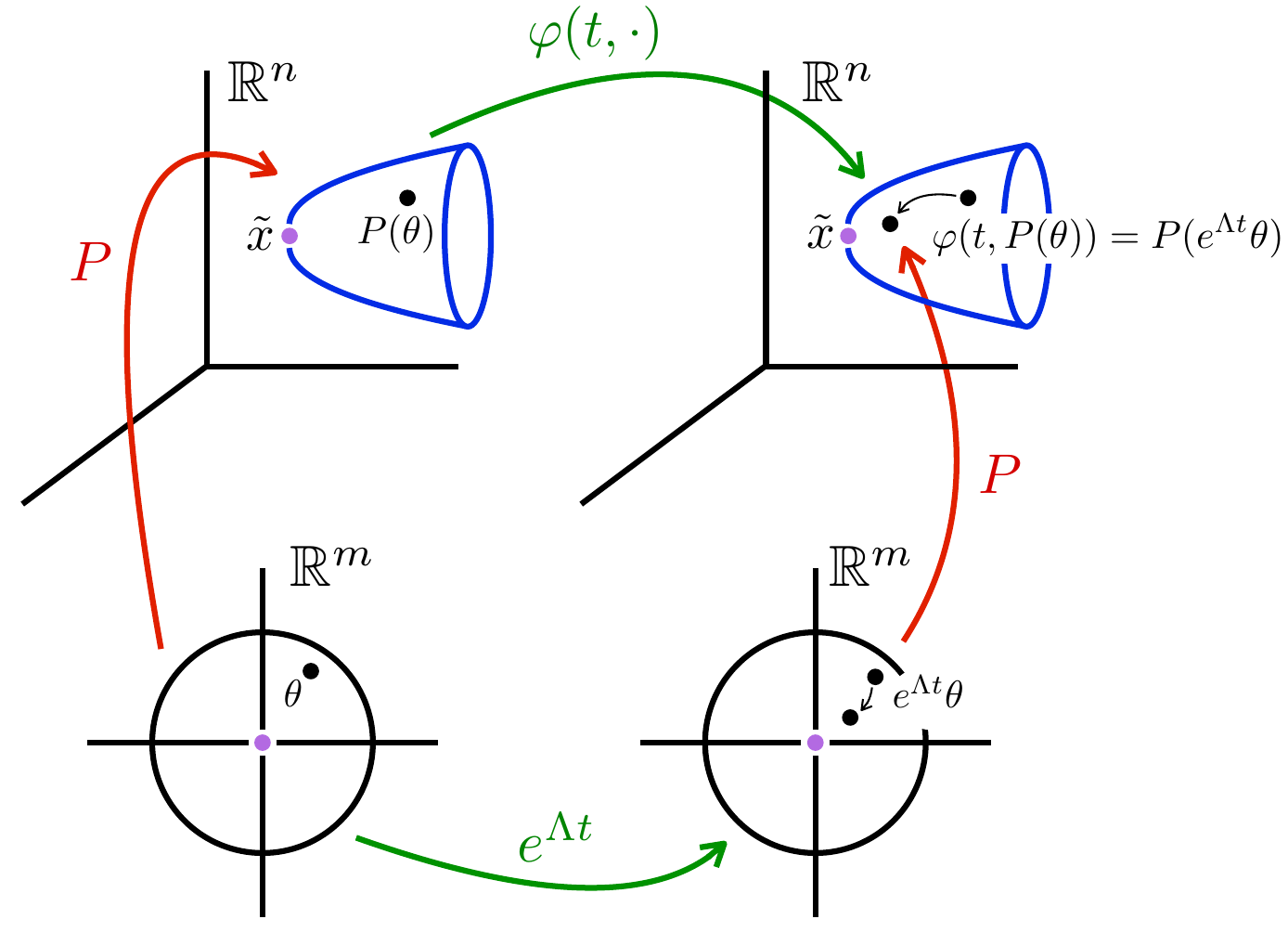}
\end{center}
\vspace{-.5cm}
\caption{\textbf{Schematic of the Parameterization Method for Vector Fields in $\R^n$:}
The figure illustrates the conjugacy described by Equation 
\eqref{eq:conjRelation}.  The bottom half of the figure 
represents the parameter space in $\R^m$ (the domain of the parameterization $P$)
while the top half of the figure represents the phase
space in $\R^n$.  The image of $P$ is the local stable manifold shown in blue.
The dynamics are depicted by moving from the left to the right side 
of the figure.  The dynamics in the parameter space is
generated by exponentiating the matrix of stable eigenvalues $\Lambda$. 
The dynamics in phase space is generated by the flow $\varphi$ associated with the vector field $g$.  The diagram {\em commutes} in the sense that applying first the chart map $P$ and then nonlinear flow $\varphi(t,\cdot)$ is required to be the same as applying the linear dynamics $e^{\Lambda t}$ and then the chart map $P$. The result is that the dynamics on the local stable manifold are described by the stable linear dynamics.} \label{fig:parmConj}
\end{figure}

Note that any function $P(\theta)$ satisfying Equation (\ref{eq:conjRelation}) is one-to-one on $B^m$.  
To see this observe that $P$ is tangent to the stable eigenspace at the origin as $DP(0) = A_0$.  
Moreover recall that $A_0$ is of full rank as its columns are linearly independent.  
By the implicit function theorem\KMa{,} $P$ is of rank $m$, and hence one-to-one, in some neighborhood $U\subset B^m$ of $0$.
Now suppose that $\theta_1, \theta_2 \in B^m$ and that $P(\theta_1) = P(\theta_2)$.  Then 
for any $t \in \mathbb{R}$, $\varphi \left( t, P(\theta_1) \right) = \varphi \left( t, P(\theta_2) \right)$
by the uniqueness of the initial value problem.  
Choose $T >0$  large enough so that 
$e^{\Lambda T} \theta_1, e^{\Lambda T} \theta_2 \in U$.
By the conjugacy relation we have that
$
P\left( e^{\Lambda t} \theta_1 \right) = P\left( e^{\Lambda t} \theta_2 \right),
$
and \KMa{because} the arguments are in $U$, the local immersion gives that $e^{\Lambda T} \theta_1= e^{\Lambda T} \theta_2$.  
But $e^{\Lambda T}$ is an isomorphism and we have $\theta_1 = \theta_2$. 
We \KMa{therefore} conclude from the discussion above that $P(B^m) = W_{loc}^s(\tx)$.

The utility of \eqref{eq:conjRelation} is limited by the appearance of the flow $\varphi$ in the equation. 
In practice the flow is only known implicitly, that is it is determined by solving the differential equation.  
The following lemma establishes a more practical infinitesimal version of  \eqref{eq:conjRelation}.

\begin{lem} \label{lem:parmLemma} 
Let  $P \colon B^m \subset \mathbb{R}^m \to \mathbb{R}^n$
be a smooth function with 
\begin{equation}\label{eq:parmLinearConstraints}
P(0) = \tx \qquad \text{and}\qquad DP(0) = A_0.
\end{equation}
Then $P(\theta)$ satisfies the conjugacy relationship (\ref{eq:conjRelation})
if and only if $P$ is a solution of the partial differential equation (PDE)
\begin{equation} \label{eq:parmFE}
\lambda_1 \theta_1 \frac{\partial}{\partial \theta_1} P(\theta_1, \ldots, \theta_m) 
+ \ldots +
\lambda_m \theta_m \frac{\partial}{\partial \theta_m} P(\theta_1, \ldots, \theta_m) 
= g(P(\theta_1, \ldots, \theta_m))
\end{equation}
for all $\theta = (\theta_1, \ldots, \theta_m) \in B^m$.
\end{lem}
\begin{proof}
Let $P \colon  B^m \to \mathbb{R}^n$ 
be a smooth function with $P(0) = \tx$ and $DP(0) = A_0$. \\ 
($\Longleftarrow$) Suppose that $P(\theta)$ solves the partial differential
equation \eqref{eq:parmFE} in $B^m$.  Choose a
fixed $\theta \in B^m$ and fix $t > 0$.  Define the function 
$\gamma \colon [0, t] \to \mathbb{R}^n$ by
\begin{equation}\label{e:defgamma}
\gamma(t) \bydef P\left( e^{\Lambda t} \theta \right).
\end{equation}
Then, $\gamma(0) = P(\theta)$ and 
\[
\gamma'(t) = \frac{d}{dt}  P\left( e^{\Lambda t} \theta \right)
   	= D P\left( e^{\Lambda t} \theta \right) \Lambda e^{\Lambda t} \theta
 	= g\left( P\left( e^{\Lambda t} \theta \right)\right)
 	= g ( \gamma(t)),
\]
where we pass from the first to the second equality by the chain rule, from the
second to the third equality by the invariance equation (\ref{eq:parmFE}) and the fact
that $e^{\Lambda t}\theta \in B^m$ when $t > 0$,
and from the third to the fourth equation by the definition of $\gamma$.
Hence $\gamma$ is the solution of the initial value problem
\begin{equation}\label{e:flowgamma}
\gamma'(t) = g(\gamma(t)) , 
\qquad \text{and} \qquad 
\gamma(0) = P(\theta).
\end{equation}
Therefore by definition $\varphi( t, \gamma(0)) =\gamma(t)$, and it follows from~\eqref{e:defgamma} and~\eqref{e:flowgamma} that 
\[
\varphi( t , P(\theta)) = P( e^{\Lambda t} \theta).
\]

($\Longrightarrow$) Suppose  that $P$ satisfies
the conjugacy relationship \eqref{eq:conjRelation}
for all $\theta \in B^m$. 
Fix $\theta \in B^m$ and differentiate both sides 
with respect to $t$ in order to obtain
\[
g(\varphi ( t , P(\theta) )) = DP( e^{\Lambda t} \theta) \Lambda e^{\Lambda t} \theta.
\]
Taking the limit as $t \to 0$ gives that $P(\theta)$ is a solution 
of \eqref{eq:parmFE}. 
\end{proof}

As a consequence of Lemma~\ref{lem:parmLemma}, it should now be clear that computing a local $m$-dimensional stable manifold is equivalent to find a solution $P\colon B^m \to \R^n$ of the PDE \eqref{eq:parmFE}. As mentioned earlier, the idea is to use a Taylor series representation of the form \eqref{eq:parameterization_multi-indices_taylor}. Note that since $g:\R^n \to \R^n$ is a polynomial vector field, the power series expansion of $g(P(\theta))$ involves Cauchy products. Denote the Taylor expansion of $g(P(\theta))$ as 
\[
g(P(\theta)) = \sum_{|\alpha|=0}^\infty \left( g(a) \right)_\alpha \theta^\alpha, \qquad j=1,\dots,n,
\]
where we abuse slightly the notation and used the same notation $g(a)$ to denote the vector field $g$ where the monomial terms in the variables $x_1,\dots,x_n$ are replaced by Cauchy products in the variables $a_1,\dots,a_n$.

Formally plugging the Taylor expansion \eqref{eq:parameterization_multi-indices_taylor} in the PDE \eqref{eq:parmFE}
results in 
\[
DP(\theta) \Lambda \theta = 
\sum_{|\alpha|=0}^\infty (\alpha \cdot \lambda )a_\alpha \theta^\alpha=\sum_{|\alpha|=0}^\infty \left( g(a) \right)_\alpha \theta^\alpha= g(P(\theta)),
\]
where $\alpha \cdot \lambda \bydef \alpha_1 \lambda_1 + \dots + \alpha_m \lambda_m$ and $a_\alpha = ((a_1)_\alpha,\dots,(a_n)_\alpha) \in \R^n$. The first order constraints \eqref{eq:parmLinearConstraints} imply that 
\[
a_0 = ((a_1)_0,\dots,(a_n)_0) = \tx \in \R^n \qquad \text{and} \qquad a_{e_j} = \xi_j \in \R^n \quad (j=1,\dots,m),
\]
where $e_j$ is the $j^{th}$ vector of the canonical basis of $\R^n$. In other words, the Taylor coefficients $a_\alpha$ for $|\alpha| \in \{0,1\}$ are fixed and do not need to be solved for. 

Computing the higher order Taylor coefficients $a_\alpha=((a_1)_\alpha,\dots,(a_n)_\alpha)$ (for $|\alpha| \ge 2$) of \eqref{eq:parameterization_multi-indices_taylor} reduces to find the solution of the zero finding problem $F(a)=0$, with $F$ given by 
\begin{equation} \label{eq:parameterization_F(a)=0}
(F(a))_\alpha \bydef (\alpha \cdot \lambda)a_\alpha - (g(a))_\alpha, \qquad |\alpha| \ge 2.
\end{equation}
Also denote, for $j=1,\dots,n$ and $|\alpha| \ge 2$,
\[
(F_j(a))_\alpha \bydef (\alpha \cdot \lambda)(a_j)_\alpha - (g_j(a))_\alpha, 
\]
so that we may write  $F(a) = (F_1(a),F_2(a), \dots, F_n(a))$.

\begin{rem}
When $|\alpha| \in \{0,1\}$, the constraints $(F(a))_\alpha=0$ correspond to finding the steady state ($|\alpha|=0$) and the stable eigenvalues/eigenvectors ($|\alpha|=1$). Since this information is already assumed to be at hand, we only need to solve for $(F(a))_\alpha=0$ for $|\alpha| \ge 2$.
\end{rem}

Denote the Banach space
\begin{equation} \label{eq:ell_one}
\ell^1 \bydef \left\{ b = (b_\alpha)_{| \alpha| \ge 2} : b_\alpha \in \R \text{ and } \|b\|_1 \bydef \sum_{|\alpha| =2}^\infty |b_\alpha| < \infty \right\}
\end{equation}
and the product Banach space $X \bydef (\ell^1)^n = \ell^1 \times \ell^1 \times \cdots \times \ell^1$ with induced norm
\begin{equation} \label{eq:X_norm}
\|a\|_X \bydef \max_{j=1,\dots,n} \| a_j \|_1.
\end{equation}
Moreover, denoting the Banach space
\begin{equation} \label{eq:t_ell_one}
\tilde \ell^1 \bydef \left\{ b = (b_\alpha)_{| \alpha| \ge 2} : b_\alpha \in \R \text{ and } \sum_{|\alpha| =2}^\infty |(\alpha \cdot \lambda)b_\alpha| < \infty \right\},
\end{equation}
and $X' \bydef (\tilde \ell^1)^n$, we get that $F:X \to X'$. 

Denote by $B_1^m \bydef \{ z = (z_1,\dots,z_m) \in \mathbb{C}^m : |z_k| \le 1, \text{ for all } k=1,\dots,m \}$ the unit polydisc in $\mathbb{C}^m$. 
We have the following result.

\begin{thm} \label{thm:equivalent_zero_finding}
Assume that Assumptions A1, A2 and A3 are satisfied. If there exists $\ta \in X$ such that $F(\ta)=0$ with $F$ given in \eqref{eq:parameterization_F(a)=0}, then the corresponding Taylor expansion $P:B_1^m \to \R^n$ given by
\begin{equation} \label{eq:parameterization_proof}
P(\theta) \bydef \tx + \sum_{k=1}^m \xi_k \theta_k + \sum_{|\alpha|=2}^\infty \ta_\alpha \theta^\alpha
\end{equation}
provides a parameterization of a local stable manifold of $\tx$, that is $P(B_1^m) = W^s_{\rm loc} (\tx)$.
\end{thm}

\begin{proof}
Assume that $\ta \in X$ solves $F(\ta)=0$. Then by construction, the function $P(\theta)$ given in \eqref{eq:parameterization_proof} converges absolutely and uniformly on $B_1^m$ as for each $j \in \{1,\dots,n\}$
\begin{align*}
\sup_{z \in B_1^m} |P_j(z)| &\le |\tx_j| + \sup_{z \in B_1^m} \left|  \sum_{k=1}^m (\xi_k)_j z_k 
+ \sum_{|\alpha|=2}^\infty (\ta_j)_\alpha z_1^{\alpha_1} \cdots z_m^{\alpha_m} \right| 
\\
& \le |\tx_j| + \sup_{z \in B_1^m}  \sum_{k=1}^m |(\xi_k)_j | |z_k| 
+ \sup_{z \in B_1^m} \sum_{|\alpha|=2}^\infty |(\ta_j)_\alpha| |z_1|^{\alpha_1} \cdots |z_m|^{\alpha_m} 
\\
& \le |\tx_j| +  \sum_{k=1}^m |(\xi_k)_j |
+ \sum_{|\alpha|=2}^\infty |(\ta_j)_\alpha| 
\\
& = |\tx_j| +  \sum_{k=1}^m |(\xi_k)_j | + \|\ta_j\|_1 < \infty,
\end{align*}
since $\|\ta\|_1<\infty$. By construction, the function $P:B_1^m \to \R^n$ given in \eqref{eq:parameterization_proof} satisfies the first order constraints \eqref{eq:parmLinearConstraints} and the PDE \eqref{eq:parmFE}. By the Lemma~\ref{lem:parmLemma},  $P$ satisfies the conjugacy relationship (\ref{eq:conjRelation}). Finally, we conclude that $P:B_1^m \to \R^n$ provides a parameterization of a local stable manifold of $\tx$, that is $P(B_1^m) = W^s_{\rm loc} (\tx)$.
\end{proof}

The strategy to compute a parameterization of $W^s_{\rm loc} (\tx)$ is now clear. Fix the lengths of the eigenvectors $\xi_1,\dots,\xi_m$ such that we can compute $\ta \in X$ such that $F(\ta)=0$. This is achieved with a Newton-Kantorovich type argument, which we now state. 

Denote by $B_r(b) \bydef \{ x \in X : \| x - b \|_X \le r\}$ the closed ball of radius $r>0$ centered at a given $b \in X$, and $B(X_1,X_2)$ the space of bounded linear operators between two Banach spaces $X_1$ and $X_2$.

\begin{thm}[{\bf A Newton-Kantorovich type theorem}] \label{thm:radii_polynomials}
Let $X$ and $X'$ be Banach spaces, $A^{\dagger} \in B(X,X')$ and $A \in B(X',X)$ be bounded linear operators.
Assume $F \colon X \to X'$ is Fr\'echet differentiable at $\ba \in X$, $A$ is injective and $A F \colon X \to X.$
Let $Y_0$, $Z_0$ and $Z_1$ be nonnegative constants, and a function $Z_2:(0,\infty) \to (0,\infty)$ satisfying
\begin{align}
\label{eq:general_Y_0}
\| A F(\ba) \|_X &\le Y_0
\\
\label{eq:general_Z_0}
\| I - A A^{\dagger}\|_{B(X)} &\le Z_0
\\
\label{eq:general_Z_1}
\| A[DF(\ba) - A^{\dagger}] \|_{B(X)} &\le Z_1,
\\
\label{eq:general_Z_2}
\| A[DF(c) - DF(\ba)]\|_{B(X)} &\le Z_2(r) r, \quad \text{for all } c \in B_r(\ba),
\end{align}
where $\| \cdot \|_{B(X)}$ denotes the operator norm.  Define the radii polynomial by 
\begin{equation} \label{eq:general_radii_polynomial}
p(r) \bydef Z_2(r) r^2 - ( 1 - Z_1 - Z_0) r + Y_0.
\end{equation}
If there exists $r_0>0$ such that $p(r_0)<0$, then there exists a unique $\ta \in B_{r_0}(\ba)$ such that $F(\ta) = 0$.
\end{thm}

The strategy of Theorem~\ref{thm:radii_polynomials} requires obtaining $\ba$ (a numerical approximation), the operator $A^{\dagger} \in B(X,X')$ (an approximation of the Fr\'echet derivative $DF(\ba)$) and the operator $A \in B(X',X)$ (an approximate inverse of $DF(\ba)$).

To compute the numerical approximation $\ba$, we first consider a finite dimensional projection of the map $F:X\to X'$. Fixing a dimensional Taylor projection number $N$, denote by $X^{(N)}$ the finite dimensional space 
\[
X^{(N)} \bydef \left\{ a = (a_1,\dots,a_n) : a_j =  \left( (a_j)_{\alpha} \right)_{|\alpha|=2}^N, \text{ for } j=1,\dots,n \right\}.
\]
Moreover, denote by $\kappa(N) \bydef \# \left\{ \alpha \in \N^m : |\alpha| \in \{2,\dots,N\} \right\}$ the number of multi-indices $\alpha$ with order between $2$ and $N$. Given a vector $b = (b_\ell)_{|\ell| \ge 0} \in \ell^1$, consider the projection
\begin{align*}
\pi^N : \ell^1& \to \R^{\kappa(N)} \\ 
b &\mapsto \pi^N b \bydef (b_\alpha)_{|\alpha|=2}^N \in \R^{\kappa(N)}.
\end{align*}
We generalize that projection to get $\Pi^N:X=(\ell^1)^{n} \to X^{(N)} \cong \R^{n\kappa(N)}$ defined by
\[
\Pi^N a \bydef (\pi^N a_1,\dots,\pi^N a_n) \in X^{(N)}.
\]
Given $a \in X$, we denote
\[
a^{(N)} \bydef \Pi^{N} a \in X^{(N)}.
\]
Moreover, we define the natural inclusion $\iota^N : \R^{\kappa(N)} \xhookrightarrow{} \ell^1$ as follows. For $b =(b_\alpha)_{|\alpha|=2}^N \in \R^{\kappa(N)}$ let $\iota^N b \in \ell^1$ be defined component-wise by
\[
\left( \iota^N b \right)_{\alpha}
= 
\begin{cases}
b_\alpha, & |\alpha|=2,\dots,N
\\
0, & |\alpha| > N.
\end{cases}
\]
Similarly, let $\inc:X^{(N)} \xhookrightarrow{} X$ be the natural inclusion defined as follows. Given $a = (a_1,\dots,a_n) \in X^{(N)} \cong \R^{n\kappa(N)}$, let
\[
\inc a \bydef \left( \iota^N a_1, \dots,\iota^N a_n \right) \in X.
\]
Finally, define the {\em finite dimensional projection} $F^{(N)} :X^{(N)} \to X^{(N)}$ by
\begin{equation} \label{eq:F_map_projection}
F^{(N)}(a) = \Pi^{(N)} F(\inc a).
\end{equation}
Also denote $F^{(N)} = \left(F_1^{(N)},\dots,F_n^{(N)} \right)$.

Assume that a numerical approximation $\ba^{(N)} = \left( \ba_1^{(N)}, \dots, \ba_n^{(N)} \right)$ such that $F^{(N)}(\ba^{(N)}) \approx 0$ has been computed (e.g.~using Newton's method).
Given $j=1,\dots,n$, denote $\ba_j = \iota^N \ba_j^{(N)} \in \ell^1$ and denote $\ba = (\ba_1,\dots,\ba_n)$, and for the sake of simplicity of the presentation, we use the same notation $\ba$ to denote $\ba \in X$ and $\ba^{(N)} \in X^{(N)}$. Denote by $DF^{(N)}(\ba)$ the Jacobian of $F^{(N)}$ at $\ba$, and let us write it as
\[
DF^{(N)}(\ba)=
\begin{pmatrix}
D_{a_1} F_1^{(N)}(\ba) & \cdots & D_{a_n} F_1^{(N)}(\ba) \\
\vdots & \ddots & \vdots \\
D_{a_1} F_n^{(N)}(\ba) & \cdots & D_{a_n} F_n^{(N)}(\ba)
\end{pmatrix} \in M_{n\kappa(N)}(\R).
\]
The next step is to construct the linear operator $A^\dag$ (an approximate derivative of the derivative $DF(\ba)$), and the linear operator $A$ (an approximate inverse of $DF(\ba)$). Let
\begin{equation} \label{eq:operator_dagA}
A^\dagger=
\begin{pmatrix}
A_{1,1}^\dagger & \cdots & A_{1,n}^\dagger\\
\vdots & \ddots & \vdots \\
A_{n,1}^\dagger & \cdots & A_{n,n}^\dagger
\end{pmatrix} ,
\end{equation}
whose action on an element $h=(h_1,\dots,h_n) \in X$ is defined by $(A^\dagger h)_i = \sum_{j=1}^n A_{i,j}^\dagger h_j$, for $i=1,\dots,n$. Here the action of $A_{i,j}^\dagger$ is defined as
\begin{align*}
(A_{i,j}^\dagger h_j)_n &= 
\begin{cases}
\bigl(D_{a_j} F_i^{(N)}(\ba) h_j^{(N)} \bigr)_\alpha &\quad\text{for }   2 \leq |\alpha| \le N ,  \\
\delta_{i,j} (\alpha \cdot \lambda) (h_j)_\alpha  &\quad\text{for }  |\alpha| > N,
\end{cases}
\end{align*}
where $\delta_{i,j}$ is the Kronecker $\delta$. 
Consider now a matrix $A^{(N)} \in M_{n\kappa(N)}(\R)$ computed so that $A^{(N)} \approx {DF^{(N)}(\ba)}^{-1}$. We decompose it into $n^2$ $\kappa(N) \times \kappa(N)$ blocks:
\[
A^{(N)}=
\begin{pmatrix}
A_{1,1}^{(N)} & \cdots& A_{1,n}^{(N)}\\
\vdots & \ddots & \vdots \\
A_{n,1}^{(N)} & \cdots& A_{n,n}^{(N)}
\end{pmatrix}.
\]
This allows defining the linear operator $A$ as
\begin{equation} \label{eq:operator_A}
A=
\begin{pmatrix}
A_{1,1} & \cdots& A_{1,n}\\
\vdots & \ddots & \vdots \\
A_{n,1} & \cdots & A_{n,n}
\end{pmatrix},
\end{equation}
whose action on an element $h=(h_1,\dots,h_n) \in X$ is defined by $(Ah)_i = \sum_{j=1}^n A_{i,j} h_j$, for $i=1,\dots,n$. 
Given $i,j \in\{1,\dots,n\}$, the action of $A_{i,j}$ is defined as
 \begin{align*}
 (A_{i,j} h_j)_n &=
 \begin{cases}
\left(A_{i,j}^{(N)} h_j^{(N)} \right)_\alpha & \text{for }   2 \leq |\alpha| \le N    \\
\delta_{i,j} \frac{1}{\alpha \cdot \lambda} (h_j)_\alpha  & \text{for }  |\alpha| > N.
 \end{cases}
  \end{align*}
Having obtained an approximate solution $\ba$ and the linear operators $\dagA$ and $A$, the next step is to construct the bounds $Y_0$, $Z_0$, $Z_1$ and $Z_2(r)$ satisfying \eqref{eq:general_Y_0}, \eqref{eq:general_Z_0}, \eqref{eq:general_Z_1} and \eqref{eq:general_Z_2}, respectively. 

\subsection{The \boldmath $Y_0$ \unboldmath bound} \label{sec:Y0}

Denote by $d$ the highest order nonlinear term of the vector field $f$. Then since $\ba$ consists of Taylor coefficients of order $N$, then $(F(\ba))_\alpha=0$ for all $|\alpha|>dN$. For $i=1,\dots,n$, we set
\[
Y_0^{(i)} \bydef
\sum_{|\alpha|=2}^{N} \biggl| \sum_{j=1}^n \left( A^{(N)}_{i,j} F^{(N)}_j(\ba) \right)_\alpha \biggr| 
+ \sum_{|\alpha|=N+1}^{dN} \biggl| \frac{1}{\alpha \cdot \lambda}  (F_i(\ba))_{\alpha} \biggr| 
\]
which is a collection of finite sums that can be evaluated with interval arithmetic. We conclude that 
\[
\| [ AF(\ba)]_i \|_1 = \biggl\| \sum_{j=1}^n A_{i,j} F_j(\ba) \biggr\|_1
\le Y_0^{(i)}, \qquad\text{for } i=1,\dots,n
\]
and we set 
\begin{equation} \label{eq:Y0_stable_manifold}
Y_0 \bydef \max \left(Y_0^{(1)},\dots,Y_0^{(n)} \right).
\end{equation}

\subsection{The \boldmath $Z_0$ \unboldmath bound} \label{sec:Z0}
We look for a bound of the form
$
\| I - A A^{\dagger}\|_{B(X)} \le Z_0
$.
Recalling the definitions of $A$ and $\dagA$ given in \eqref{eq:operator_A} and \eqref{eq:operator_dagA},
let $B \bydef I - A \dagA$ the bounded linear operator represented as
\begin{equation*} 
B = 
		\begin{pmatrix}
		 B_{1,1} & \cdots & B_{1,n} \\
		 \vdots & \ddots & \vdots \\
		 B_{n,1} & \cdots & B_{n,n}
		\end{pmatrix}.
\end{equation*}
We remark that $( B_{i,j} )_{n_1,n_2}=0$ for any $i,j \in \{1,\dots,n\}$, whenever $n_1 > N$  or $n_2 > N $.
Hence we can compute the norms $\|B_{i,j}\|_{B(\ell^1)}$ using the following standard result.

\begin{lem}\label{l:Blnu1norm}
Given $\Gamma \in B(\ell^1)$ a bounded linear operator, acting as $(\Gamma a)_\beta =\sum_{|\alpha| \ge 2} \Gamma_{\beta,\alpha} a_\alpha$ for $|\beta| \ge 2$. 
\begin{equation} \label{eq:operator_norm_ell_nu_1}
   \| \Gamma \|_{B(\ell^1)} = \sup_{|\alpha| \ge 2} \sum_{|\beta| \ge 2} | \Gamma_{\beta,\alpha} |.
\end{equation}
\end{lem}

Given $h =(h_1,\dots,h_n)  \in X = (\ell^1)^n$ with $\|h\|_X = \max(\|h_1\|_1,\dots,\|h_n\|_1) \leq 1$,
and for $i=1,\dots,n$, we obtain
\[
\|(Bh)_i \|_1 = \biggl\|  \sum_{j=1}^n B_{i,j} h_j \biggr\|_1 
\le   \sum_{j=1}^n \| B_{i,j} \|_{B(\ell^1)}.
\]
Hence we define
\begin{equation} \label{eq:Z0_1d_stable_manifold}
Z_0 \bydef \max_{i=1,\dots,n}
\left(
\sum_{j=1}^n \| B_{i,j} \|_{B(\ell^1)}
  \right),
\end{equation}
where each norm $\|B_{i,j}\|_{B(\ell^1)}$ can be computed using formula \eqref{eq:operator_norm_ell_nu_1} with vanishing tail terms.

\subsection{The \boldmath $Z_1$ \unboldmath bound} \label{sec:Z1} 

Recall that we look for the bound 
$
\| A[DF(\ba) - A^{\dagger} ] \|_{B(X)} \le Z_1
$.
Given $h=(h_1,\dots,h_n) \in X$ with $\|h\|_X \le 1$, set
\[
z \bydef [DF(\ba) -  A^{\dagger} ] h.
\]
Then, for each $i=1,\dots,n$, $(z_i)_\alpha = 0$ for $|\alpha| = 2,\dots,N $ and for $|\alpha| >N$,
\[
(z_i)_\alpha = -\left( Dg_i(\ba)h \right)_\alpha = -\left( \sum_{j=1}^n \frac{\partial g_i}{\partial a_j}(\ba)h_j \right)_\alpha
\]
%
Denote 
\[
\lambda^*(N) \bydef \min_{|\alpha|>N} |\alpha \cdot \lambda|.
\]
Since the tail of $A_{i,j}$ is zero for $i\neq j$, then $Az = (A_{1,1} z_1,\dots,A_{n,n} z_n)$. 
Then a straightforward calculation yields, for each $i \in \{1,\dots,n\}$, that
\[
\| A_{i,i} z_i \|_1  \le Z_1^{(i)} \bydef \frac{1}{\lambda^*(N)} \sum_{j=1}^n \left\| \frac{\partial g_i}{\partial a_j}(\ba) \right\|_1,
\]
%
so that we set 
\begin{equation} \label{eq:Z1_stable_manifold}
Z_1 \bydef \max \left(Z_1^{(1)},\dots,Z_1^{(n)} \right).
\end{equation}

\subsection{The \boldmath $Z_2$ \unboldmath bound} \label{sec:Z2}

For a fixed $r_*>0$, set 
\begin{equation} \label{eq:Z2_via_second_derivative}
  Z_2(r_*) \bydef \sup_{b \in B_{r_*}(\ba)} \biggl( \max_{i=1,\dots,n} \sum_{k,m=1}^n \left\| \sum_{j=1}^n A_{ij} 
 \frac{\partial^2 g_j}{\partial a_m \partial a_k} 
 \bigl( b \bigr) \right\|_{B(\ell^1)} \biggr)
\end{equation}
which satisfies (by the Mean Value Inequality in Banach spaces)
\[
\| A[DF(c) - DF(\ba)]\|_{B(X)} \le Z_2(r_*) r, \quad \text{for all } c \in B_r(\ba), \quad \text{for all } r \le r_*.
\]
Evaluating the bound \eqref{eq:Z2_via_second_derivative} is straightforward with interval arithmetic and the easily computed formulas of the second derivatives of each component $f_i$ of the vector field $f$.

\subsection{Rigorous enclosure of the points on \boldmath $W^s_{\rm loc} (\tx)$ \unboldmath} \label{sec:C0_manifold}

Assume that assumptions A1, A2 and A3 are satisfied for a fixed point $\tx$. Let $\lambda_1,\dots,\lambda_m<0$ be the corresponding non-resonant real (stable) eigenvalues and $\xi_1,\dots,\xi_m \in \R^n$ be some associated {\em stable eigenvectors}.

Consider $N \ge 2$ the order of the Taylor approximation, and as before, 
assume that a numerical approximation $\ba^{(N)} = \left( \ba_1^{(N)}, \dots, \ba_n^{(N)} \right)$ such that $F^{(N)}(\ba^{(N)}) \approx 0$ has been computed. Denote by
\begin{equation} \label{eq:manifold_approximation}
P^{(N)}(\theta) \bydef \tx + \sum_{k=1}^m \xi_k \theta_k + \sum_{|\alpha|=2}^N \ba_\alpha \theta^\alpha.
\end{equation}
Using a computer program in MATLAB using the interval arithmetic package INTLAB, we can compute rigorously the bounds $Y_0$, $Z_0$, $Z_1$ and $Z_2$ satisfying \eqref{eq:Y0_stable_manifold}, \eqref{eq:Z0_1d_stable_manifold}, \eqref{eq:Z1_stable_manifold} and \eqref{eq:Z2_via_second_derivative}, respectively.
Define the radii polynomial $p(r)$ defined in \eqref{eq:general_radii_polynomial}, and assume the existence of $r_0>0$ such that $p(r_0)<0$. From Theorem~\ref{thm:radii_polynomials}, there exists a unique $\ta \in B_{r_0}(\ba)$ such that $F(\ta) = 0$. By Theorem~\ref{thm:equivalent_zero_finding}, the corresponding Taylor expansion $P:B_1^m \to \R^n$ given by \eqref{eq:parameterization_proof} provides a parameterization of a local stable manifold of $\tx$, that is $P(B_1^m) = W^s_{\rm loc} (\tx)$. From the computer-assisted proof, we immediately obtain a rigorous upper bound for the $C^0$ error bound between the approximate parameterization  \eqref{eq:manifold_approximation} and the true parameterization. More explicitly, for a fixed $j=1,\dots,n$
\begin{align*}
\sup_{z \in B_1^m} |P_j(z)-P_j^{(N)}(z)| &=
\sup_{z \in B_1^m} \left| \sum_{|\alpha|=2}^\infty ((\ta_j)_\alpha-(\ba_j)_\alpha) z^\alpha \right| \\
& \le
\sup_{z \in B_1^m} \sum_{|\alpha|=2}^\infty |(\ta_j)_\alpha-(\ba_j)_\alpha| |z_1|^{\alpha_1} \cdots |z_m|^{\alpha_m} \\
& \le \sum_{|\alpha|=2}^\infty |(\ta_j)_\alpha-(\ba_j)_\alpha| =\| \ta_j -\ba_j \|_1 \le \| \ta -\ba \|_X  < r_0.
\end{align*}
Using that estimate, given a point $z \in B_1^m$ in parameter space, one may evaluate rigorously the corresponding value $P(z) \in W^s_{\rm loc} (\tx)$ on the local stable manifold using the following enclosure 
\begin{equation} \label{eq:rigorous_enclosure_on_manifold}
P_j(z) \in P_j^{(N)}(z) + [-r_0,r_0], \qquad j=1,\dots,n
\end{equation}
where $P_j^{(N)}(z)$ can be computed with interval arithmetic using the formula \eqref{eq:manifold_approximation}.

\section{Saddle-type blow-up solutions: basic methodology for validations and extensions}

\label{sec:methodology}
\KMa{In this section, we provide a methodology for validating saddle-type blow-up solutions with computer-assisted proofs.}
A remarkable feature obtained from theorems mentioned in Section \ref{sec:pre_blow_up} is that stability of equilibria on the horizon does not matter for characterizing blow-up solutions. 
Therefore we can characterize blow-up solutions whose blow-up direction is characterized by {\em unstable} equilibria\footnote{
Potentially the similar characterization of blow-up solutions can be achieved with general invariant sets on the horizon.
But we pay attention only to equilibria on the horizon in the present study.
} 
\KMa{in the same way as stable ones}.
\KMa{When we emphasize} structure of equilibria on the horizon, we shall call them as follows.
\begin{dfn}\rm
We say that a blow-up solution is {\em sink-type} (resp. {\em saddle-type}) if it is transformed into a trajectory on \KMa{$W^s_{\rm loc}(x_\ast; g)$} with a sink (resp. saddle) equilibrium $x_\ast$ on the horizon for the associated desingularized vector field $g$ \KMa{introduced in Section \ref{sec:pre_blow_up}}.
\end{dfn}
There are many studies of blow-up solutions through analytic arguments (e.g. \cite{FM2002, HV1997, W2013}) or numerical simulations (e.g. \cite{AIU2017, C2016, CHO2007, ZS2017}), many of which would be sink-type through related numerical simulations and computer-assisted proofs (e.g. \cite{Mat2018, MT2020_1, MT2020_2}).
Whereas, saddle-type blow-up solutions are quite difficult to calculate and to understand the role in global dynamics, \KMa{because} generic small perturbations of \KMa{initial points} (for (\ref{ODE-original})) break the structure.
\KMa{Even in the context of dynamics at infinity (i.e., without concerning the blow-up nature of solutions)}, there are very limited studies for characterizing \KMa{trajectories asymptotic to the horizon} themselves and their global nature, except special cases such as {\em planar} dynamical systems (e.g. \cite{DH1999, DLA2006}).
\KMa{On the other hand, saddle-type blow-up solutions themselves can exist in various types of differential equations, many of which do not concern with its sensitivity under perturbations of initial points, but are interested only in their existence and/or persistence of blow-up structure under perturbation of initial points is mentioned implicitly (cf. \cite{Ha2016, Ha2017, NZ2015} for complex-valued PDEs).}
\par
\KMa{
Here we will see that our validation methodology provide not only a systematic way to capture saddle-type blow-up solutions but also distributions of a collection of blow-up solutions in the phase space, both of which are with mathematical rigor.
Moreover, as seen in preceding works, methodologies with computer-assisted proofs provide {\em explicit} enclosures of computation objects.
This property enables us to visualize the distribution of solution profiles and blow-up times depending on initial points of blow-up solutions.
As a byproduct of the application of parameterization method reviewed in Section \ref{sec:parameterization_method}, we obtain an explicit formula of $t_{\max}$ as a function of initial \KMa{points} and its smoothness.
}


\subsection{Basic methodology}
\label{sec:method_basic}
First we discuss a basic methodology for validating \KMa{(locally defined) blow-up solutions and their extension}.
The fundamental steps consist of the following:
\begin{enumerate}
\item Validation of local stable manifolds of equilibria on the horizon for desingularized vector fields;
\item Extension of validated stable manifolds via rigorous integration of desingularized vector fields.
\end{enumerate}
These steps are shown to provide a collection of {\em divergent} solutions of (\ref{ODE-original}), according to Theorems \ref{thm:blowup-dir}, \ref{thm:blowup-Poincare}, and \ref{thm:blowup-parabolic} except the evaluation of $t_{\max}$.
When an equilibrium $p$ on the horizon is {\em stable}, the validation procedures reported in \cite{MT2020_1, MT2020_2, TMSTMO2017} allow (a) studying the local stable manifold of $p$ by means of {\em locally defined Lyapunov functions}; and (b) computing rigorous enclosure of \KMa{solutions} converging to $p$, hence yielding a rigorous bound of the blow-up time.
\KMa{
Although the same strategy or similar topological arguments such as {\em covering relations} (e.g. \cite{ZCov2009}) can work effectively, 
we apply the parameterization method to validating local stable manifolds of equilibria on the horizon for desingularized vector fields here instead.

}
\par
An important merit of the parameterization method is that {\em only the stable information of equilibria can be treated through the whole computations \KMa{involving invariant manifolds, no matter how unstable equilibria or general invariant sets are}}.
In other words, if we can compute stable eigenvectors at the equilibria and a topological conjugacy $P$ with high accuracy, we obtain the local stable manifold without containing intrinsic unstable information of equilibria\footnote{
\KMa{
In topological arguments such as local Lyapunov functions and covering relations, topological information of {\em both} stable and unstable directions around equilibria are necessary to validate locally defined invariant manifolds, which cause a big difference of treatments between stable and unstable invariant sets.
}
}.
\KMa{Moreover, we obtain the embedding of the parameterized invariant manifolds and hence the {\em distribution of locally defined invariant manifolds in the whole phase space can be captured at the same time}.
This distribution greatly helps us with investigating the behavior of trajectories far from invariant manifolds.}
\KMa{Universality} of such features \KMa{in the parameterization} is shown in many preceding works (e.g. \cite{MR4127962,BLM2016,MR3706909,MR3792792,MR2821596}) for obtaining global nature of dynamical systems.
\par
After validating the locally parameterized stable manifolds, \KMa{these manifolds can be extended through time-integrations of the {\em time-reversal} desingularized vector fields.} 
\KMa{In particular, we obtain {\em globalized} stable manifolds whose preimages under compactifications are (candidates of) families of saddle-type blow-up solutions\footnote{
Needless to say, the proposing methodology can be applied to sink-type blow-up solutions.
}}.
\KMa{Globalization of invariant manifolds enables us to investigate global nature of dynamical systems, including blow-up solutions in the present study, while the methodology itself is standard and essentially identical with the one used in preceding works (e.g. \cite{MT2020_1, MT2020_2, TMSTMO2017})}.

%
%

\begin{rem}[Rigorous integrators of ODEs]\label{rem:integrator}
Many methods for rigorously integrating solution trajectories of vector fields have been proposed over the last thirty years. The most famous achievement is the resolution of Smale's 14th problem by W.~Tucker \cite{Tucker2002}.
We refer to \cite{COSY, Bunger:2020aa, Immler:2018aa,Kashi1,MR3148084,Lohner,Zgli} for different methods for rigorous integration of ODEs.
These methods are based on fixed-point arguments, which is equivalent to show the existence of solution trajectories, and several techniques of interval arithmetic.
For the sake of forward time integration, we use a C++ Library for rigorous integration of ODEs, which is named the \emph{kv library} \cite{kv}. This integrator is based on an interval representation of the solutions' Taylor series and the Affine arithmetic \cite{Rump2015}, which is a technique for preventing the so-called wrapping effect in interval analysis. 
\end{rem}


\KMa{The remaining issue is \KMa{the finiteness} of $t_{\max}$ and its explicit enclosure to assure that our validated trajectories indeed \KMa{correspond} to blow-up solutions for the original system.}
The blow-up time $t_{\max}$ generally depends on initial \KMa{points of solutions}. 
We now introduce an explicit estimate methodology for obtaining blow-up times.
\par
\KMa{
The blow-up time $t_{\max}$ is defined by the improper integral as $\tau \to \infty$ in (\ref{time-desing-directional-integral}), (\ref{time-desing-Poin-integral}) or (\ref{time-desing-para-integral}), where $\tau$ is the corresponding time-scale.
The basic approach to enclose $t_{\max}$ is to divide the integral into two parts:
\begin{equation}
\label{tmax-basic}
t_{\max} = t_0 + \int_{\tau_0}^{\bar \tau} h(\tau) d\tau +  \int_{\bar \tau}^{\infty} h(\tau) d\tau \equiv t_0 + t_{\max, 1}(\bar \tau) + t_{\max, 2}(\bar \tau)
\end{equation}
for some $\bar \tau > \tau_0$, where $h$ is a functional representing integrands for characterizing $t_{\max}$ depending on solutions of the desingularized vector field $g$.
The key point of the successive treatments is to enclose $t_{\max, 2}(\bar \tau)$ by using the asymptotic information of trajectories, say the fact that trajectories of our interest are located on a local stable manifold $W^s_{\rm loc}(p; g)$ of a saddle equilibrium $p$.
Once the manifold $W^s_{\rm loc}(p; g)$ is constructed through the parameterization $P$, the functional $h$ is expressed by means of a (nonlinear) combination of $P$, and the enclosure of $t_{\max, 2}(\bar \tau)$ is also computed through $P$ itself or its enclosure.
As for $t_{\max, 1}(\bar \tau)$, we directly enclose the integral through the enclosed trajectories via ODE integrations.
\par
When we extend the local stable manifold, then integrate the vector field $g$ in the reverse-time direction and evaluate $t_{\max,1}(\bar \tau)$ by
\begin{equation*}
\label{blow-up-time-traj}
\int_{\tau_0}^{\bar \tau} h(\tau) d\tau = \int_{-\bar \tau}^{-\tau_0} h(\tilde \tau) d\tilde \tau\quad \KMa{\text{with }\tilde \tau = -\tau}.
\end{equation*}
The functional $h$ is given as follows, depending on the choice of compactifications:
\begin{description}
\item[Directional:] $\KMa{h}(\tau) = s(\tau)^k$.
\item[Poincar\'{e}-type:] $\KMa{h}(\tau) = (1-p(x(\tau))^{2c})^{k/2c} = \left( 1-\sum_{i=1}^n x_i(\tau)^{2\beta_i} \right)^{k/2c}$.
\item[Parabolic-type:] $\KMa{h}(\tau) = \left(1-\frac{2c-1}{2c}(1-p(x(\tau))^{2c})\right)(1-p(x(\tau))^{2c})^k$.
\end{description}
}

\begin{rem}
The absence of constant terms in the integrand of $t_{\max}$ is the most essential property to show that $t_{\max} < \infty$ in the preceding work \cite{Mat2018} when \KMa{equilibria on the horizon} are hyperbolic, where the Hartman-Grobman-type argument is applied to extracting the exponentially decaying property of the integrand.
This property is essentially independent of the choice of compactifications associated with appropriately chosen time-scale desingularizations.
The present argument explicitly extracts this property to verify $t_{\max} < \infty$ by means of the parameterization method.
\end{rem}

\par
Summarizing the above arguments, our methodology for validating (saddle-type) blow-up solutions consists of the following.
\begin{enumerate}
\item Validate the local stable manifold of an equilibrium on the horizon for desingularized vector fields via the parameterization method.
\item Extend the validated stable manifold via (backward) integration, that is done by considering the time-reversed desingularized vector fields.
\item Compute a rigorous enclosure of the blow-up time $t_{\max}$ through the decomposition of the form (\ref{tmax-basic}) as well as direct integrations \KMa{through trajectories and parameterizations}.
\end{enumerate}

\par
In the subsequent sections, applicability of the present methodology is shown.
In particular, we aim at showing the following features, respectively:
\begin{itemize}
\item {\bf Section \ref{sec:demo1}} shows an application of {\em directional compactifications} for validating saddle-type blow-up profiles and computing the validated curve $t_{\max}$ as a function of initial points.
\item {\bf Section \ref{sec:demo2}} shows an {\em application to higher-dimensional systems}. 
In the present study we consider an artificial $3$-dimensional system.
The Poincar\'{e}-type compactification is applied to an asymptotically homogeneous vector field.
This example shows the global phase portrait involving multiple saddle-type blow-up solutions.
\item {\bf Section \ref{sec:demo3}} shows a characteristic nature of saddle-type blow-up solutions with bounded global solutions which separate the whole phase space into \KMa{four} sets, one of which is the set of points such that solutions through them determine time-global solutions for both time directions and \KMa{the others are the sets} of points such that solutions through them are blow-up solutions \KMa{in positive and/or negative time directions}.
Dependence of $t_{\max}$ as a function of initial \KMa{points} including saddle-type blow-up solutions is also addressed.
The parabolic-type compactification is applied to an asymptotically quasi-homogeneous vector field.
\end{itemize}

\subsection{Smooth dependence of \boldmath$t_{\max}$\unboldmath~on initial points}
\label{sec:tmax-smoothness}
Explicit expressions of $t_{\max}$ shown in Section \ref{sec:method_basic} indicate that $t_{\max}$ depends continuously, possibly smoothly, on initial points within stable manifolds of hyperbolic equilibria (for desingularized vector fields) on the horizon, which is just a consequence of standard calculus.
One of \KMa{benefits} of applying the parameterization method \KMa{reviewed} in Section~\ref{sec:parameterization_method} is that $t_{\max}$ can be treated as a {\em locally analytic} function on initial points of solutions. 
Here we discuss the dependence of $t_{\max}$ on initial points in more details.
\par
Consider the desingularized vector field $g$ associated with the directional (resp. Poincar\'{e}-type and parabolic-type) compactification with the \KMa{associated} time-scale desingularization.
Let $p_\ast \in \mathcal{E}$ be a hyperbolic equilibrium for $g$.
First note that typical choices of time-scale desingularizations \KMa{$h$}, namely the integrand of $t_{\max}$ \KMa{mentioned in (\ref{tmax-basic})}, satisfy the following properties (so that trajectories for $g$ is orbitally equivalent to the original dynamical system (cf. \cite{Mat2019})):
\begin{itemize}
\item It vanishes at $p_\ast$. 
\item It is positive along $W^s_{\rm loc}(p_\ast; g)$.
\item It is smooth, in particular analytic, except $k/2c \not \in \mathbb{N}$ in the case of the Poincar\'{e}-type compactifications. 
\end{itemize}

Assume that all assumptions in Theorem \ref{thm:radii_polynomials} for $F$ given in (\ref{eq:parameterization_F(a)=0}) as well as (A1), (A2) and (A3) associated with the hyperbolic equilibrium $p_\ast$ for $g$ are satisfied, in which case the local stable manifold $W^s_{\rm loc}(p_\ast;g)$ is parameterized by an {\em analytic} function $P$ defined on the \KMa{$m$-dimensional} unit polydisc $B_1^m$ so that $W^s_{\rm loc}(p_\ast;g) = P(B_1^m)$.
For typical asymptotically quasi-homogeneous fields, the function \KMa{$h$} can be chosen as a polynomial or a rational function whose denominator is polynomial and positive on $W^s_{\rm loc}(p_\ast;g)$.
From these observations, we obtain the following proposition.

\begin{prop}[Analytic function through parameterization]
\label{prop-blow-up-time-analytic}
Let $p_\ast$ be an hyperbolic equilibrium for a dynamical system \KMa{generated by a vector field $g$ which is analytic in a neighborhood of $p_\ast$} satisfying (A1), (A2) and (A3).
Assume that a parameterization $P$ of $W^s_{\rm loc}(p_\ast)$ satisfying $P(0) = p_\ast$ is defined on the \KMa{$m$-dimensional} unit polydisc $B^m_1$, in particular $W^s_{\rm loc}(p_\ast) = P(B^m_1)$.
Let \KMa{$h$} be an analytic function defined in a neighborhood of $W^s_{\rm loc}(p_\ast)$ satisfying $\KMa{h}(p_\ast) = 0$.
Then the integral
\begin{equation}
\label{U-integral}
U(\theta) \bydef \int_0^\infty \KMa{h}\circ P (e^{\Lambda \tau} \theta) d\tau,\quad \theta\in B_1^m
\end{equation}
is an analytic function on $B^m_1$ satisfying $U(0)=0$.
\end{prop}
\begin{proof}
\KMa{Because} $\KMa{h}$ and $P$ are analytic, then so is $\KMa{h}\circ P$ and hence the integrand of $U$ is written by the convergent series
\begin{equation*}
\KMa{h}\circ P (e^{\Lambda \tau}\theta) = \sum_{|\alpha| \geq 0} c_{\alpha} \left( e^{\Lambda \tau} \theta \right)^\alpha.
\end{equation*}
Denoting $\alpha \cdot \lambda = \sum_{i=1}^m \alpha_i \lambda_i $,
the assumption $\KMa{h}\circ P(0) = 0$ implies that $c_{\bf 0} = 0$ and 
\begin{align*}
U(\theta) &= \int_0^\infty \sum_{|\alpha| > 0} c_{\alpha} \left( e^{\Lambda \tau} \theta \right)^\alpha d\tau = \int_0^\infty  \sum_{|\alpha| > 0} c_{\alpha} \theta^\alpha e^{(\alpha \cdot \lambda )\tau}  d\tau \\
	&= \sum_{|\alpha| > 0} c_{\alpha} \theta^\alpha \left( e^{(\alpha \cdot \lambda)\tau}  d\tau \right) \\
	&= -\sum_{|\alpha| > 0} \frac{c_{\alpha} \theta^\alpha }{\alpha \cdot \lambda},
\end{align*}
which converges uniformly on $B^m_1$.
Indeed, letting $\sigma_{\rm gap} \bydef \min_{j=1,\dots,m} |\lambda_j| > 0$, then
\begin{align*}
\left| U(\theta)_j \right| & = \left| \sum_{|\alpha| > 0} \frac{(c_j)_{\alpha} \theta^\alpha }{\alpha \cdot \lambda} \right| \leq \frac{1}{\sigma_{\rm gap}} \left| \sum_{|\alpha| > 0} (c_j)_{\alpha}\theta^\alpha \right|.
\end{align*}
holds for $j=1,\KMa{\ldots}, n$, uniformly in $B^m_1$.
\end{proof}

This proposition provides a fundamental feature of blow-up times.
\KMa{Combining with the choice of functionals $h$ mentioned in Section \ref{sec:method_basic}, the compositions of $h$ and the parameterization $P$ are given as follows for each compactification:
}
\begin{description}
\item[Directional:] $\KMa{h}(s(\tau), x(\tau)) = s(\tau)^k$ and $s(\tau) = P_1(e^{\Lambda \tau}\theta)$.
\item[Poincar\'{e}-type with $k/2c\in \mathbb{N}$:] $\KMa{h}(x(\tau)) = (1-p(x(\tau))^{2c})^{k/2c} = \left( 1-\sum_{i=1}^n x_i(\tau)^{2\beta_i} \right)^{k/2c}$ and $x(\tau) = P(e^{\Lambda \tau}\theta)$.
\item[Parabolic-type:] $\KMa{h}(x(\tau)) = \left(1-\frac{2c-1}{2c}(1-p(x(\tau))^{2c})\right)(1-p(x(\tau))^{2c})^k$ and $x(\tau) = P(e^{\Lambda \tau}\theta)$.
\end{description}
\KMa{The function $U(\theta)$ in (\ref{U-integral})} equals to $t_{\max} = t_{\max}(\theta)$ under corresponding compactifications and time-scale desingularizations.
Note that the function $\KMa{h}$ corresponds to the time-scale transformation factor. 
Different choice of time-scale desingularizations provides different $\KMa{h}$ and, consequently, different determination of $U(\theta) = t_{\max}(\theta)$. 
\par
The inverse $T^{-1}$ of compactifications {\em away from the horizon} can be described by analytic functions \KMa{because} it is defined by the $n$-tuples of composite functions of radicals and rational functions whose singularities in the sense of the loss of regularity and convergence of infinite series are located \KMa{on the horizon}.
The analytic dependence of $t_{\max}$ on bounded initial \KMa{points} for (\ref{ODE-original}) \KMa{near blow-up} is therefore inherited by restricting our attention to stable manifolds of equilibria \KMa{on the horizon for desingularized vector fields}.

\begin{thm}[Analyticity of blow-up times]
\label{thm-blow-up-time-analytic}
Let $t_{\rm max}$ be given by \KMa{the directional (resp. the Poincar\'{e}-type with $k/2c\in \mathbb{N}$, or the parabolic-type) compactification given by (\ref{time-desing-directional-integral}) (resp. (\ref{time-desing-Poin-integral}) and (\ref{time-desing-para-integral})).}
Let $W^s_{\rm loc}(p_\ast; g)$ be a local stable manifold \KMa{of a hyperbolic saddle $p_\ast$ on the horizon} for the desingularized vector field \KMa{$g = g_d$} \KMa{(resp. $g=g_{qP}$ and $g = g_{para}$)} given by the parameterization $P$ satisfying all requirements presented in Proposition \ref{prop-blow-up-time-analytic}.
Let $y_0\in \mathbb{R}^n$ be a point such that the solution $y(t)$ to (\ref{ODE-original}) with $y(\KMa{t_0})=y_0$ is mapped into the \KMa{trajectory $\{(T(y))(\tau)\}_{\tau \geq \tau_0}$} included in $W^s_{\rm loc}(p_\ast; g)$ through \KMa{$T = T_d$ (resp. $T = T_{qP}$ and $T = T_{para}$)} and the \KMa{corresponding} time-scale desingularization. 
\par
Then the blow-up time $t_{\max}$ is real analytic at $y_0$ in 
${\rm int}_{T^{-1}(\mathcal{D})}T^{-1}(W^s_{\rm loc}(p_\ast; g))$, where $\mathcal{D}$ is \KMa{given} by (\ref{D-directional}) \KMa{(resp. (\ref{D-global}) for Poincar\'{e}- and parabolic-types)}.
Moreover, $t_{\max} = t_{\max}(y_0)$ converges to $0$ as $y_0$ goes to infinity along the solution $y(t)$.
\end{thm}

\begin{proof}
Let $y_0 \in {\rm int}_{T^{-1}(\mathcal{D})}T^{-1}(W^s_{\rm loc}(p_\ast; g))$ be arbitrary.
Then there is a unique point $\theta_0 \in B^m_1$ such that $y_0 = T^{-1}(P(\theta_0))$.
We shall write $\theta_0 = P^{-1}(T(y_0))$, where the expression of $P^{-1}$ reflects the one-to-one property of $P$ on $B^m_1$. 
\KMa{Because} $T$ is analytic in \KMa{$\mathcal{D}$} (Remark \ref{rem-comp-anal}), then so is $P^{-1}\circ T$ in $ {\rm int}_{T^{-1}(\mathcal{D})}T^{-1}(W^s_{\rm loc}(p_\ast; g))$\footnote{
Analyticity of $P^{-1}$ follows from that of $P$ by assumption, linear isomorphism property of $DP$ and the inverse function theorem for analytic functions. 
See e.g. \cite{D1960} for the latter argument.
}.
Then the blow-up time $t_{\max} = t_{\max}(y_0)$ at $y_0$ is written by
\begin{equation*}
t_{\max} = t_{\max}(\theta_0) = \KMa{t_0 + } \int_{\KMa{\tau_0}}^\infty \KMa{h}\circ P (e^{\Lambda \tau} (P^{-1}(T(y_0)))) d\tau \equiv t_{\max}(y_0),
\end{equation*}
\KMa{where $h$ is a function mentioned just after the proof of Proposition \ref{prop-blow-up-time-analytic}.}
The integrand is analytic in ${\rm int}_{T^{-1}(\mathcal{D})}T^{-1}(W^s_{\rm loc}(p_\ast; g))$, according to the same argument as the proof of Proposition \ref{prop-blow-up-time-analytic}.
Therefore $t_{\max} = t_{\max}(y_0)$ is analytic at $y_0 \in {\rm int}_{T^{-1}(\mathcal{D})}T^{-1}(W^s_{\rm loc}(p_\ast; g))$.
\par
Our assumption for the solution $y(t)$ implies that the property of $y(t)$ going to infinity as $t\to t_{\max}$ corresponds to $(T(y))(\tau) \to p_\ast$ as $\tau \to \infty$.
Moreover, for any $y_0 \in {\rm int}_{T^{-1}(\mathcal{D})}T^{-1}(W^s_{\rm loc}(p_\ast; g))$, we can choose a point $\tilde y_0 \in \mathbb{R}^n$ such that $T(\tilde y_0) \in W^s_{\rm loc}(p_\ast; g)$ and that $y_0 = y(t) = y(t; \tilde y_0\KMa{, t_0})$ for some $t>\KMa{t_0}$, where $t$ is uniquely determined by $\tilde y_0$.
The last assertion is equivalent to $T(y_0) = (T(\tilde y_0))(\bar \tau)$ for $\bar \tau > \KMa{\tau_0}$ uniquely determined by $t$ and the time-scale desingularization.
Fix the point $\tilde y_0$.
\KMa{Consider the decomposition (\ref{tmax-basic}) of $t_{\max} = t_{\max}(\tilde y_0)$ with the integrand $S_h(\tau)\equiv h\circ P (e^{\Lambda \tau} (P^{-1}(T(y_0))))$}.
%
Notice that the second term in the right-hand side is the contribution of $y_0$ to the determination of $t_{\max}$. 
As as result, we have 
\begin{equation*}
t_{\max}(y_0) = \int_{\bar \tau}^\infty \KMa{S_h}(\tau) d\tau.
\end{equation*}
As mentioned, the convergence of $T(y_0)$ to $p_\ast$ corresponds to $\bar \tau \to \infty$.
\KMa{Because $S_h$} is analytic in $B_1^m$, the integral $t_{\max}(y_0)$ goes to $0$ as $\bar \tau \to \infty$.
This implies the final statement in the theorem.
\end{proof}

Theorem \ref{thm-blow-up-time-analytic} indicates that $t_{\max}$ depends analytically on initial points on $T^{-1}(W^s_{\rm loc}(p_\ast; g))$, provided that the non-resonance condition holds for eigenvalues of $Dg(p_\ast)$.
Furthermore, a computer-assisted proof for the existence of $P$ as discussed in Section \ref{sec:parameterization_method} provides the {\em explicit} region where the analyticity of $t_{\max}$ as a function of initial points of trajectories is guaranteed.
Extending $W^s_{\rm loc}(p_\ast; g)$ through the flow and using the smooth dependence of the flow on initial points, we can extend $t_{\max}$ as a smooth function of the initial \KMa{points} whose smoothness depends on that for the flow, as long as $W^s_{\rm loc}(p_\ast; g)$ is smoothly continued.
In particular, $t_{\max}$ can be analytically continued if the vector field $g$ is analytic.
Note that the analyticity, or even continuity of $t_{\max}$ is not guaranteed {\em as a function of $y$ in $\mathbb{R}^n$} \KMa{because} the expression of $t_{\max}$ as an analytic function only makes sense on $W_{\rm loc}^s(p_\ast;g)$.
The different choice of $p_\ast$ induce a different expression of $P$, and hence of $t_{\max}$.
\par
\bigskip
\KMa{In the end of this section,} we shall derive a detailed implementation of $t_{\max}$ for directional compactifications, namely \KMa{an} estimate of the integral $\int_{\bar \tau}^\infty s(\tau)^k d\tau$ given in (\ref{time-desing-directional-integral}).
The corresponding calculations of $t_{\max}$ for Poincar\'{e}-type with $k/2c \in \mathbb{N}$ and parabolic-type compactifications are achieved in the same manner.
Let $p_\ast \in \mathcal{E}$ be a hyperbolic equilibrium for \KMa{the} desingularized vector field $g = g_d$.
Assume that the parameterization method around $p_\ast$ works and the local stable manifold $W^s_{\rm loc}(p_\ast; g)$ is obtained through the ($m$-dimensional) stable polydisk $B_1^m$ and the parameterization $P$.
For simplicity, stable eigenvalues $\{\lambda_i\}_{i=1}^m$ of the linearized matrix of the desingularized vector field at $p_\ast$ are assumed to be simple and real.
In particular, $\lambda_i < 0$ for $i=1,\KMa{\ldots}, m$.
Recalling \eqref{eq:diagonalL}, write
\begin{equation*}
\Lambda = \begin{pmatrix}
\lambda_1 & & \\
& \ddots & \\
& & \lambda_m
\end{pmatrix}.
\end{equation*}
Then the solution $(s(\tau), \hat x(\tau)) \in W^s_{\rm loc}(p_\ast; g)$ is written by 
\begin{equation*}
(s(\tau), \hat x(\tau)) = P\left( e^{\Lambda \tau}\theta\right)\text{ with }s(\tau) = P_1\left( e^{\Lambda \tau}\theta\right),\quad \theta = (\theta_1,\KMa{\ldots}, \theta_m) \in B_1^m, 
\end{equation*}
where 
\begin{equation}
\label{param_expression}
P(\theta) = \sum_{|\alpha| \geq 0}a_{\alpha}\theta^{\alpha} \equiv \begin{pmatrix}
P_1(\theta) \\ \vdots \\ P_n(\theta)
\end{pmatrix}\in \mathbb{R}^n,\quad \theta = \begin{pmatrix}
\theta_1 \\ \vdots \\ \theta_m
\end{pmatrix}\in \mathbb{R}^{m},\quad 
a_{\alpha} = \begin{pmatrix}
(a_1)_\alpha \\ \vdots \\ (a_n)_\alpha
\end{pmatrix}\in \mathbb{R}^n,
\end{equation}
is the parameterization of $W_{\rm loc}^s(p_\ast; g)$.
The rightmost integral in (\ref{time-desing-directional-integral}) can be calculated as follows, once we obtain a concrete form of $P$:
\begin{align*}
\int_{\bar \tau}^\infty s(\tau)^k d\tau &= \int_{\bar \tau}^\infty \left( P_1\left( e^{\Lambda (\tau - \bar \tau)}\theta\right) \right)^k d\tau \equiv \int_0^\infty \left( P_1\left( e^{\Lambda \tilde \tau}\theta\right) \right)^k d\tilde \tau \\
 &= \int_0^\infty \left(  \sum_{|\alpha| \geq 0} (a_1)_\alpha \left( e^{\Lambda \tilde \tau}\theta \right)^{\alpha} \right)^k d\tilde \tau\\
	&= \int_0^\infty \left( \sum_{|\alpha| \geq 0} (a_1)_\alpha e^{\left(\alpha \cdot \lambda \right)\tilde \tau}\theta^{\alpha} \right)^k d\tilde \tau.
\end{align*}
Denote the {\em Cauchy product over multi-indices} by
\begin{equation}
\label{conv-multi}
(a\ast b)_\alpha = \sum_{\beta + \gamma = \alpha}a_{\beta} b_{\gamma},\quad a_\beta, b_\gamma \in \mathbb{R}\quad \text{ for }\quad \alpha, \beta, \gamma\in \mathbb{Z}_{\geq 0}^m,
\end{equation}
and given $k \in \N$ denote 
\[
\left(a^k \right)_\alpha = (\stackrel{k \text{ times}}{\overbrace{a \ast \cdots \ast a}} )_\alpha.
\]
Here we observe that $(a_1)_{\bf 0} = 0$, since $P(0) = p_\ast$ is the equilibrium for the desingularized vector field (\ref{ODE-desing-directional}) and $P_1(0) = 0$ from our choice of compactifications.
Using the previous notation and the above fact, the above integral is formally written as follows:
%
\begin{align}
\notag
\int_{\bar \tau}^\infty s(\tau)^k d\tau &= \int_0^\infty \left\{ \sum_{|\alpha| \geq 0} (a_1^k)_\alpha  e^{\left(\alpha \cdot \lambda \right)\tilde \tau}\theta^{\alpha} \right\} d\tilde \tau\\
\notag
	&=  \sum_{|\alpha| \geq 0} (a_1^k)_\alpha  \theta^{\alpha} \left( \int_0^\infty  e^{\left(\alpha \cdot \lambda \right)\tilde \tau} d\tilde \tau \right)\\
\notag
	&=  \sum_{|\alpha| > 0} (a_1^k)_\alpha  \theta^{\alpha} \left( \int_0^\infty  e^{\left(\alpha \cdot \lambda \right)\tilde \tau} d\tilde \tau \right) \\
\label{blow-up-time-dir-asym}
	&= -\sum_{|\alpha| > 0} (a_1^k)_\alpha  \frac{\theta^{\alpha}}{\alpha \cdot \lambda}.
\end{align}
In particular, the denominator $\alpha \cdot \lambda$ is strictly negative for all possible $\alpha$, and the analyticity of $P$ implies that the above infinite sum is convergent uniformly in $B^m_1$.
\par
The final formula (\ref{blow-up-time-dir-asym}) implies that we can calculate the rigorous value of $t_{\max}$ near blow-up, once we obtain the parameterization of the local stable manifold $W_{\rm loc}^s(p_\ast; g)$ and fix the point $\theta \in B^m$, namely $P(\theta) \in W_{\rm loc}^s(p_\ast; g)$.
As seen below, the similar expressions of $t_{\max}$ to (\ref{blow-up-time-dir-asym}) can be obtained for Poincar\'{e}-type and parabolic-type compactifications.

\begin{rem}[Special case]
If $n_s = 1$, the explicit expression (\ref{blow-up-time-dir-asym}) admits the simpler form:
\begin{align*}
(a\ast b)_n = \sum_{j \ge 0} a_j b_{n-j}, \quad a = (a_j)_{j \ge 0}, b = (b_j)_{j \ge 0}.
\end{align*}
Indeed, $\alpha$ becomes a single index $l$ and
\begin{align*}
t_{\max} &= 
	 \sum_{|\alpha| > 0} (a_1^k)_\alpha  \theta^{\alpha} \left( \int_0^\infty  e^{\left(\alpha \cdot \lambda \right)\tilde \tau} d\tilde \tau \right) = -\frac{1}{\lambda} \sum_{l = k}^{\infty} (a_1^k)_l \frac{\theta^{l}}{l},
\end{align*}
where we have used the fact that $(a_1)_0 = 0$ and that the Cauchy product 
$(a_1^k)_l  = (\stackrel{k \text{ times}}{\overbrace{a_1 \ast \cdots \ast a_1}} )_l$, with $l < k$ contains at least one $(a_1)_0$.
\end{rem}

\begin{rem}[Integrands, and smoothness of $t_{\max}$]
\label{rem:smoothness}
The concrete procedure to compute the integral (\ref{blow-up-time-dir-asym}) or its upper bound depends on problems, namely the choice of compactifications and time-scale desingularizations.
\begin{itemize}
\item Our first example (Section \ref{sec:demo1}) applies a directional compactification, while the time-scale desingularization has the different form from \eqref{time-desing-directional} so that the resulting desingularized vector field is polynomial.
Instead, $t_{\max}$ requires integrations of {\em rational}-type functions.
Nevertheless, the essence of the above argument, namely {\em the absence of constant terms in the integrand of $t_{\max}$}, can be applied to verifying that $t_{\max} < \infty$.
Analyticity of \KMa{the} integrand follows from that for both the numerator and the denominator with additional boundedness property of the denominator.
Detailed derivation of $t_{\max}$ or its upper bound is shown in subsequent sections.
\item In the case of Poincar\'{e}-type compactifications, analyticity of $t_{\max}$ is not guaranteed when $k/2c\not \in \mathbb{N}$, \KMa{because} the function $h(x) = x^{k/2c}$ is not analytic at $x=0$.
This failure comes from the \lq\lq mismatch" of properties of vector fields in the sense that the order $k+1$ and the type $\alpha$, consequently the natural number $c$, determining an appropriate Poincar\'{e}-type compactifications are determined by the asymptotic quasi-homogeneity of vector fields.
We then need further estimates for calculating $t_{\max}$ in such a case. 
The difficulty originated from this issue can be overcome by choosing the parabolic-type compactifications.
\end{itemize}
\end{rem}

\begin{rem}[Lyapunov functions versus parameterizations for expressing $t_{\max}$]
\KMa{In the preceding studies (e.g. \cite{MT2020_1, MT2020_2, TMSTMO2017}), $t_{\max}$ in all examples there are enclosed by means of Lyapunov functions. 
Local Lyapunov functions only provide} upper bounds of $t_{\max}$, \KMa{because} they do not trace concrete trajectories on stable manifolds, but values of functionals on trajectories, implying that smoothness arguments for $t_{\max}$ as a function of initial points cannot be derived.
Instead, simple inequalities by means of Lyapunov functions provide upper bounds of $t_{\max}$ even in the case of Poincar\'{e}-type compactifications with $k/2c\not \in \mathbb{N}$, as demonstrated in \cite{TMSTMO2017}.
Moreover, non-resonance condition (A3) is not required for estimations.
\par
On the other hand, we can trace trajectories on stable manifolds by means of parameterizations, indicating that $t_{\max}$ is \lq\lq exactly" calculated through the integration of given functions depending on \KMa{solutions}.
In particular, we can explicitly discuss properties of $t_{\max}$ as a functions of initial points.
In compensation for these precise information, however, we have to take care of analytic information of dynamical systems to ensure smoothness or analyticity of functions of interests, such as non-resonance condition (A3) for analyticity of $P$ providing the conjugacy to linearizations, matching of integers $k$ and $c$ for Poincar\'{e}-type compactifications mentioned in Remark \ref{rem:smoothness}.
\end{rem}

\section{Example 1: validation and visualization of globally extended saddle-type blow-ups}

\label{sec:demo1}
In what follows, we show several applications of our proposed methodology not only to show its applicability but also to reveal several remarkable features of saddle-type blow-up solutions.
The first problem is concerned with saddle-type blow-up solutions for the following system:
\begin{equation}
\label{two-fluid}
\begin{cases}
\beta' = vB_1(\beta) - c\beta - c_1, & \\
v' =  v^2 B_2(\beta) - cv - c_2, & 
\end{cases}\quad {}'=\frac{d}{d\zeta},
\end{equation}
where
\begin{equation*}
B_1(\beta) = \frac{(\beta-\rho_1)(\beta-\rho_2)}{\beta},\quad B_2(\beta) = \frac{\beta^2- \rho_1\rho_2}{2\beta^2}
\end{equation*}
and $\rho_2 > \rho_1$ are positive constants.
Moreover, 
\begin{equation}
\label{speed-two-phase}
c = \frac{v_R B_1(\beta_R) - v_L B_1(\beta_L)}{\beta_R - \beta_L}
\end{equation}
and $(c_1,c_2) = (c_{1L}, c_{2L})$ with 
\begin{equation}
\label{constants-two-phase}
\begin{cases}
c_{1L} = v_L B_1(\beta_L) - c\beta_L, & \\
c_{2L} = v_L^2 B_2(\beta_L) -cv_L. & \\
\end{cases}
\end{equation}
Points $(\beta_L, v_L)$ and $(\beta_R, v_R)$ are given in advance.

\begin{rem}
The system (\ref{two-fluid}) stems from the Riemann problem of the following system of conservation laws describing the (simplified) two-phase, one-dimensional imcompressible flow \cite{KSS2003}:
\begin{equation}
\label{two-fluid-PDE}
\beta_t + (vB_1(\beta))_x = 0,\quad
v_t + (v^2 B_2(\beta))_x = 0
\end{equation}
with
\begin{equation}
\label{data-two-fluid-PDE}
(\beta(x,0), v(x,0)) = 
\begin{cases}
U_L \equiv (\beta_L, v_L) &\text{$x<0$},\\
U_R \equiv (\beta_R, v_R) &\text{$x>0$}.
\end{cases}
\end{equation}
Observe that $B_1(\beta) < 0$ for $\beta\in (\rho_1, \rho_2)$ and $B_1(\beta) > 0$ for $0< \beta < \rho_1, \beta > \rho_2$. 
Details are stated in \cite{KSS2003}. 
\par
The system (\ref{two-fluid}) is the reduced problem of (\ref{two-fluid-PDE}) satisfying {\em viscosity profile criterion}, namely the traveling wave problem with respect to the frame coordinate $\zeta = x-ct$ with the boundary condition
\begin{equation*}
\lim_{\zeta \to -\infty}(\beta(\zeta), v(\zeta)) = (\beta_L, v_L),\quad \lim_{\zeta \to +\infty}(\beta(\zeta), v(\zeta)) = (\beta_R, v_R),
\end{equation*}
where $c$ is the speed of traveling waves.
Saddle-type blow-up solutions for (\ref{two-fluid}) are considered as components of {\em singular shock wave} solutions\footnote{
To make the correspondence precisely, the extended fast-slow system setting is required.
Detail is shown in \cite{KSS2003}.
} to (\ref{two-fluid-PDE}).
\end{rem}
We choose the directional compactification (\ref{dir-cpt}) of type $(0,1)$ : $(\beta, v)\mapsto (x_1, s) = (\beta, v^{-1})$ (cf. \cite{KSS2003, Mat2018}).
Direct calculations yield the following desingularized vector field on $\{r\geq 0\}\times \{\rho_1\leq \beta \leq \rho_2\}$:
\begin{equation}
\label{two-fluid-desing}
\begin{cases}
\displaystyle{\frac{dx_1}{d\tau} = B_1(x_1) - cx_1 s - c_1s}, & \\
\displaystyle{\frac{ds}{d\tau} =  -s\left\{B_2(x_1) - cs - c_2s^2\right\}}, & 
\end{cases}
\end{equation}
where $\tau$ is the desingularized time-scale given by $d\tau = s^{-1}dt$.
Obviously, $(x_1, s) = (\rho_1,0)\equiv p_1$ and $(\rho_2, 0)\equiv p_2$ are equilibria of (\ref{two-fluid-desing}) on the horizon $\mathcal{E} = \{s=0\}$ and the vector field on $\mathcal{E}\setminus \{p_1,p_2\}$ is monotone on each component.
\par
On the other hand, the vector field (\ref{two-fluid-desing}) is {\em rational}.
In order to nicely apply the parameterization method, we introduce {\em further} time-scale transformation as follows:
\begin{equation*}
\frac{d\tau}{d\eta} = x_1^{-2}.
\end{equation*}
Then the resulting vector field is
\begin{equation}
\label{two-fluid-desing-poly}
\begin{cases}
\displaystyle{\frac{dx_1}{d\eta} = x_1(x_1-\rho_1)(x_1-\rho_2) - cx_1^3 s - c_1x_1^2 s}, & \\
\displaystyle{\frac{ds}{d\eta} =  -s\left\{ \frac{1}{2}(x_1^2 - \rho_1\rho_2) - x_1^2 (c s + c_2 s^2)\right\}}. & 
\end{cases}
\end{equation}
Note that typical solutions of (\ref{two-fluid-desing}) are considered within the region $\{\rho_1 \leq x_1\leq \rho_2\}$ and $\rho_1 > 0$.
Therefore the new vector field (\ref{two-fluid-desing-poly}) is intrinsically the time-reparameterized vector field of (\ref{two-fluid-desing}) and hence these vector fields provide topologically the same information as each other.
\par
The horizon is $\{s=0\}$ and equilibria on the horizon is $(x_1, s) = (\rho_1, 0), (\rho_2, 0)$.
Looking at (\ref{two-fluid-desing-poly}) {\em only}, $(x_1, s) = (0, 0)$ can be also a stationary point, but it is not appropriate from our requirement.

\begin{rem}[Technical details]
\label{rem-tech-two-phase}
When we solve the problem (\ref{two-fluid-desing-poly}) in practice, we need to fix several parameters.
In the present case,
\begin{itemize}
\item First, we fix $x_L \equiv (x_{1,L}, s_L) = (1.9, 0.25)$ as a sample data.
Then, following the directional compactification $(x_1, s) = (\beta, v^{-1})$, we obtain $(\beta_L, v_L) = (1.9, 4)$.
Next, we fix $x_R \equiv (x_{1,R}, s_R) = (1.5, 0.2)$ similarly.
Then we obtain $(\beta_R, v_R) = (1.5, 5)$.
Independently, we need to fix $(\rho_1, \rho_2)$. 
In the present case, we fix $(\rho_1, \rho_2) = (1,2)$.
\item Following standard arguments of systems of conservation laws, compute  $B_1(\beta), B_2(\beta)$ and $c$ given above for $(\beta, v) = (\beta_L, v_L), (\beta_R, v_R)$.
\end{itemize}
\end{rem}


In the present study, we compute the stable manifold of the saddle equilibrium on the horizon $(x_1, s) = (2,0)$ for (\ref{two-fluid-desing-poly}) in $\{s\geq 0\}$ with parameters shown in Remark \ref{rem-tech-two-phase}.

\subsection{A local one-dimensional stable manifold of $p_2$ in \eqref{two-fluid-desing-poly}}

Consider the system of desingularized ODEs 
\begin{equation} \label{eq:ODE_1d_manifold}
\dot{x} = g(x) = \begin{pmatrix} g_1(x_1,x_2) \\  g_2(x_1,x_2) \end{pmatrix} \bydef 
\begin{pmatrix}
{\displaystyle
x_1^3 - (\rho_1+\rho_2) x_1^2 + \rho_1 \rho_2 x_1 - cx_1^3 x_2 - c_1 x_1^2 x_2
}
\vspace{.1cm}
\\
{\displaystyle
-\frac{1}{2}x_1^2x_2  + \frac{1}{2} \rho_1\rho_2 x_2 + c x_1^2 x_2^2 + c_2 x_1^2 x_2^3
}
\end{pmatrix},
\end{equation}
\KMa{which is exactly (\ref{two-fluid-desing-poly}) by replacing $(x_1,s)$ with $(x_1, x_2)$.
The dot $\dot {}$ denotes $d /d\eta$.}
Furthermore, at $x^{(2)} \bydef (\rho_2,0)$
\[
Dg(x^{(2)}) = 
\begin{pmatrix} 
\rho_2(\rho_2-\rho_1) & -\rho_2^2(c \rho_2+c_1) \\ 0 & -\frac{\rho_2}{2}(\rho_2-\rho_1)
\end{pmatrix}.
\]
We focus on the one-dimensional stable manifold of the steady state $x^{(2)}$ with stable eigenvalue $\lambda \bydef -\frac{\rho_2}{2}(\rho_2-\rho_1)<0$ 
and corresponding stable eigenvector 
\[
v \bydef \begin{pmatrix} -\rho_2^2(c \rho_2+c_1) \\  -\frac{3 \rho_2}{2}(\rho_2-\rho_1) \end{pmatrix}.
\]
Our goal is to produce an analytic function $P\colon (-\nu,\nu)\to \R^2$ that parameterizes $W_{loc}^s(x^{(2)})$.
The Taylor series representation has the form
\[
P(\theta) = \sum_{n=0}^\infty a_n \theta^n \quad \text{where}\ a_n = \begin{pmatrix} (a_1)_n \\ (a_2)_n \end{pmatrix}.
\]
By Lemma~\ref{lem:parmLemma} $P$ will represent the stable manifold if 
\[
P(0) = \begin{pmatrix} \rho_2 \\ 0 \end{pmatrix},\quad DP(0) = v = \begin{pmatrix} -\rho_2^2(c \rho_2+c_1) \\  -\frac{3 \rho_2}{2}(\rho_2-\rho_1) \end{pmatrix}
, \quad\text{and}\quad \lambda \theta \frac{\partial P}{\partial \theta} (\theta) = g(P(\theta)).
\]
From this we can immediately conclude that
\begin{equation}
\label{eq:manifold:a0a1}
\begin{pmatrix} (a_1)_0 \\ (a_1)_0 \end{pmatrix} = \begin{pmatrix} \rho_2 \\ 0 \end{pmatrix},\quad
\begin{pmatrix} (a_1)_1 \\ (a_2)_1 \end{pmatrix} = \begin{pmatrix} -\rho_2^2(c \rho_2+c_1) \\ -\frac{3 \rho_2}{2}(\rho_2-\rho_1) \end{pmatrix},\quad\text{and}\quad
\lambda  \sum_{n=0}^\infty n a_n \theta^n = g \left( \sum_{n=0}^\infty a_n \theta^n \right),
\end{equation}
where 
\[
g \left( \sum_{n=0}^\infty a_n \theta^n \right) = 
\sum_{n=0}^\infty \begin{pmatrix} 
{
\displaystyle
(a_1^3)_n-(\rho_1+\rho_2)(a_1^2)_n + \rho_1 \rho_2 (a_1)_n - c (a_1^3a_2)_n - c_1 (a_1^2a_2)_n
}
\vspace{.2cm}
\\ 
{
\displaystyle
-\frac{1}{2} (a_1^2a_2)_n + \frac{1}{2} \rho_1\rho_2 (a_2)_n + c (a_1^2a_2^2) + c_2 (a_1^2a_2^3)
}
\end{pmatrix} \theta^n.
\]

Let 
\[
\ell^1 \bydef \left\{ b = (b_n)_{n \ge 2} ~:~ \|b\|_1 \bydef \sum_{n \ge 2} |b_n|  < \infty \right\}. 
\]
For $j=1,2$, denote $a_j=((a_j)_n)_{n \ge 2}$, and $a = (a_1,a_2)$.
Define $F= (F_1,F_2)\colon (\ell^1)^2 \to (\ell^1)^2$ by
\begin{align}
\label{eq:F1_1d_stable_manifold}
(F_1(a))_n &\bydef \lambda n (a_1)_n - \left( (a_1^3)_n-(\rho_1+\rho_2)(a_1^2)_n + \rho_1 \rho_2 (a_1)_n - c (a_1^3a_2)_n - c_1 (a_1^2a_2)_n \right) 
\\
\label{eq:F2_1d_stable_manifold}
(F_2(a))_n &\bydef \lambda n (a_2)_n - \left( -\frac{1}{2} (a_1^2a_2)_n + \frac{1}{2} \rho_1\rho_2 (a_2)_n + c (a_1^2a_2^2) + c_2 (a_1^2a_2^3) \right) 
\end{align}
for $n \ge 2$, and observe that if there exists $\ta \in (\ell^1)^2$ such that $F(\ta)=0$, then we have obtained the desired parameterization.

\subsubsection{A computer-assisted proof}

Fixing $N=300$, we computed the bounds $Y_0$, $Z_0$, $Z_1$ and $Z_2$ as presented in Sections~\ref{sec:Y0},~\ref{sec:Z0},~\ref{sec:Z1}~and ~\ref{sec:Z2}, respectively. Then, we applied Theorem~\ref{thm:radii_polynomials} to prove existence of $\ta \in B_r(\ba)$ such that $F_1(\ta)=F_2(\ta)=0$ with $F_1$ and $F_2$ given in \eqref{eq:F1_1d_stable_manifold} and \eqref{eq:F2_1d_stable_manifold}, respectively. More explicitly, we got that $\| \ta - \ba \|_X \le r = 4.2 \times 10^{-13}$. 

The Taylor series representation of the parameterization of the local stable manifold has the form
\[
P(\theta) = \sum_{n=0}^\infty \ta_n \theta^n \quad \text{where}\ \ta_n = \begin{pmatrix} (\ta_1)_n \\ (\ta_2)_n \end{pmatrix}
\]
and denote by
\[
P^{(N)}(\theta) = \sum_{n=0}^N \ba_n \theta^n \quad \text{where}\ \ba_n = \begin{pmatrix} (\ba_1)_n \\ (\ba_2)_n \end{pmatrix}
\]
the numerical approximation of the local stable manifold. Then, 
\begin{align*}
\| P - P^{(N)} \|_\infty &= \sup_{\theta \in (-\nu,\nu)} \| P(\theta) - P^{(N)}(\theta) \|_\infty
\\
& = \sup_{\theta \in (-\nu,\nu)} \max \left( | P_1(\theta) - P_1^{(N)}(\theta) |,| P_2(\theta) - P_2^{(N)}(\theta) | \right)
\\
& \le \sup_{\theta \in (-\nu,\nu)} \max \left( \sum_{n=0}^\infty | (\ta_1)_n - (\ba_1)_n | |\theta|^n ,\sum_{n=0}^\infty | (\ta_2)_n - (\ba_2)_n | |\theta|^n \right)
\\
& \le \max \left( \sum_{n=0}^\infty | (\ta_1)_n - (\ba_1)_n | \nu^n ,\sum_{n=0}^\infty | (\ta_2)_n - (\ba_2)_n | \nu^n \right)
\\
& = \max \left( \| \ta_1 - \ba_1 \|_\nu,\| \ta_2 - \ba_2 \|_\nu \right) 
\\
& = \| \ta - \ba \|_X \le r = 4.2 \times 10^{-13}.
\end{align*}

\begin{figure}[htbp]\em
\centering
\includegraphics[width=7.0cm]{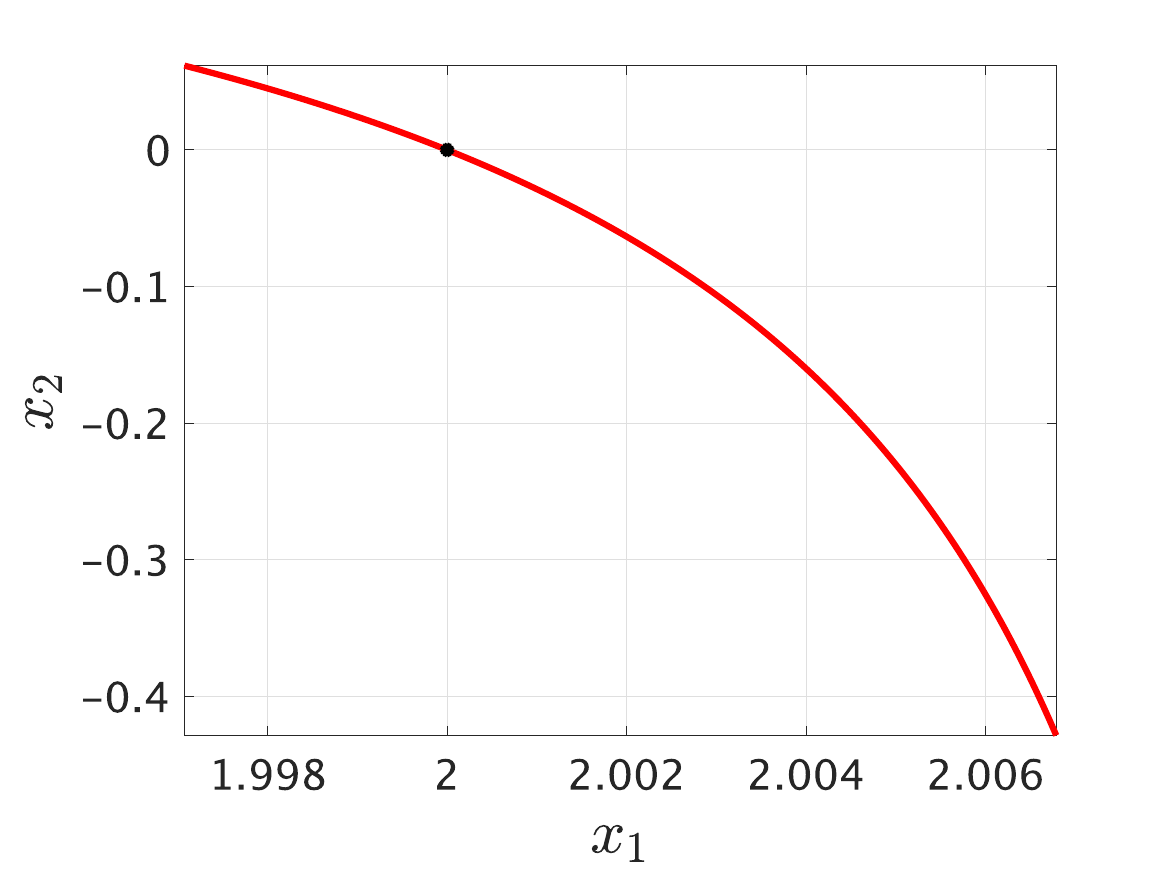}
\caption{The rigorously computed stable manifold with rigorous error bound $\| P - P^{(N)} \|_\infty \le r = 4.2 \times 10^{-13}$.} 
\flushleft
\KMa{Note that the plotted local stable manifold is defined for the desingularized vector field (\ref{eq:ODE_1d_manifold}), which itself makes sense for both positive and negative $x_2$. 
On the other hand, this makes sense only in $\{x_2 > 0\}$ as the corresponding object to the original vector field (\ref{two-fluid}), while the horizon $\{x_2 = 0\}$ corresponds to the infinity in the original $(\beta, v)$-phase space.}
\label{fig-1d-stable-manifold-ex1}
\end{figure}

\subsubsection{Computing the blow-up time}

Given a point $(x_1(0),s(0)) \in W_{loc}^s(p_2)$ (with $p_2=(2,0)$), the blow-up time is given by
\begin{equation}
\label{blow-up-time-demo1}
t_{\max} = \int_{0}^\infty \frac{s(\eta)}{x_1(\eta)^2} ~ d \eta.
\end{equation}
Given that $(x_1(0),s(0)) = (P_1(\theta),P_2(\theta))$ for a given $\theta \in (-\nu,\nu)$, we get from \eqref{eq:conjRelation} that 
$\varphi \left( t, P(\theta) \right) = P\left( e^{\lambda t} \theta \right)$ for all $t \ge 0$. 
Hence, the solution $(x_1(t),s(t))$ with \KMa{the initial point} $(x_1(0),s(0))= (P_1(\theta),P_2(\theta))$ is given by 
$(x_1(t),s(t)) = P\left( e^{\lambda t} \theta \right)$.

Rescaling the time interval $\eta \in [0,\infty]$ to $u \in [\theta,0]$ leads (via the change of coordinates $u = e^{\lambda \eta} \theta$) to
\begin{equation} \label{eq:t_max_v1}
t_{\max} = \int_{0}^\infty \frac{s(\eta)}{x_1(\eta)^2} ~ d \eta =
\int_{0}^\infty \frac{P_2 \left( e^{\lambda \eta} \theta \right)}{[P_1 \left( e^{\lambda \eta} \theta \right)]^2} ~ d \eta = 
\int_\theta^0 \frac{1}{\lambda u} \frac{P_2(u)}{[P_1(u)]^2}~du.
\end{equation} 
Now, note that
\[
P_2(u) = \sum_{n\ge 0} (\ta_2)_n u^n = \sum_{n\ge 1} (\ta_2)_n u^n
\]
since $(\ta_2)_0 = (p_2)_2 = 0$. Denote 
\[
Q(u) \bydef \frac{P_2(u)}{u} = \frac{1}{u} \sum_{n\ge 1} (\ta_2)_n u^n = \sum_{n\ge 0} (\ta_2)_{n+1} u^n = \sum_{n\ge 0} \tilde q_n u^n 
\]
where $\tilde  q_n \bydef (\ta_2)_{n+1}$ for $n \ge 0$. Hence, equation \eqref{eq:t_max_v1} becomes
\begin{equation*} 
t_{\max} = \frac{1}{\lambda} \int_\theta^0  \frac{Q(u)}{[P_1(u)]^2}~du.
\end{equation*} 
Assume now that we have (again using rigorous numerics) obtained 
\[
R(u) \bydef \frac{Q(u)}{[P_1(u)]^2} = \sum_{n\ge 0} r_n u^n
\]
with rigorous error bounds. Using that information, 
\begin{equation} \label{eq:t_max_v2}
t_{\max} = \frac{1}{\lambda} \int_\theta^0 \sum_{n\ge 0} r_n u^n~du = 
\frac{1}{\lambda} \sum_{n\ge 0} r_n \int_\theta^0  u^n~du = 
-\frac{1}{\lambda} \sum_{n\ge 0} \frac{r_n}{n+1} \theta^{n+1},
\end{equation} 
which is in essence computable (that is we can provide a numerical approximation together with rigorous error bounds). In the Figure~\ref{fig:t_max} below, we present a rigorous numerical computation (with rigorous bounds) of the value of $t_{\max}$ as a function of $\theta$, that is as a function of the initial \KMa{points} $P(\theta)$ on $W^s_{loc}(p_2)$. The rigorous error bound is obtained by computing rigorously the Taylor coefficients of $r_n$ in the expansion \eqref{eq:t_max_v2}. We present how to do that next. 

\begin{figure}[htbp]\em
\centering
\includegraphics[width=7.0cm]{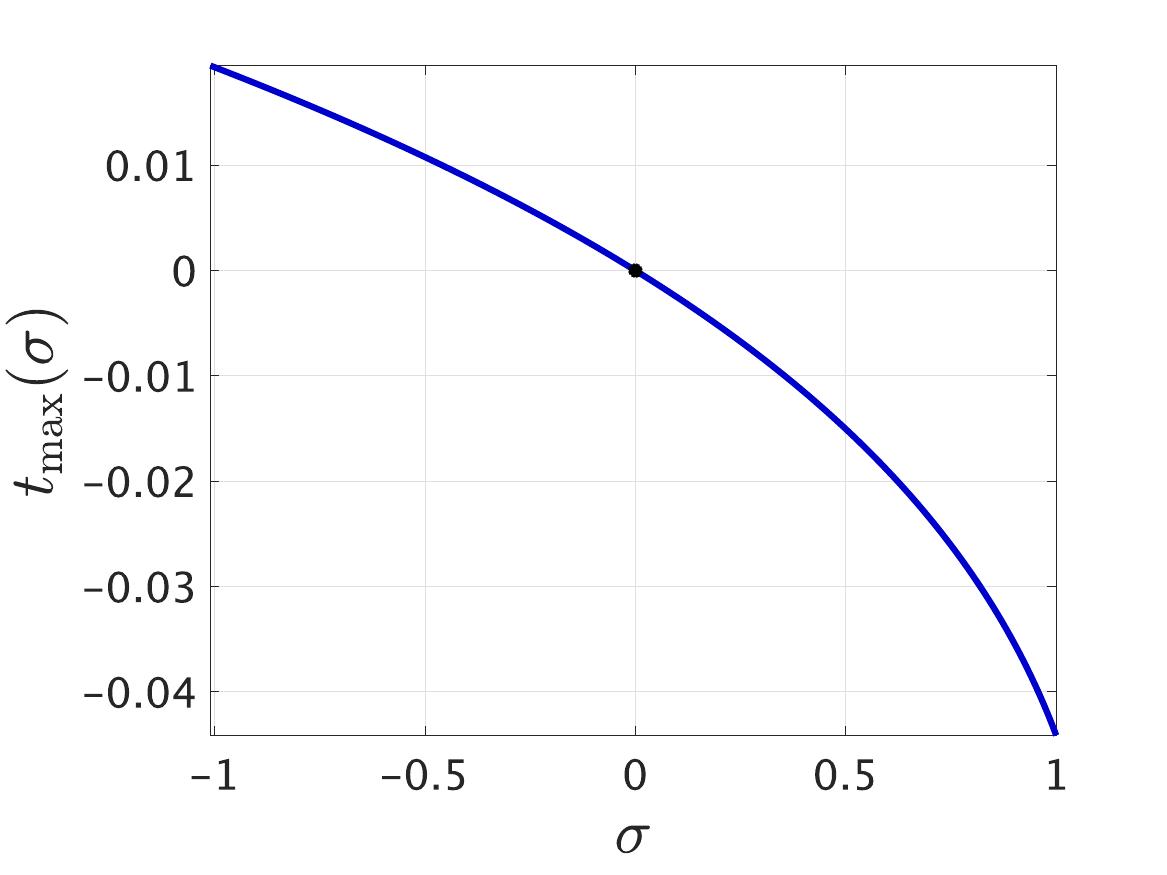}
\caption{The numerical values of $t_{\max}$ according to formula \eqref{eq:t_max_v2}.}
\label{fig:t_max}
\flushleft
Here the coordinate $\sigma$ denotes $\theta/\nu$.
\KMa{As in Figure \ref{fig-1d-stable-manifold-ex1}, the graph makes sense only in the region $\{t_{\max}(\sigma) \geq  0\}$ as the object defined by the blow-up solution of the original vector field (\ref{two-fluid}).
In the present validation result, initial points of the blow-up solution are distributed in the half-polydisk $\{\sigma \leq 0\}$ in the parameter space.
}
\end{figure}

\subsubsection{Rigorous computation of the coefficients \boldmath$r_n$\unboldmath}
\label{sec:rig_time_coeff_demo1}

Given $\ta=(\ta_1,\ta_2)$ with $\| \ta - \ba \|_X \le r = 4.2 \times 10^{-13}$ the power series coefficients of $P_i(u) = \sum_{n \ge 0} (\ta_i)_n u^n$. The goal in this section is to compute rigorously the coefficients $r_n$ of $R(u) = \sum_{n\ge 0} r_n u^n$ such that $[P_1(u)]^2 R(u) = Q(u)$. This amounts to solve the Taylor coefficients equation
\begin{equation} \label{eq:f_solve_for_r}
\psi(r) \bydef \ta_1^2 r - \tilde q = 0 . 
\end{equation} 
Using Newton's method, assume that we computed $\br$ such that $\psi(\br)\approx 0$
Denote by $D\psi^{(N)}(\br)$ the Jacobian of $\psi^{(N)}$ at $\br$. The next step is to construct the linear operator $A^\dag$ (an approximate derivative of the derivative $D\psi(\br)$), and the linear operator $A$ (an approximate inverse of $D\psi(\br)$). Let $\dagA$ be defined as
\[
(A^\dagger h)_n = 
\begin{cases}
\bigl(D\psi^{(N)}(\br) h^{(N)} \bigr)_n &\quad\text{for }   0 \leq n \le N ,  \\
(\ba_1^2)_0  &\quad\text{for }  n > N,
\end{cases}
\]
Consider now a matrix $A^{(N)} \in M_{N+1}(\R)$ computed so that $A^{(N)} \approx {D\psi^{(N)}(\br)}^{-1}$. This allows defining the linear operator $A$ whose action on an element $h \in \ell_\nu^1$ 
 \[
 (A h)_n =
 \begin{cases}
\left(A^{(N)} h^{(N)} \right)_n & \text{for }   0 \leq n \le N    \\
\frac{1}{(\ba_1^2)_0} h_n  & \text{for }  n > N.
 \end{cases}
 \]
Having obtained an approximate solution $\br$ and the linear operators $\dagA$ and $A$, the next step is to construct the bounds $Y_0$, $Z_0$, $Z_1$ and $Z_2(r)$ satisfying \eqref{eq:general_Y_0}, \eqref{eq:general_Z_0}, \eqref{eq:general_Z_1} and \eqref{eq:general_Z_2}, respectively. Note that since problem \eqref{eq:f_solve_for_r} is linear, then $Z_2=0$. \\

\noindent {\bf The bound \boldmath$Y_0$\unboldmath.} We look for a bound such that $\| A \psi(\br)\|_\nu \le Y_0$. Expand
\[
\psi(\br) = \ta_1^2 \br - \tilde q = (\ba_1+\delta_1)^2 \br - (\bar q + \delta_q) = \bar \psi (\br) + \psi^\delta(\br),
\]
where 
\[
\bar \psi (\br) \bydef \ba_1^2 \br - \bar q
\quad \text{and} \quad
\psi^\delta(\br) \bydef 2 \delta_1 \ba_1 \br + \delta_1^2 \br - \delta_q
\]
and $\delta_1 \bydef \ta_1-\ba_1$ and $\delta_q \bydef \tilde q - \bar q$. Hence, we can compute $Y_0$ such that
\begin{align*}
\|A \psi(\br)\|_\nu &\le \|A \bar \psi(\br)\|_\nu + \|A \|_{B(\ell_\nu^1)} \| \psi^\delta(\br)\|_\nu
\\
&\le \|A \bar \psi(\br)\|_\nu + \|A \|_{B(\ell_\nu^1)} \left( 2 \| \ba_1\|_\nu \| \br \|_\nu + \| \br \|_\nu r_0 + \frac{1}{\nu}  \right) r_0 \le Y_0,
\end{align*}
where we used that 
\[
\| \delta_q \|_\nu = \sum_{n \ge 0} | (\ta_2)_{n+1} - (\ba_2)_{n+1} | \nu^n = \frac{1}{\nu} \sum_{n \ge 0} | (\ta_2)_{n+1} - (\ba_2)_{n+1} | \nu^{n+1} \le \frac{1}{\nu} \| \ta_2 - \ba_2 \|_\nu \le  \frac{r_0}{\nu}.
\]

\noindent {\bf The bound \boldmath$Z_0$\unboldmath.} It is the same computation as the one presented in Section~\ref{sec:Z0}. \\

\noindent {\bf The bound \boldmath$Z_1$\unboldmath.} Given $h \in \ell_\nu^1$, denote 
\[
z \bydef D\psi(\br)h - \dagA h
\]
which is given component wise by 
\[
z_n =
\begin{cases}
\left( ( \ta_1^2 -  \ba_1^2)h \right)_n = \left( 2 \delta_1 \ba_1 h + \delta_1^2 h \right)_n
&\quad\text{for }   0 \leq n \le N ,  \\
( \ta_1^2 h)_n -  (\ba_1^2)_0 h_n = ( 2 (\delta_1)_0 (\ba_1)_0  + (\delta_1)_0^2   )h_n + \sum_{k=1}^n (\ta_1^2)_k h_{n-k}
&\quad\text{for }  n > N.
\end{cases}
\]
Define $\beta_k = (\ta_1^2)_k$ for $k>0$ and $\beta_0=0$. 
Hence, 
\begin{align*}
\|A z\|_\nu &\le \| A\|_{B(\ell_\nu^1)} \| 2 \delta_1 \ba_1 h + \delta_1^2 h \|_\nu + \frac{1}{(\ba_1^2)_0} \sum_{n \ge N+1} |(\beta * h)_n| \nu^n 
\\
& \le \| A\|_{B(\ell_\nu^1)} \left( 2 r_0 \|\ba_1\|_\nu + r_0^2 \right) + \frac{1}{(\ba_1^2)_0} \| \beta\|_\nu ,
\end{align*}
where
\[
\| \beta\|_\nu = \sum_{n \ge 1} |(\ta_1^2)_n| \nu^n
\le 
\sum_{n = 1}^{2N+2} |(\ba_1^2)_n| \nu^n + 2 \|\ba_1\|_\nu r_0 + r_0^2. 
\]
We therefore set
\[
Z_1 \bydef
 \| A\|_{B(\ell_\nu^1)} \left( 2 r_0 \|\ba_1\|_\nu + r_0^2 \right) + \frac{1}{(\ba_1^2)_0} \left( \sum_{n = 1}^{2N+2} |(\ba_1^2)_n| \nu^n + 2 \|\ba_1\|_\nu r_0 + r_0^2 \right).
\]

Assume that using the radii polynomial approach of Theorem~\ref{thm:radii_polynomials}, we prove the existence 
$\tilde r \in B_{r_{\min}}(\br)$ such that $\psi(\tilde r)=0$. Hence, given $\theta \in (-\nu,\nu)$, $t_{\max}$ given in \eqref{eq:t_max_v2} can be controlled 
\begin{align*}
t_{\max} &= 
-\frac{1}{\lambda} \sum_{n\ge 0} \frac{\tilde r_n}{n+1} \theta^{n+1} 
\\
& \in -\frac{1}{\lambda} \sum_{n = 0}^N \frac{\br_n}{n+1} \theta^{n+1} 
+ 
\frac{1}{|\lambda|} \sum_{n \ge 0} \frac{|\tilde r_n - \br_n|}{n+1} |\theta|^{n+1}[-1,1]
\\
& \in -\frac{1}{\lambda} \sum_{n = 0}^N \frac{\br_n}{n+1} \theta^{n+1} 
+ \frac{r_{\min}}{|\lambda|} [-1,1],
\end{align*}
which can be evaluated rigorously with interval arithmetic.

\begin{rem}
The above estimate directly shows the analyticity of $t_{\max}$ on $\theta$, which is implicitly guaranteed by analyticity of the parameterization $P$ and the uniform boundedness of the denominator $x_1(\eta) = P_1(u)$ away from $0$ on $W_{loc}^s(p_2)$. 
See Figure \ref{fig-1d-stable-manifold-ex1} about the latter fact.
\end{rem}

\subsection{Extension of the stable manifold of $p_2$ in \eqref{two-fluid-desing-poly} and blow-up time validations}

Once we validate the local stable manifold of a saddle equilibrium, we can extend the manifold integrating (\ref{two-fluid-desing-poly}) in the backward time direction, which is achieved by standard rigorous integrator of ODEs.
Recall that we rewrite the system of differential equations \eqref{two-fluid-desing-poly} as in \eqref{eq:ODE_1d_manifold}, that is 
\begin{equation*}
\dot{x} = g(x) = \begin{pmatrix} g_1(x_1,x_2) \\  g_2(x_1,x_2) \end{pmatrix} \bydef 
\begin{pmatrix}
{\displaystyle
	x_1^3 - (\rho_1+\rho_2) x_1^2 + \rho_1 \rho_2 x_1 - cx_1^3 x_2 - c_1 x_1^2 x_2
}
\vspace{.1cm}
\\
{\displaystyle
	-\frac{1}{2}x_1^2x_2  + \frac{1}{2} \rho_1\rho_2 x_2 + c x_1^2 x_2^2 + c_2 x_1^2 x_2^3
}
\end{pmatrix},
\end{equation*}
where $\dot{{}}=\frac{d}{d\eta}$, $x_2\equiv s$, $(\rho_1,\rho_2)=(1,2)$, $(\beta_R, v_R) = (1.5, 5)$, $(\beta_L, v_L) = (1.9, 4)$ with the constant $c$ in (\ref{speed-two-phase})
and $(c_1, c_2) = (c_{1L}, c_{2L})$ satisfying
\[
\begin{cases}
c_{1L} = v_L B_1(\beta_L) - c\beta_L, & \\
c_{2L} = v_L^2 B_2(\beta_L) -cv_L. & \\
\end{cases}
\]

We integrate \eqref{two-fluid-desing-poly} backward in time. 
Taking $\xi\bydef -\eta$,
we integrate
\begin{equation}\label{eq:ODE_on_W1_backward}
\begin{cases}
{\displaystyle
	\frac{dx_1}{d\xi}=-\left(x_1^3 - (\rho_1+\rho_2) x_1^2 + \rho_1 \rho_2 x_1 - cx_1^3 x_2 - c_1 x_1^2 x_2\right),
}
\vspace{.1cm}
\\
{\displaystyle
	\frac{dx_2}{d\xi}=-\left(-\frac{1}{2}x_1^2x_2  + \frac{1}{2} \rho_1\rho_2 x_2 + c x_1^2 x_2^2 + c_2 x_1^2 x_2^3\right).
}
\end{cases}
\end{equation}
from $0$ to $\xi_0$ with the initial point $(x_1(0),x_2(0))=p_0=P(\theta)|_{\theta=-1}$, which is on the local stable manifold $W_{loc}^s(p_2)$.
The rigorous integrator we have used is mentioned in Remark \ref{rem:integrator}.
Furthermore, we rigorously compute the passing time in the original time scale using the following formula:
\[
t_{\xi_0} = \int_{0}^{\xi_0} \frac{x_2(\xi)}{x_1(\xi)^2} ~ d \xi,
\]
where $x_1(\xi)$ and $x_2(\xi)$ denote the solution of \eqref{eq:ODE_on_W1_backward}.

\par
In the present example, (\ref{two-fluid-desing-poly}) is integrated with the initial point at the boundary of locally validated stable manifold, which is the boundary of the red curve in Figure \ref{fig-1d-stable-manifold-ex1} with $x_2 > 0$, in the backward time direction and compute an enclosure of the evolution time in the original time-scale:
\begin{equation*}
t_{-\eta} = \int_{0}^{-\eta} \frac{x_2(\tilde \eta)}{x_1(\tilde \eta)^2} ~ d \tilde \eta.
\end{equation*}
The blow-up time of the corresponding blow-up solution with the initial point $T_d^{-1}(x_1(\eta), x_2(\eta))$ is then enclosed by the sum of enclosures of $t_{\max}$ and $t_{-\eta}$.
Figure \ref{fig:t_max_global_ex1} draws the blow-up time $t_{\max}$ of blow-up solutions as a function of initial points on $T_d^{-1}(W^s(p_2))$.
Note that the point in the figure where the corresponding blow-up time tends to infinity is the source equilibrium for \eqref{two-fluid-desing-poly}, which corresponds to the bounded source for \eqref{two-fluid}.
Rigorous enclosures of $t_{\max}$ on several sample points are shown in Table \ref{Tab:Demo1}.
Finally, we can reconstruct the true blow-up profile of the validated saddle-type blow-up solution through the directional compactification $T_d$, which is drawn in Figure~\ref{fig:blowup_profile_ex1}.
Note that this profile cannot be computed in the direct way since small perturbations of initial points violate the profile\footnote{
As far as we have calculated (in non-rigorous sense), solutions of  (\ref{two-fluid-desing-poly}) through points {\em near} validated solutions (in Figure \ref{fig:t_max_global_ex1}) go to the direction so that the $x_1$-component goes to $+\infty$ directly, or rounding the bounded source (near $P_5$ in Figure \ref{fig:t_max_global_ex1}).
}.

\begin{figure}[htbp]\em
	\centering
	\includegraphics[width=.6\textwidth]{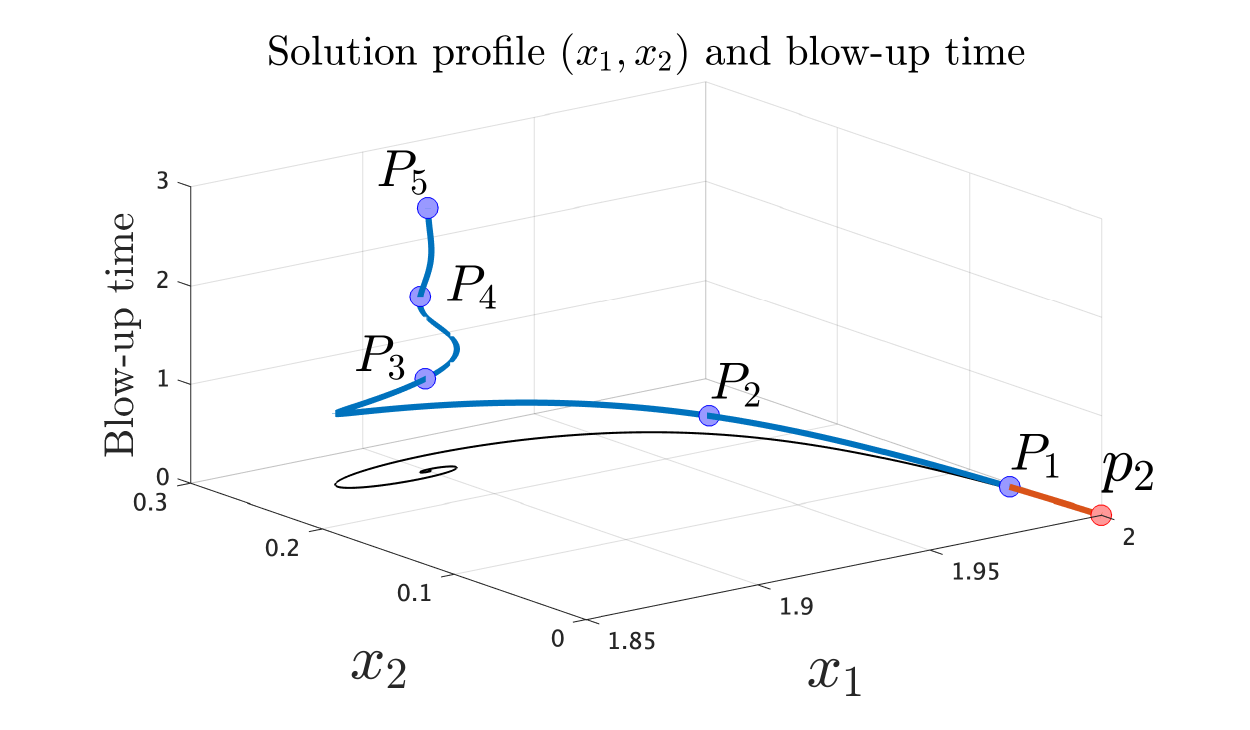}
	\caption{The extended stable manifold $W^s(p_2)$ for \eqref{two-fluid-desing-poly} and corresponding blow-up times}
	\flushleft
	The blue curve is the validated stable manifold $W^s(p_2)$\KMa{, while the black curve is the projection onto the $(x_1, x_2)$-plane}.
	Numbers near points along the curve correspond to those shown in Table \ref{Tab:Demo1} where the rigorous enclosures of blow-up times are shown.
	\label{fig:t_max_global_ex1}
\end{figure}

\begin{table}[ht]
\centering
{
\begin{tabular}{ccccc}
\hline
Points (label) & $x_1$ & $x_2$ & Blow-up time\\
\hline\\[-2mm]
$P_1$ & $1.99704842870_{221}^{362}$ & $0.06209042154_{03}^{164}$ & $0.01945344745_{624}^{758}$\\
$P_2$ & $1.971379977171_{031}^{454}$ & $0.22265490227_{37387}^{46532}$ & $0.1821531459_{776968}^{806739}$\\
$P_3$ & $1.895702934910_{105}^{671}$ & $0.24142735005_{28752}^{30887}$ & $1.00170345745_{2293}^{7477}$\\
$P_4$ & $1.89771158641_{7872}^{819}$ & $0.250316449049_{6631}^{8725}$ & $1.78231786657_{067}^{7252}$\\
$P_5$ & $1.899856004192_{361}^{656}$ & $0.25017265254_{49681}^{51455}$ & $2.6651422937_{42664}^{50833}$\\
\hline 
\end{tabular}%
}
\caption{Blow-up time enclosures for \eqref{two-fluid}}
\flushleft
\lq\lq Points (label)" correspond to points drawn in Figure \ref{fig:t_max_global_ex1}.
\lq\lq Blow-up time" is the validated enclosure of blow-up time for \eqref{two-fluid} through the preimage of points under $T_d$.
\label{Tab:Demo1}
\end{table}%

\begin{figure}[htbp]\em
\centering
\includegraphics[width=.6\textwidth]{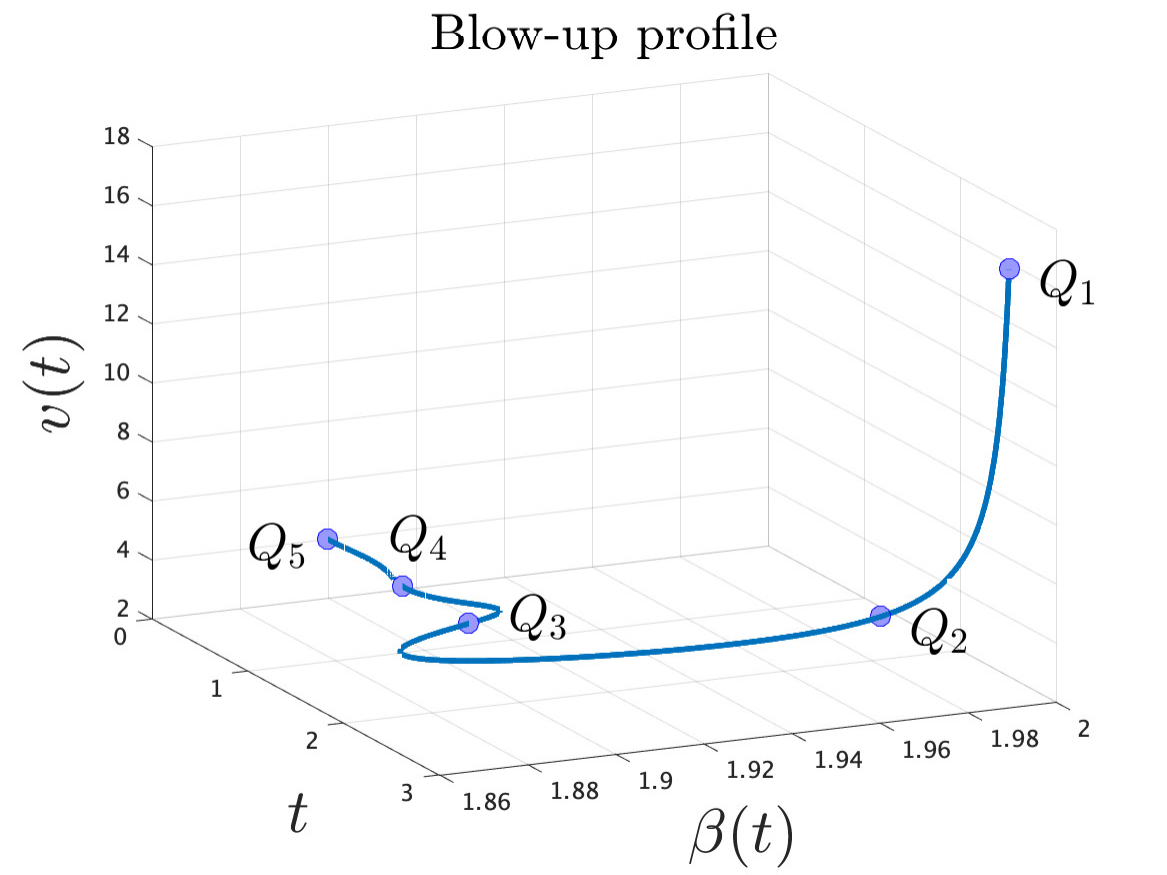}
\caption{Blow-up profile corresponding to Figure \ref{fig:t_max_global_ex1}}
\flushleft
Each point $Q_i$ ($i=1,\KMa{\ldots}, 5$) corresponds to the preimage of $P_i$ in Figure \ref{fig:t_max_global_ex1} under the directional compactification $(x_1, x_2) = (\beta, v^{-1})$.
The initial time $t=0$ is set so that $Q_5 = (\beta(0), v(0))$.
\label{fig:blowup_profile_ex1}
\end{figure}

\begin{rem}
The integrand of $t_{\max}$ has a different form from typical integrands shown in Section \ref{sec:pre_blow_up}. 
Indeed, the integrand of (\ref{blow-up-time-demo1}) is a rational function consisting of two analytic functions.
Nevertheless, the function $x_1(\eta)$ determining the denominator attains the value around $2$ with sufficiently small error bounds so that the function $1/x_1(\eta)^2$ is analytic at $x_1(0)$, which is justified through the parameterization $P$, provided the trajectory $\{x_1(\eta), s(\eta)\}_{\eta\in [0,\infty)}$ is located on the interior of $W_{loc}^s(p_2)$.
In particular, Proposition \ref{prop-blow-up-time-analytic} and Theorem \ref{thm-blow-up-time-analytic} can be still applied to showing that $t_{\max}$ \KMa{defined by} (\ref{blow-up-time-demo1}) depends analytically on initial \KMa{points}.
Note that arguments in Section \ref{sec:rig_time_coeff_demo1} directly confirm the analyticity of $t_{\max}$.
\end{rem}

\section{Example 2: application to higher-dimensional systems}

\label{sec:demo2}
The second example is the following (artificial) system in $\mathbb{R}^3$:
\begin{equation}
\label{Nagumo-at-infty}
\begin{cases}
y_1' = y_1(y_1^2-1), & \\
y_2' = y_1^2 y_2 + y_1^2 y_3, & \\
y_3' = y_1^2 y_3 + \delta^{-1}\left\{cy_1^2 y_3 - y_2(y_2-a y_1)(y_1-y_2) + wy_1^3\right\}. & 
\end{cases}
\end{equation}
The present system is asymptotically homogeneous of order $3$, namely asymptotically quasi-homogeneous of type $\alpha = (1,1,1)$.
We thus apply the Poincar\'{e}-type compactification\footnote{
In the present demonstration, radicals in the Poincar\'{e}-type compactification do not prevent us from $C^1$ studies of dynamical systems.
In particular, the linear stability analysis of equilibria on the horizon makes sense.
Indeed, the lower-order terms in (\ref{Nagumo-at-infty}) are chosen so that our methodology properly works, following discussions in \cite{Mat2018}.
} to obtain the associated desingularized vector field as written by \KMa{(\ref{vectorfield-cw-Poincare})}.
In the present case, $k = 2, n=3$, $\alpha_j = \beta_j = c = 1$ for $j=1,\KMa{\ldots}, n$ and hence
\begin{equation}
\label{central-Nagumo}
\begin{cases}
\tilde f_1(x) = x_1^3 - \left( 1- \sum_{i=1}^3 x_i^2 \right)x_1, & \\
\tilde f_2(x) = x_1^2 x_2 + x_1^2 x_3, & \\
\tilde f_3(x) = x_1^2 x_3 + \delta^{-1}\left\{cx_1^2 x_3 - x_2(x_2-a x_1)(x_1-x_2) + wx_1^3\right\}, & 
\end{cases}
\end{equation}
derived by (\ref{f-tilde-Poincare}), is applied to determining (\ref{vectorfield-cw-Poincare}).
The concrete form is
\begin{align}
\notag
\dot x_1 &= g_1(x) \bydef \tilde f_1(x) - x_1G(x),
\\
\label{desing-demo2}
\dot x_2 &= g_2(x) \bydef \tilde f_2(x) - x_2G(x), 
\\
\notag
\dot x_3 &= g_3(x) \bydef \tilde f_3(x) - x_3G(x), 
\end{align}
where
\begin{align*}
G(x) &\bydef \sum_{j=1}^3 x_j \tilde f_j(x)\\
	&= x_1 \left\{x_1^3 - \left( 1- \sum_{i=1}^3 x_i^2 \right)x_1 \right\} + x_2 \left\{ x_1^2 x_2 + x_1^2 x_3 \right\} \\
	&\quad + x_3 \left[x_1^2 x_3 + \delta^{-1}\left\{cx_1^2 x_3 - x_2(x_2-a x_1)(x_1-x_2) + wx_1^3\right\}\right]\\
	&= x_1^2 \left\{ -1 + 2x_1^2 + x_1x_2 + 2x_2^2 + \delta^{-1}w x_1 x_3+ (2+\delta^{-1}c)x_3^2 \right\} -\delta^{-1}x_2x_3(x_2-a x_1)(x_1-x_2).
\end{align*}
The direct calculation of the Jacobian matrix of (\ref{desing-demo2}) is quite lengthy.
Assuming that the Jacobian matrix of $\tilde f$ with respect to $x$ is calculated, the Jacobian matrix of $g$ with respect to $x$ is calculated as follows:
\begin{equation*}
\frac{\partial g_i}{\partial x_j} = \frac{\partial \tilde f_i}{\partial x_j} - \delta_{ij} \left( \sum_{k=1}^3 x_k \tilde f_k(x)\right) - x_i  \sum_{k=1}^3 \left\{ \delta_{jk} \tilde f_k(x) + x_k \frac{\partial \tilde f_k}{\partial x_j} \right\},
\end{equation*}
where $\delta_{ij}$ is the Kronecker's delta.
In the present case, the Jacobian matrix of $\tilde f$ is
\begin{align*}
J\tilde f &= 
\begin{pmatrix}
3x_1^2 - \left( 1- \sum_{i=1}^3 x_i^2 \right) +2x_1^2 & 2x_1x_2 & 2x_1x_3 \\
2x_1(x_2+x_3) & x_1^2 & x_1^2 \\
\tilde f_{31} & \tilde f_{32} & \tilde f_{33}
\end{pmatrix},\\
\tilde f_{31} &= 2x_1x_3 + \delta^{-1} \left\{ 2cx_1x_3 + ax_2(x_1 - x_2) - x_2(x_2-ax_1) + 3wx_1^2\right\},\\
\tilde f_{32} &=  \delta^{-1} \left\{ -(x_2-ax_1)(x_1-x_2) - x_2(x_1-x_2) + x_2(x_2-ax_1) \right\},\\
\tilde f_{33} &= (1+\delta^{-1}c)x_1^2.
\end{align*}

We observe that there are (at least) three equilibria on the horizon $\{p(x)^2 \equiv \sum_{i=1}^3 x_i^2 = 1\}$, one of which, denoted by $p_0$, has a one-dimensional stable manifold and two of which, denoted by $p_1$ and $p_2$, have two-dimensional stable manifolds.
In the present study we fix the following parameters:
\begin{equation*}
(a, c, \delta, w) = (0.3, 0.7, 9.0, 0.02).
\end{equation*}
We have computed the concrete position and associated eigenvalues, which are approximately given as follows:
\begin{align*}
p_0 & \approx (0.9333789, 0.3588924, 0),\\
\lambda_1(p_0) &\approx -1.74239248,\quad \lambda_2(p_0) \approx 0.033880 + 0.1430256 i,\quad \lambda_3(p_0)  = \overline{\lambda_2(p_0)},\\
p_1 &\approx (0.7180928, 0.6959473, 0),\\
\lambda_1(p_1) &\approx -0.11437086,\quad \lambda_2(p_1) = 0.1544775,\quad \lambda_3(p_1)  \approx -1.0313145,\\
p_2 &\approx (0.9985628, -0.0535924, 0),\\
\lambda_1(p_2) &\approx -1.994255,\quad \lambda_2(p_2) \approx -0.1870901,\quad \lambda_3(p_2)  \approx 0.26464449.
\end{align*}

On the other hand, (\ref{central-Nagumo}) possesses a source in a bounded region, namely $\{\sum_{i=1}^3 x_i^2 < 1\}$, which is
\begin{equation*}
p_b \approx (0.7071051816183367, 0.001504037399468, -0.001504037399468).
\end{equation*}

The parameterization method applied to three equilibria on the horizon; $p_0$, $p_1$ and $p_2$,  for (\ref{desing-demo2}) provides local stable manifolds with rigorous error enclosures. 
Distributions of these local stable manifolds are drawn in Figure \ref{fig-Nagumo-global}.

\begin{figure}[htbp]\em
	\centering
	\includegraphics[width=9cm]{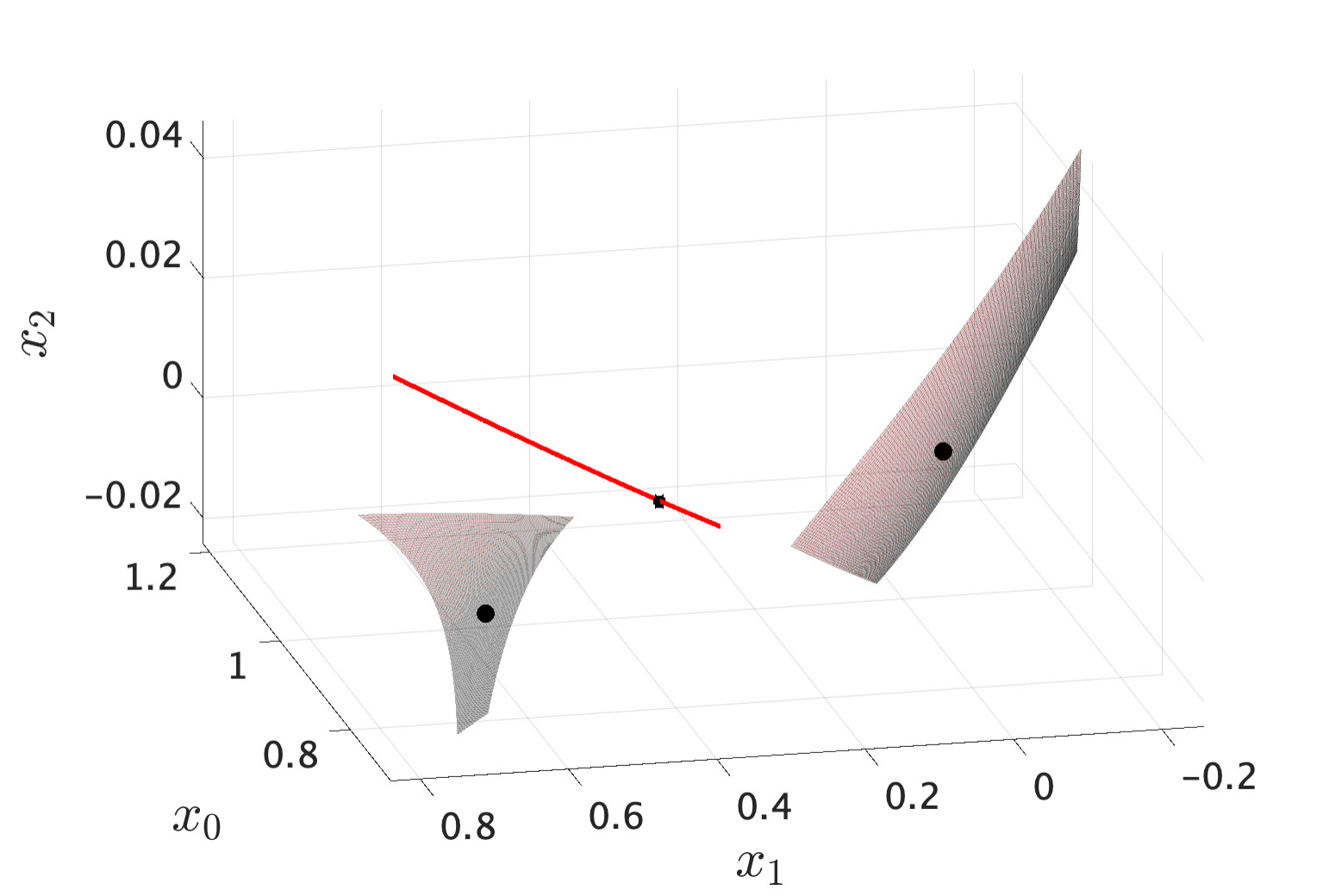}
	\caption{The rigorously computed local stable manifolds for hyperbolic equilibria for (\ref{desing-demo2}).} 
	\flushleft
	The $C^0$ rigorous error bound for the manifold around $p_1$ (left) is $\|P-P^{(N)}\|_\infty \leq r=8.2 \times 10^{-9}$ with $N = 50$, while it is $\|P-P^{(N)}\|_\infty \leq r =9.8 \times 10^{-10}$ with $N = 60$ for the manifold around $p_2$ (right) and $\|P-P^{(N)}\|_\infty \leq r =9.8 \times 10^{-13}$ with $N = 160$ for the manifold around $p_0$ (center). The black dots are equilibria on the horizon; denoting $p_1$, $p_0$ and $p_2$ from the left to the right.
\label{fig-Nagumo-stable}
\end{figure}

\subsection{Blow-up time computation}
\label{sec:blow_up_time_2}
Since the compactification is homogeneous (namely $\alpha = (1,\KMa{\ldots}, 1)$ for defining compactifications) and $k=2$ in the present example, the maximal existence time $t_{\max}$ is
\begin{equation}
\label{tmax-demo2}
t_{\max} = \int_0^\infty \kappa(x(\tau))^{-k} d\tau = \int_0^\infty \left(1-\|x\|^2\right) d\tau,
\end{equation}
according to (\ref{time-desing-Poin-integral}).
Let $P$ be a parameterization around $x_\ast \in \mathcal{E}$ whose image of $B^{n_s}$ determines the local stable manifold $W^s_{loc}(x_\ast)$ of $x_\ast$ such that $P(0) = x_\ast$.
$P$ is assumed to have a polynomial expression (cf. (\ref{param_expression}))
\begin{equation*}
P(\theta) = \sum_{|\alpha| \geq 0}a_{\alpha}\theta^{\alpha},\quad \theta = \begin{pmatrix}
\theta_1 \\ \vdots \\ \theta_{n_s}
\end{pmatrix}\in \mathbb{R}^{n_s},\quad 
a_{\alpha} = \begin{pmatrix}
(a_1)_\alpha \\ \vdots \\ (a_n)_\alpha
\end{pmatrix}\in \mathbb{R}^n
\end{equation*}
satisfying $a_{\bf 0} = x_\ast$.
$\alpha = (\alpha_1, \KMa{\ldots}, \alpha_{n_s}) \in \mathbb{Z}_{\geq 0}^{n_s}$ denotes the multi-index and $\theta^\alpha = \theta_1^{\alpha_1}\cdots \theta_{n_s}^{\alpha_{n_s}}$.
Assuming that the solution trajectory $x(\tau)$ is on $W^s_{loc}(x_\ast)$, the parameterization argument indicates that
\begin{equation*}
x(\tau) = P(Q^{-1} e^{\Lambda \tau}Q\theta_0),\quad \Lambda = {\rm diag}(\lambda_1, \KMa{\ldots}, \lambda_{n_s})\quad\text{ with }\quad {\rm Re}\,\lambda_i < 0.
\end{equation*}
For a while, we further assume that $Q=I$, $\lambda_i \in \mathbb{R}$ for $i=1,\KMa{\ldots}, n_s$ and $k=2$.
Then 
\begin{equation*}
P(\theta) = \sum_{|\alpha| \geq 0}a_{\alpha}\theta(\tau)^{\alpha},\quad \theta(\tau) = \begin{pmatrix}
e^{\lambda_1 \tau}(\theta_{1})_0 \\ \vdots \\ e^{\lambda_{n_s} \tau} (\theta_{n_s})_0
\end{pmatrix},\quad 
\theta_0 = \begin{pmatrix}
(\theta_{1})_0 \\ \vdots \\ (\theta_{n_s})_0
\end{pmatrix},\quad a_\alpha \equiv ((a_1)_\alpha, \KMa{\ldots}, (a_n)_\alpha )\in \mathbb{R}^n
\end{equation*}
and
\begin{align*}
t_{\max} &= \int_0^\infty \left\{ 1- \sum_{i=1}^n \left(\sum_{|\alpha| \geq 0} (a_i)_{\alpha} \theta(\tau)^{\alpha}\right)^2 \right\} d\tau \\
&= \int_0^\infty \left\{ 1- \sum_{i=1}^n \left(\sum_{|\beta| \geq 0}\sum_{|\gamma| \geq 0} (a_i)_\beta  (a_i)_\gamma e^{(\sum_{j=1}^m (\beta_j + \gamma_j)\lambda_j) \tau}\theta_0^{\beta + \gamma} \right) \right\} d\tau \\
&= \int_0^\infty  \left\{ 1 - \sum_{i=1}^n \sum_{|\alpha| \geq 0} ( a_i \ast a_i )_\alpha e^{(\alpha \cdot \lambda)\tau} \theta_0^\alpha \right\} d\tau,
\end{align*}
where $(a\ast b)_{\alpha}$ denotes the discrete convolution over the multi-index $\alpha\in \mathbb{Z}_{\geq 0}^{n_s}$ given in (\ref{conv-multi}) and $\theta_0^\alpha = ((\theta_{1})_0)^{\alpha_1}\cdots ((\theta_{n_s})_0)^{\alpha_{n_s}}$.
Here we use the fact 
\begin{equation*}
\sum_{i=1}^n \sum_{|\alpha| = 0} ( a_i \ast a_i )_\alpha e^{(\alpha \cdot \lambda)\tau} \theta^\alpha = \sum_{i=1}^n \left( (a_i)_{\bf 0} \right)^2 = \|x_\ast\|^2 = 1
\end{equation*}
\KMa{because} $P(0) = x_\ast$ and $x_\ast \in \mathcal{E} = \{\|x\|=1\}$.
Thus we have 
\begin{align*}
\int_0^\infty  \left\{ 1 - \sum_{i=1}^n \sum_{|\alpha| \geq 0} ( a_i \ast a_i )_\alpha e^{(\alpha \cdot \lambda)\tau} \theta_0^\alpha \right\} d\tau 
	&= - \int_0^\infty \sum_{|\alpha| > 0} \sum_{i=1}^n ( a_i \ast a_i )_\alpha e^{(\alpha \cdot \lambda)\tau} \theta_0^\alpha  d\tau\\
	&= - \sum_{|\alpha| > 0} \left(\sum_{i=1}^n( a_i \ast a_i )_\alpha \right) \frac{\theta_0^\alpha}{\alpha \cdot \lambda},
\end{align*}
where the denominator is strictly negative for all possible $\alpha$ and the analyticity of $P$ ensures the convergence of the above series. 
Finally, we have the following expression of $t_{\max}$:
\begin{align}
\label{tmax_Poincare}
t_{\max} &= - \sum_{|\alpha| > 0} \left(\sum_{i=1}^n( a_i \ast a_i )_\alpha \right) \frac{\theta_0^\alpha}{\alpha \cdot \lambda}.
\end{align}
Remark that the above expression makes sense only if 
\begin{equation*}
\|P(\theta_0)\|^2 =  \sum_{i=1}^n \left(\sum_{|\alpha| \geq 0} (a_i)_{\alpha} \theta(\tau)^{\alpha}\right)^2 = 1 +  \sum_{|\alpha| > 0} \left(\sum_{i=1}^n ( a_i \ast a_i )_\alpha \right) \theta(\tau)^{\alpha} < 1
\end{equation*}
by definition of the Poincar\'{e} compactification.
With an explicit expression or enclosure of $P(\theta)$, the quantity (\ref{tmax_Poincare}) or its enclosure is rigorously calculated for each $\theta_0 \in B^{n_s}$.
The above procedure is applied with  $n=3$ and $n_s = 1$ or $2$ in the present problem.
\par
If $n_s = 1$, the expression (\ref{tmax_Poincare}) can be simplified by considering the single index $l\geq 1$ instead of the multi-index $\alpha$ to obtain
\begin{align*}
t_{\max} &= - \frac{1}{\lambda}\sum_{l \geq 1}  \left\{ \sum_{i=1}^n  ( a_i \ast a_i )_l  \right\}\frac{\theta_0^{ l }}{l }.
\end{align*}


In practice the computation of the Taylor coefficients $a_1 ,\dots,a_n$ comes from a successful application of the Newton-Kantorovich type theorem (Theorem~\ref{thm:radii_polynomials}) applied to $F:X\to X'$ given in \eqref{eq:parameterization_F(a)=0}. More precisely, denote by $\ba_1,\dots,\ba_n$ the numerical approximations (of order $N$) and $r_0>0$ such that the true coefficients satisfy 
\[
\|a - \ba \|_X = \max_{j=1,\dots,n} \| a_j -\ba_j \|_1 \le r_0.
\]
Denote $b = a - \ba$ and note that 
\begin{align*}
t_{\max} &= - \sum_{|\alpha| > 0} \left(\sum_{i=1}^n(a_i \ast a_i )_\alpha \right) \frac{\theta_0^\alpha}{\alpha \cdot \lambda}
\\
&=  - \sum_{|\alpha| = 0}^{2N} \left(\sum_{i=1}^n(\ba_i \ast \ba_i )_\alpha \right) \frac{\theta_0^\alpha}{\alpha \cdot \lambda}
- 2 \sum_{|\alpha| > 0} \left(\sum_{i=1}^n(\ba_i \ast b_i )_\alpha \right) \frac{\theta_0^\alpha}{\alpha \cdot \lambda} \\
& \quad
- \sum_{|\alpha| > 0} \left(\sum_{i=1}^n(b_i \ast b_i )_\alpha \right) \frac{\theta_0^\alpha}{\alpha \cdot \lambda}.
\end{align*} 
Denote, the spectral gap of the stable eigenvalues by
\[
\sigma_{\rm gap} \bydef \min_{j=1,\dots,n_s} |\lambda_j| > 0
\]
and note that $\sigma_{\rm gap} = \min_{|\alpha|>0} |\alpha \cdot \lambda|$. Hence, for all $\theta_0 \in B_1^{n_s}$, 
\begin{align*}
\left| 
2 \sum_{|\alpha| > 0} \left(\sum_{i=1}^n(\ba_i \ast b_i )_\alpha \right) \frac{\theta_0^\alpha}{\alpha \cdot \lambda} \right| 
& = 
\left| 
2 \sum_{i=1}^n \left( \sum_{|\alpha| > 0}(\ba_i \ast b_i )_\alpha  \frac{\theta_0^\alpha}{\alpha \cdot \lambda} \right) \right| 
\\
& \le \frac{2}{\sigma_{\rm gap}} \sum_{i=1}^n  \sum_{|\alpha| > 0} |(\ba_i \ast b_i )_\alpha | 
\\
& = \frac{2}{\sigma_{\rm gap}} \sum_{i=1}^n \| \ba_i \ast b_i \|_1
\\
& \le \left( \frac{2}{\sigma_{\rm gap}} \sum_{i=1}^n \| \ba_i \|_1\right) r_0.
\end{align*}
Similarly, we can show that
\[
\left| - \sum_{|\alpha| > 0} \left(\sum_{i=1}^n(b_i \ast b_i)_\alpha \right) \frac{\theta_0^\alpha}{\alpha \cdot \lambda} \right|
\le \frac{n r_0^2}{\sigma_{\rm gap}}.
\]
Denoting 
\[
\tilde \delta \bydef \left( \frac{2}{\sigma_{\rm gap}} \sum_{i=1}^n \| \ba_i \|_1\right) r_0 + \frac{n r_0^2}{\sigma_{\rm gap}},
\]
then a rigorous enclosure of $t_{\max}$ is given by the computable formula
\[
t_{\max} \in - \sum_{|\alpha| = 0}^{2N} \left(\sum_{i=1}^n(\ba_i \ast \ba_i )_\alpha \right) \frac{\theta_0^\alpha}{\alpha \cdot \lambda}
+ [-\tilde \delta,\tilde \delta].
\]

\subsection{Distribution of $t_{\max}$ near blow-up}
\label{sec:distribution_tmax}
In the present example, saddle equilibria $p_1$ and $p_2$ on the horizon both have $2$-dimensional stable manifolds.
Once the parameterization method is applied to validating these invariant manifolds, the blow-up time $t_{\max}$ defined by (\ref{tmax-demo2}) is obtained as a function of the parameter $\theta$ determining local stable manifolds.
In particular, we can validate distributions of $t_{\max}$ on local stable manifolds.
\par
Figure \ref{fig:blow-up-time_demo2} draws the distributions of $t_{\max}$.
\KMa{Because} the vector field (\ref{desing-demo2}) itself can be defined outside $\overline{\mathcal{D}}$, namely in $\{\|x\| > 1\}$ also, $t_{\max}$ can attain negative values.
Nevertheless, from the viewpoint that (\ref{desing-demo2}) is obtained from (\ref{Nagumo-at-infty}) through the compactification, only the positive values make sense as the blow-up time of solutions to (\ref{Nagumo-at-infty}).
Now we pay attention to the following facts, which follow from fundamental arguments of compactifications (cf. \cite{Mat2018}):
\vspace{-.1cm}
\begin{itemize}
\item The horizon $\mathcal{E}$ is a codimension one invariant submanifold of $\mathbb{R}^3$.
\vspace{-.2cm}
\item The integrand determining $t_{\max}$ (e.g. (\ref{tmax-demo2})) is identically zero on $\mathcal{E}$.
\vspace{-.1cm}
\end{itemize}

\begin{figure}[h!]\em
\begin{minipage}{0.5\hsize}
\centering
\includegraphics[width=7cm]{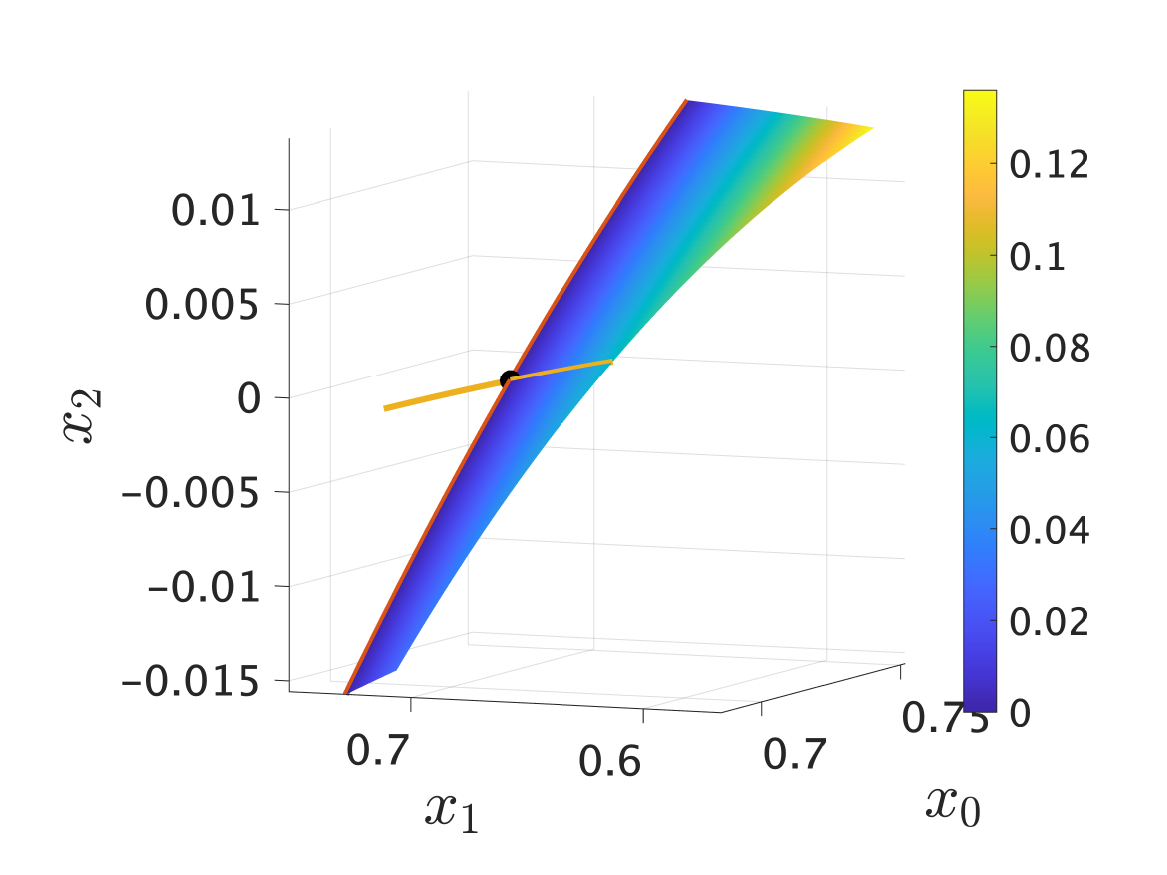}
(a)
\end{minipage}
\begin{minipage}{0.5\hsize}
\centering
\includegraphics[width=7cm]{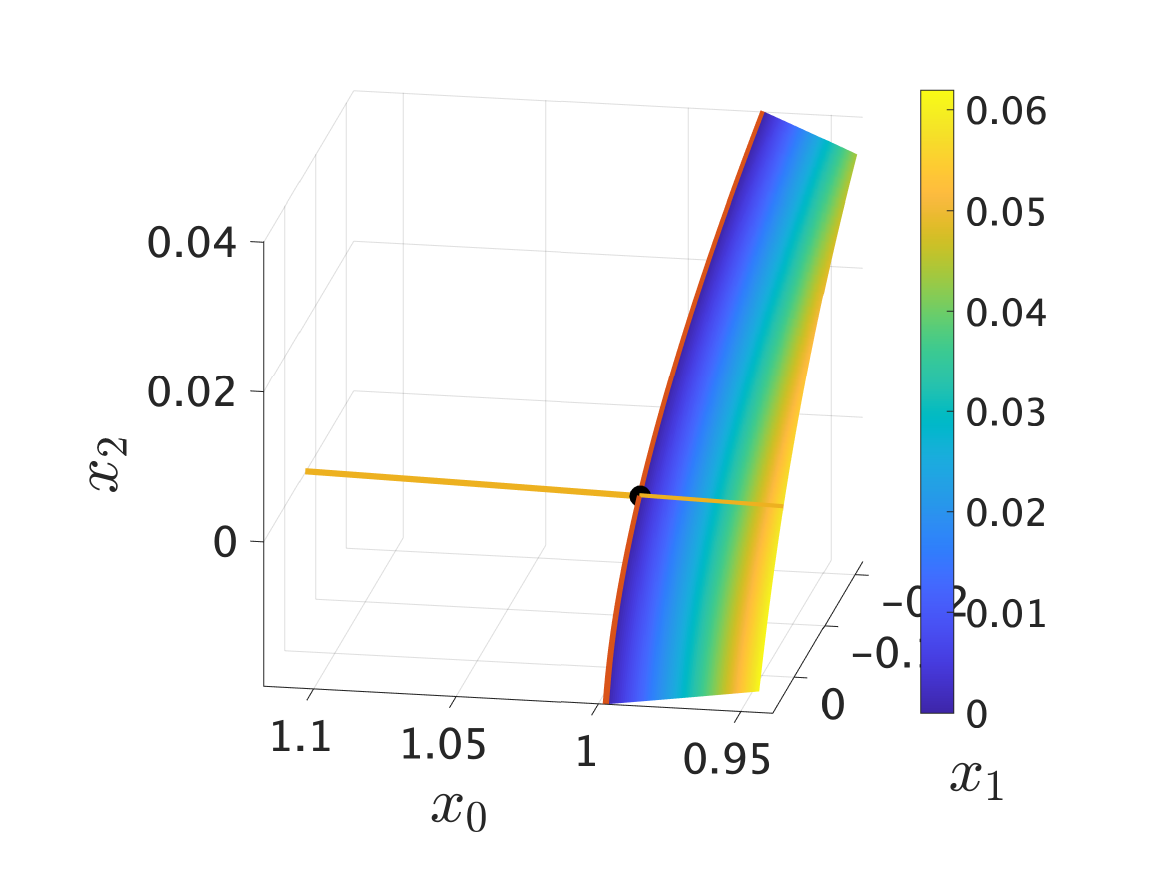}
(b)
\end{minipage}
\caption{Distribution of $t_{\max}$ in (\ref{tmax-demo2}).}
\flushleft
(a) Distribution of $t_{\max}$ around $p_1$.
(b) Distribution of $t_{\max}$ around $p_2$.
Surfaces are validated local stable manifolds of equilibria (black dots).
Only positive values of $t_{\max}$ make sense as blow-up times of blow-up solutions for (\ref{Nagumo-at-infty}), and hence these surfaces are drawn only in the regions where $t_{\max} \geq 0$.
In both figures, yellow curves and red curves denote $P(\{\theta_2 = 0\})$ and $P(\{\theta_1 = 0\})$, respectively.
\KMa{The graphs of $P(\{\theta_2 = 0\})$ are drawn outside the horizon (red curves) because these curves correspond to coordinate axes of local stable manifolds.}
The red curves are located in the horizon $\mathcal{E}$\KMa{, reflecting the invariant structure of $\mathcal{E}$}. 
According to eigendirections at equilibria, asymptotic behavior of trajectories on these manifolds are essentially governed by dynamics on $P(\{\theta_1 = 0\})$. 
On the other hand, dynamics in \KMa{this direction makes} little contributions to $t_{\max}$.
\label{fig:blow-up-time_demo2}
\end{figure}

Results in Figure \ref{fig:blow-up-time_demo2} indeed reflect the above nature.
For example, one-dimensional submanifold of two-dimensional stable manifolds of $p_1$ and $p_2$ are located on the horizon where $t_{\max}$ is identically zero.
Our computations further indicate that the region $\{t_{\max} > 0\}$ is included in $\{\|x\| < 1\}$.
Looking at the region $\{t_{\max} > 0\}$, like the previous example in Figure \ref{fig:t_max}, we can discuss the distribution of blow-up times.
\par
From our present observations, we have an interesting result about the distribution of blow-up times.
In the present example, eigenvalues determining stable submanifolds on the horizon have smaller moduli than the transverse direction.
In other words, {\em the leading (stable) eigendirections are directed tangent to the horizon (red curves in Figure \ref{fig:blow-up-time_demo2}) in both manifolds}.
Asymptotic behavior of trajectories around equilibria is therefore essentially determined by the exponential decay behavior in the direction parallel to the horizon.
On the other hand, level sets of $t_{\max}$ are distributed so that they are foliated parallel to the horizon, equivalently the level set $t_{\max} = 0$, in both cases.
These observations may look strange from the viewpoint of the asymptotic behavior around (hyperbolic) equilibria.
Indeed, dynamics around hyperbolic equilibria of interest are essentially governed by leading eigendirection, implying that the behavior along the leading eigendirection should mainly contribute to estimate $t_{\max}$.
However, the integrand in (\ref{tmax-demo2}) is almost zero near the horizon. 
More precisely, according to the proof of the blow-up criterion theorem (Theorem \ref{thm:blowup-Poincare} whose proof is found in \cite{Mat2018}), the integrand as a function of $\tau$ decays exponentially fast near the horizon\footnote{
Hyperbolicity of equilibria is used for the proof, implying that the dynamical property of equilibria, and potentially general invariant sets, plays a key role in determining the distribution of $t_{\max}$ around $0$.
}.
Therefore asymptotic behavior of solution trajectories near the horizon does little contributions to $t_{\max}$.
As a consequence, {\em blow-up time is essentially foliated parallel to the horizon, no matter where the leading eigendirection is distributed}.
This is a reason why the level set of $t_{\max}$ is distributed parallel to the horizon.


\subsection{Extension of blow-up solutions}
\label{sec:demo2_extension}
As demonstrated in Section \ref{sec:demo1}, we can extend local stable manifolds globally by rigorous integration of (\ref{desing-demo2}) in backward time direction.
In the present case, we have a (bounded) source equilibrium $p_b$ and we have succeeded in validating connecting orbits between three equilibria on the horizon and $p_b$.
The validated {\em global} stable manifolds are drawn in Figure \ref{fig-Nagumo-global}. 
These stable manifolds separate the asymptotic behavior of solution trajectories outside the manifolds, although we omit the detailed description of phase portraits \KMa{because} it is hard to clearly visualize.
\par
\KMa{Note that the present validation of connecting orbits is done by the method typically used in the similar works (e.g., \cite{MT2020_1}). 
In particular, solutions approaching to trapping regions of equilibria are validated for the existence of global-in-time existence of solutions.
In the present work, trapping regions of sink equilibria are validated by means of local Lyapunov functions (cf. \cite{MT2020_1}), while the parameterization for {\em sink} equilibria can be also applied to constructing trapping regions.}

\begin{figure}[h!]\em
\centering
\includegraphics[width=10cm]{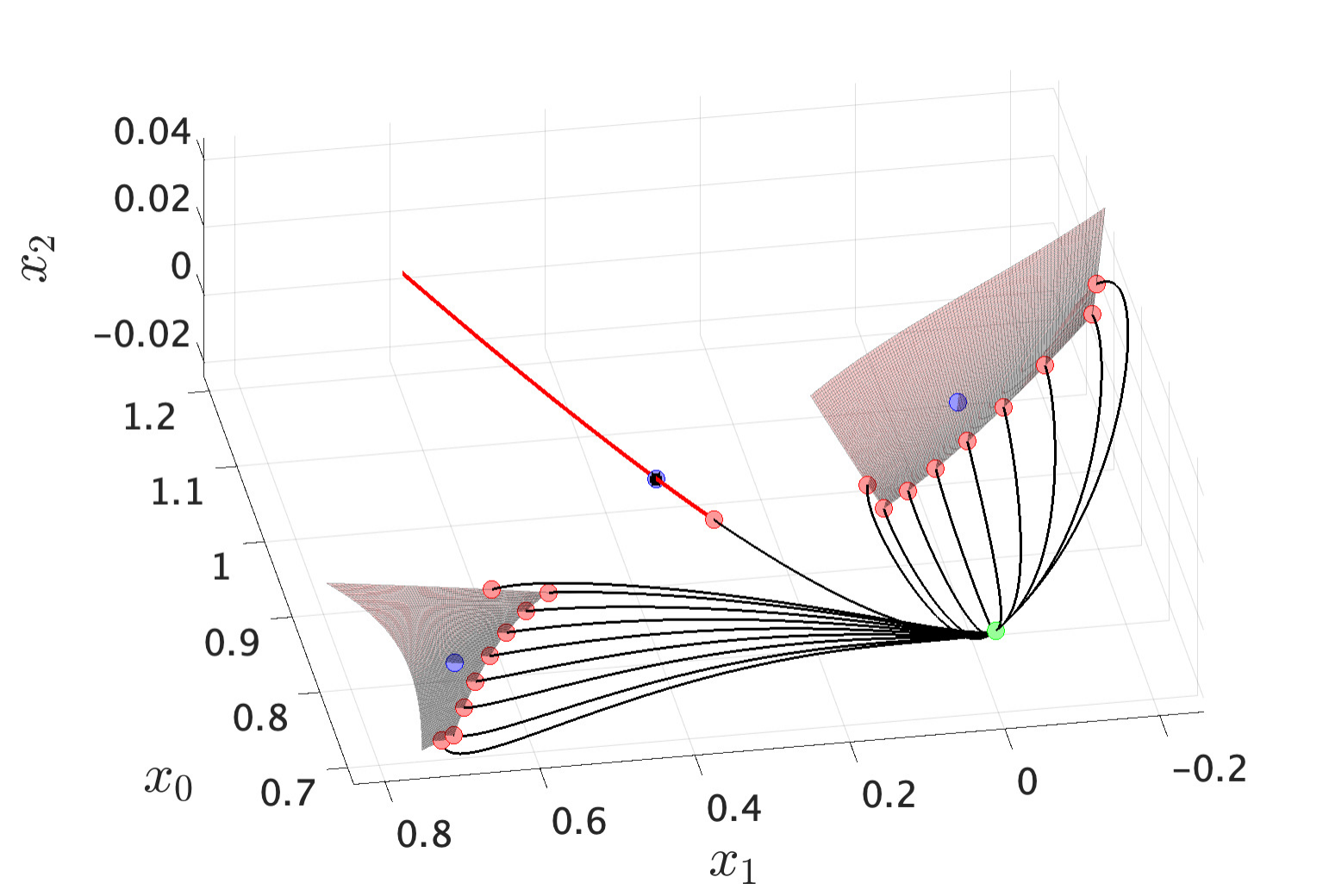}
\caption{The rigorously computed trajectories on global stable manifolds \KMa{of} hyperbolic equilibria for (\ref{desing-demo2}).} 
\flushleft
Local stable manifolds for (\ref{desing-demo2}) colored by pink and red are validated by the parameterization method, Figure \ref{fig-Nagumo-stable}.
The \KMa{green} dot denotes the (bounded) source equilibrium $p_b$.
\label{fig-Nagumo-global}
\end{figure}

\section{Example 3: presence of separatrix involving blow-ups}

\label{sec:demo3}
The final example is 
\begin{equation}
\label{KK}
\begin{cases}
u' = u^2 - v, & \\
v' = \frac{1}{3}u^3 - u. &
\end{cases}
\end{equation}
The present vector field originally comes from the Keyfitz-Kranser model \cite{KK1990} demonstrating a non-trivial example of system of conservation laws including {\em singular shock waves}.
See \cite{KK1990} or references therein for details.
A brief introduction of the model is also shown in \cite{MT2020_2}.
Our purpose here is to validate blow-up solutions for (\ref{KK}) as well as bounded heteroclinic connections among bounded equilibria towards the global phase portrait.
The present study unravels a significant characteristic of saddle-type blow-up solutions, which shall be called a {\em blow-up separatrix}.
\par
Firstly, a direct calculation yields the following.
\begin{lem}
The vector field (\ref{KK}) is asymptotically quasi-homogeneous of type $(1,2)$ and order $2$.
\end{lem}
Note that (\ref{KK}) is {\em not quasi-homogeneous}.
On the other hand, the system (\ref{KK}) possesses the symmetry
\begin{equation}
\label{KK-symmetry}
(t, u, v) \mapsto (-t, -u, v).
\end{equation}
Namely, if $(u(t), v(t))$ is a solution to (\ref{KK}), then so is $(-u(-t), v(-t))$.
This property is used to understand the global phase portrait of (\ref{KK}) including infinity.
\par
To study the dynamics at infinity, we introduce the quasi-parabolic compactification of type $(1,2)$ given by
\begin{equation*}
u = \frac{x_1}{1-p(x)^4},\quad v = \frac{x_2}{(1-p(x)^4)^2},\quad p(x)^4 = x_1^4 + x_2^2.
\end{equation*}
Then the corresponding desingularized vector field $g$ is given by the following: 
\begin{equation}
\label{KK-desing}
\begin{cases}
\dot x_1 = g_1(x) \bydef (x_1^2 - x_2) H_1(x) - x_1H_2(x) & \\
\dot x_2 = g_2(x) \bydef\left\{ \frac{1}{3}x_1^3  - (1 - p(x)^4)^2 x_1 \right\} H_1(x) - 2x_2 H_2(x), &
\end{cases}
\end{equation}
where $\dot {} = \frac{d}{d\tau}$ and
\begin{equation*}
H_1(x) = \frac{1}{4}\left\{ 1 + 3p(x)^4\right\},\quad H_2(x) = x_1^3 (x_1^2 - x_2) + \frac{x_2}{2}\left\{ \frac{1}{3}x_1^3 - (1 - p(x)^4)^2 x_1\right\}.
\end{equation*}
Fortunately, we know that {\em all equilibria (including the origin) are hyperbolic} and hence we do not need additional desingularization.
Detailed information of our targeting equilibria are the following:
\begin{itemize}
\item The origin $p_0 = (x_1, x_2) = (0,0)$, which is {\bf saddle}.
\item A bounded equilibrium $p_b^+ = (x_1, x_2) \approx (0.7328506362011802, 0.5370700549804747)$, which is {\bf source}.
\item A bounded equilibrium $p_b^- = (x_1, x_2) \approx (-0.7328506362011802, 0.5370700549804747)$, which is {\bf sink}.
\item Equilibrium on the horizon $p_{\infty, s}^{\pm} = (x_1, x_2) \approx (\pm 0.8861081289780320, 0.6192579489210105)$, which are {\bf saddle}.
\item Equilibria on the horizon $p_{\infty}^{\pm} = (x_1, x_2) \approx (\pm 0.989136995894977, 0.206758557005180)$.
The point $p_{\infty}^+$ is {\bf sink}, while $p_{\infty}^-$ is {\bf source}.
\end{itemize}
\KMa{Sample (non-rigorous) numerical computations} indicate that there is a chain of global trajectories connecting $p_0$ and $\KMa{p_b^+}$, and $\KMa{p_b^+}$ and $\KMa{p_{\infty, s}^+}$, respectively. 
The numerically computed global phase portrait including \KMa{the horizon} is shown in Figure \ref{fig-KK-global}.
The figure indicates that the whole phase space is separated into two subdomains by a heteroclinic chain among equilibria, including those on the horizon.

\begin{figure}[htbp]\em
\centering
\includegraphics[width=7.0cm]{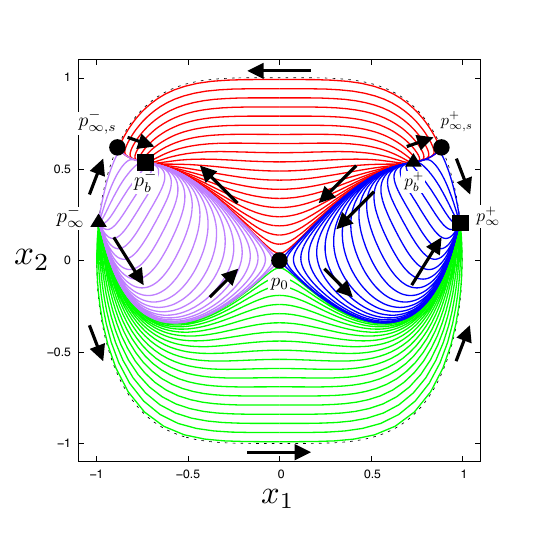}
\caption{A global phase portrait of (\ref{KK-desing}) through rough numerical simulations. 
Black squares, circles and triangles denote sink, saddle and source equilibria, respectively.
Note that all objects here are obtained by (non-rigorous) numerical integration of (\ref{KK-desing}).
The flow \KMa{directions are shown by black arrows}.
The boundary of the collection of curves \KMa{(dotted curve)} is the horizon $\mathcal{E}$.
\KMa{The whole region $\overline{\mathcal{D}}$ is separated into four regions; points admitting global-in-time trajectories (red), points admitting blow-up only in positive time direction (blue), points admitting blow-up only in negative time direction (purple), and points admitting blow-up both in positive and negative time directions (green).}
Chain of connecting orbits, some of which correspond to saddle-type blow-up solutions.}
\label{fig-KK-global}
\end{figure}

\begin{rem}
Here we have chosen the parabolic-type compactification in the present argument for the following reasons.
First, our objective here is the {\em global} phase portrait for (\ref{KK}), which is insufficient to study only one local chart, namely directional compactifications.
The change of coordinates by numerics (both in rigorous and non-rigorous sense) requires unnecessary and difficult tasks.
Second, Poincar\'{e}-type compactifications are inappropriate to study (\ref{KK}) including dynamics at infinity, \KMa{because} (\ref{KK}) is \KMa{quasi-homogeneous only in the asymptotic sense}, and the application to Poincar\'{e}-type compactifications to such a system cause the loss of regularity of the desingularized vector field \KMa{on the horizon}, as mentioned in \KMa{Section \ref{sec:choice-compactification}}.
\end{rem}

One of our main goals here is to construct the chain, mainly connecting orbits among $\{p_{\infty, s}^+, p_b^+, p_0\}$.
Like in the previous examples, the local stable manifold $W^s_{\rm loc}(p_{\infty, s}^+)$ of the saddle $p_{\infty, s}^+$ on the horizon can be validated by the parameterization method. 
Validated local stable manifolds of $p_{\infty, s}^+$ as well as $p_0$ are shown in Figure \ref{fig-1d-stable-manifold-ex3}.
These are validated through the parameterization method in the same way as Sections \ref{sec:demo1} and \ref{sec:demo2}.
We omit the detailed implementation of the method applied to the present problem \KMa{because} the basic idea is identical, while we need lengthy calculations of terms we should enclose.
\par
We then extend the manifold inside $\mathcal{D} \equiv \{p(x) < 1\}$ by the rigorous integration of (\ref{KK-desing}).
According to numerical simulations (Figure \ref{fig-KK-global}), \KMa{$W^s_{\rm loc}(p_{\infty, s}^+)$} is connected to the source $p_b^+$. 
\KMa{Rigorous} integration of (\ref{KK-desing}) in backward time direction provide the computer-assisted validation of the connecting orbit from $p_{\infty, s}^+$ to $p_b^+$ \KMa{by constructing a trapping region of $p_b^+$ in backward time, which is a standard techniques for validating global-in-time trajectories and applied in e.g. \cite{MT2020_1}.}
On the other hand, we have another bounded equilibrium; the origin $p_0$.
Eigenvalue validation indicates that $p_0$ is a saddle, and the global trajectory connecting the source $p_b^+$ and the origin $p_0$ is also validated by extending the local stable manifold $W^s_{\rm loc}(p_0)$ of $p_0$ via the parameterization and the rigorous integration of (\ref{KK-desing}) in backward time direction.
By symmetry, we obtain the chain of connecting orbits among the points $\{p_{\infty, s}^\pm, p_b^\pm, p_0\}$.
Note that all these points are validated with rigorous errors through the parameterization method.
Also note that the connecting orbit between $p_{\infty, s}^\pm$ exists through the fact that the horizon $\mathcal{E}$ is invariant and there are no equilibria between them (cf. \cite{Mat2018}).


As a consequence, an invariant closed curve consisting of connecting orbits among equilibria $\{p_{\infty, s}^\pm, p_b^\pm, p_0\}$ is constructed, as indicated in Figure \ref{fig-KK-global}, with computer-assisted proof.
The well-known Jordan's Closed Curve Theorem indicates that the invariant closed curve decomposes the phase space $\overline{\mathcal{D}}$ into two regions\KMa{\footnote{
\KMa{Numerically observed phase portrait in Figure \ref{fig-KK-global} implies the existence of chains of connecting orbits providing a finer decomposition of the phase space. 
But we omit such a precise decomposition because the process is basically identical and the essential consequence is similar.}
}}. 
In the sequel, we study the nature of solutions through points on these separated regions from the viewpoint of blow-up behavior.

\begin{figure}[htbp]\em
\centering
\includegraphics[width=7.0cm]{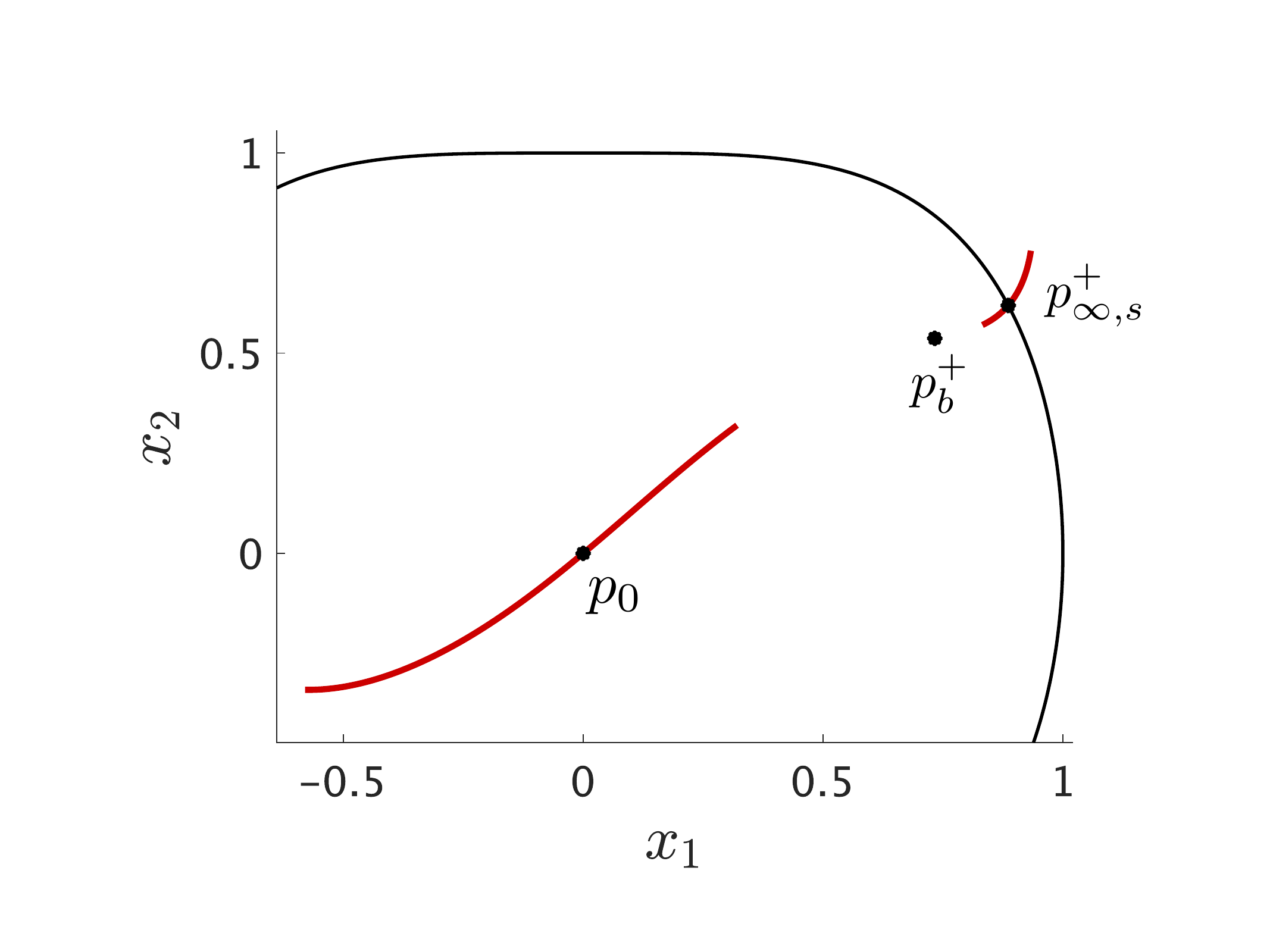}
\caption{Local stable manifolds $W^s_{\rm loc}(p_{\infty, s}^+)$ and $W^s_{\rm loc}(p_0)$ for (\ref{KK-desing}).} 
\flushleft
Black dots are equilibria $p_0, p_b^+$ and $p_{\infty, s}^+$ from the left, respectively.
Red curves are validated local stable manifolds with rigorous error bounds $\|P-P^{(N)}\|_\infty \leq r = 5.171\times 10^{-14}$ for $p_0$ and $\|P-P^{(N)}\|_\infty \leq r = 1.381\times 10^{-10}$ for $p_{\infty, s}^+$, respectively.
\KMa{The black curve denotes the horizon $\mathcal{E}$. 
Although validated local stable manifolds are characterized for (\ref{KK-desing}) which themselves make sense outside the horizon also, they make sense inside the horizon as the corresponding objects to the original vector field (\ref{KK}).}
In both validations, the approximation order $N$ is chosen as $N=100$.
\KMa{Because} $p_b^+$ is source, it does not admit a non-trivial stable manifold.
\label{fig-1d-stable-manifold-ex3}
\end{figure}

\subsection{Blow-up time computation}
\label{sec:blow_up_time_3}
The maximal existence time of the solution $(y_1(t), y_2(t))$ for the original vector field is given as follows (see \KMa{(\ref{time-desing-para-integral})} and \cite{MT2020_1}):
\begin{align*}
t_{\max} &= \int_0^\infty \frac{1}{4}\left\{ 1+ 3\left(x_1(\tau)^4 + x_2(\tau)^2 \right)\right\} (1-x_1(\tau)^4 - x_2(\tau)^2 ) d\tau.
\end{align*}
Let $x_\ast = (x_{\ast,1}, x_{\ast,2})\in \mathcal{E}$ be a {\em saddle} equilibrium.
Note that $x_{\ast,1}^4 + x_{\ast,2}^2 = 1$ by definition of the present parabolic-type compactification.
%

As in the previous case, let $P$ be a parameterization whose image of $B^{n_s}$ determines the local stable manifold of $x_\ast$ such that $P(0) = x_\ast$.
$P$ is assumed to have a power series expression (\ref{param_expression}) satisfying $a_{\bf 0} = x_\ast$.
Assume that the trajectory $\{(x_1(\tau), x_2(\tau))\}$ is included in $W^s_{\rm loc}(x_\ast)$ for the desingularized vector field.
%
In the present case, $n=2$, and we consider only the case $n_s = 1$.
Calculations below are slightly simplified by introducing $u = e^{\lambda \tau}\theta_0$, where $\lambda$ be the stable eigenvalue at $x_\ast$.
Indeed, we have
\begin{equation*}
x(\tau) = P(e^{\lambda \tau}\theta_0) = P(u) = \sum_{j \geq 0}a_j u^j \in \mathbb{R}^2,\quad u \in \mathbb{R},\quad a_j \equiv ((a_1)_j, (a_2)_j)^T \in \mathbb{R}^2.
\end{equation*}
Letting $a_i = \{(a_i)_j\}_{j\geq 0}$ for $i=1,2$, we have
\begin{align*}
t_{\max} &=
 \int_{\theta_0}^0 \frac{1}{4}\left( 1+ 3( P_1^4(u) + P_2^2(u) )\right)
	 \left( 1- P_1^4(u) - P_2^2(u) \right) \frac{du}{\lambda u}\\
	&= -\int_0^{\theta_0} \frac{1}{4} \left( 1+ 3\sum_{j \geq 0} \left(  ( a^4_1  )_{j}  + ( a^2_2 )_j \right) u^{j} \right) 
\left( 1- \sum_{j \geq 0} \left( ( a^4_1)_{j}  + (a^2_2)_j \right) u^{j} \right) \frac{du}{\lambda u}.
\end{align*}

Using that
\[
P_1^4(u) + P_2^2(u) = \sum_{j \geq 0} \left(  (a^4_1)_{j}  + ( a^2_2 )_j \right) u^{j} = 1 +  \sum_{j \geq 1} \left( ( a^4_1)_{j}  + (a^2_2)_j \right)u^{j} ,
\]
we have the following exact formula for $t_{\max}$:
\begin{align}
\notag
t_{\max} &= \int_0^{\theta_0} \left( 1+ \frac{3}{4}\sum_{j \geq 1} \left(  ( a_1^4 )_{j}  +  (a_2^2)_j \right) u^{j} \right)
\left( \sum_{j \geq 1} \left( (a_1^4)_{j}  +  (a_2^2 )_j  \right) u^{j} \right)\frac{du}{\lambda u} \\
\notag
	&= \int_0^{\theta_0} \left( \sum_{j \geq 1} \left( (a_1^4)_j + (a_2^2)_j \right) u^{j-1} + \frac{3}{4} \sum_{j \geq 2} \left( (a_2^4)_j + 2(a_1^4 \ast a_2^2 )_j + (a_1^8)_j \right) u^{j -1} \right) \frac{du}{\lambda} \\
\label{tmax_exact_demo3}
	&= \frac{1}{\lambda} \left( \sum_{j \geq 1} \left( (a_1^4)_j + (a_2^2)_j \right) \frac{\theta_0^j}{j} + \frac{3}{4} \sum_{j \geq 2} \left( (a_2^4)_j + 2(a_1^4 \ast a_2^2 )_j + (a_1^8)_j \right) \frac{\theta_0^j}{j} \right).
\end{align}

\subsection{Chain of connecting orbits as separatrix}
\label{sec:dep_blowup_time}

In what follows, we discuss a \KMa{global nature of saddle-type blow-up solutions in dynamical systems}.
In Figure \ref{fig-KK-global}, we numerically observe that the compactified phase space is separated into \KMa{four} domains, one of which consists of points whose trajectories tend to the origin as $\tau \to \pm \infty$, while another consists of points whose trajectories tend to equilibria on the horizon as \KMa{either both $\tau \to \pm \infty$, or only $\tau \to -\infty$ or $\tau \to + \infty$}.
Namely, the latter \KMa{sets consist of} initial points \KMa{which solutions through these points} blow up in finite times in the original coordinate.
A significant importance of this observation is that these \KMa{four} domains are divided by \KMa{sequences} of \KMa{trajectories} {\em including \KMa{ones inducing blow-up solutions}}.
In particular, saddle-type blow-up solutions themselves or bounded \KMa{global-in-time} trajectories connecting blow-up solutions can locally divide initial points into the \KMa{above domains}.
\par
\bigskip
\KMa{As demonstrated in Section \ref{sec:demo2_extension} and mentioned previously, connecting orbits between equilibria can be validated through the parameterization, extension of local (un)stable manifolds and construction of trapping regions (namely, local stable manifolds of sink equilibria).
In two-dimensional systems like (\ref{KK}), the detailed nature of global dynamics can be \KMa{easily} considered by studying asymptotic behavior of solutions through neighborhoods of connecting orbits.
Moreover, our validated connecting orbits involve blow-up solutions, and the characteristic value $t_{\max}$ is associated to all points on validated connecting orbits and solutions close to them.
Here we study connecting orbits involving hyperbolic saddles on the horizon and global-in-time solutions for the desingularized vector field, and the corresponding characteristics in the original vector field, yielding significantly different nature of asymptotic behavior.
In particular, we investigate the following issues:
\begin{itemize}
\item Dependence of blow-up characterizations on magnitude of initial points.
\item Continuous dependence of $t_{\max}$ on initial points.
\end{itemize}
}
\KMa{(Local) stable manifolds of saddle equilibria locally separate neighborhoods of the equilibria, as well as those asymptotic behavior, unlike sink and source equilibria.
The first issue is then equivalent to a non-trivial question here is {\em whether such a separation around the horizon can significantly change the asymptotic behavior of solutions for the original vector field}.
}
\par
\par
\KMa{Now we have a hyperbolic saddle on the horizon $p_{\infty, s}^+$, a bounded source $p_b^+$ and the origin $p_0$ as a hyperbolic saddle. 
As shown in Figures \ref{fig-KK-global} and \ref{fig-1d-stable-manifold-ex3}, local stable manifolds of $p_{\infty, s}^+$ and $p_b^+$ are validated through the parameterization method and extended through the integration of (\ref{KK-desing}) like connecting orbits in Figure \ref{fig-Nagumo-global}.}
Let $C_{sep}$ be the union of validated connecting orbits:
\begin{equation}
\label{Csep}
\KMa{C_{sep} := \left(\overline{W^u(p_b^+)\cap W^s(p_0)}\right) \cup \left(\overline{W^u(p_b^+)\cap W^s(p_{\infty, s}^+)}\right).}
\end{equation}

\subsubsection{Dependence of blow-up characterizations on magnitude of initial points}

\KMa{First we consider the following issue.}
\begin{prob}
Does the blow-up behavior depend on magnitudes of initial points ?
\end{prob}
In arguments of blow-up criteria, \KMa{magnitudes (equivalently, norms)} or values of several functionals of initial points are typically concerned for determining whether or not the corresponding solutions blow up.
In many cases, there are mathematical arguments showing that initial points whose norms or associated functionals are sufficiently large induce finite-time blow-up.
On the other hand, there are also several mathematical results of blow-up behavior which do not mention the magnitude of initial \KMa{points}.
The aim of the present issue here is to reveal a qualitative characterization of \KMa{asymptotic behavior around saddle-type blow-up solutions}, which partially gives an answer to the above question.
\par
Now we choose two pairs of initial points.
One pair is located close to $p_{\infty, s}^+$, while another pair is located close to the origin.
In both pairs, two initial points are located at the opposite side to each other across $C_{sep}$.
More precisely, the former pair is chosen close to $(x_1, x_2) = (0.83, 0.53)$, while the latter pair is chosen close to $(x_1, x_2) = (0.32, 0.32)$.
The corresponding points {\em in the original coordinate} are approximately
\begin{equation}
\label{points-original-demo3}
(u,v) = (3.39444993, 4.69501202)\quad\text{ and }\quad (u,v) =(0.36072017, 0.40662201),
\end{equation}
respectively.
\KMa{Details are drawn in Figure \ref{fig-KK-separatrix}.}
The methodology shown in Section \ref{sec:methodology} is applied to validating \KMa{global-in-time} trajectories for (\ref{KK-desing}) through each point, showing that the \KMa{asymptotic behavior} of trajectories are completely separated for both pairs of initial points.
More precisely, 
\KMa{
\begin{itemize}
\item across saddle-type blow-up solutions, the asymptotic behavior of solutions as those for (\ref{KK}) significantly change, one of which attains $t_{\max} = \infty$, while another attains $t_{\max} < \infty$. 
\end{itemize}
Moreover, we also observe that
\begin{itemize}
\item such a nature can be observed {\em even near the origin}, where another connecting orbit between $p_0$ and $p_b^+$ locally separates the phase space and is connected to the saddle-type blow-up solution generated by $p_{\infty, s}^+$.
\end{itemize}}
See Figures \ref{fig-KK-separatrix} and \ref{fig-KK-separatrix-initial-zoom}.
\KMa{From the above observation, we can say that {\em the magnitude of initial points is not always essential to determine the blow-up behavior}.
In other words, the chain $C_{sep}$ plays a role in the {\em separatrix} dividing global-in-time solutions and blow-up solutions.}
The key point is that the chain $C_{sep}$ including the saddle on the horizon locally separates the phase space, and that there are sinks $p_{\infty}^+ \in \mathcal{E}$ and $p_b^- \in \mathcal{D}$ inducing global-in-time solutions for (\ref{KK-desing}) approaching to them.
The significant change of solutions in the original vector field is then responsible for the existence of saddle-type blow-up solutions, in particular $C_{sep}$, sinks on the horizon and another sinks on the other side of $C_{sep}$.
Nevertheless, saddle-type blow-up solutions themselves play a role in the trigger of the above nature.
\KMa{Finally note that the present observation can be applied to other dynamical systems like (\ref{Nagumo-at-infty}), where the global extension of stable manifolds characterizing saddle-type blow-up solutions is validated in Section \ref{sec:demo2_extension} (cf. Figure \ref{fig-Nagumo-global}).
}

\begin{figure}[htbp]\em
	\begin{minipage}{1.0\hsize}
		\centering
		\includegraphics[width=8.5cm]{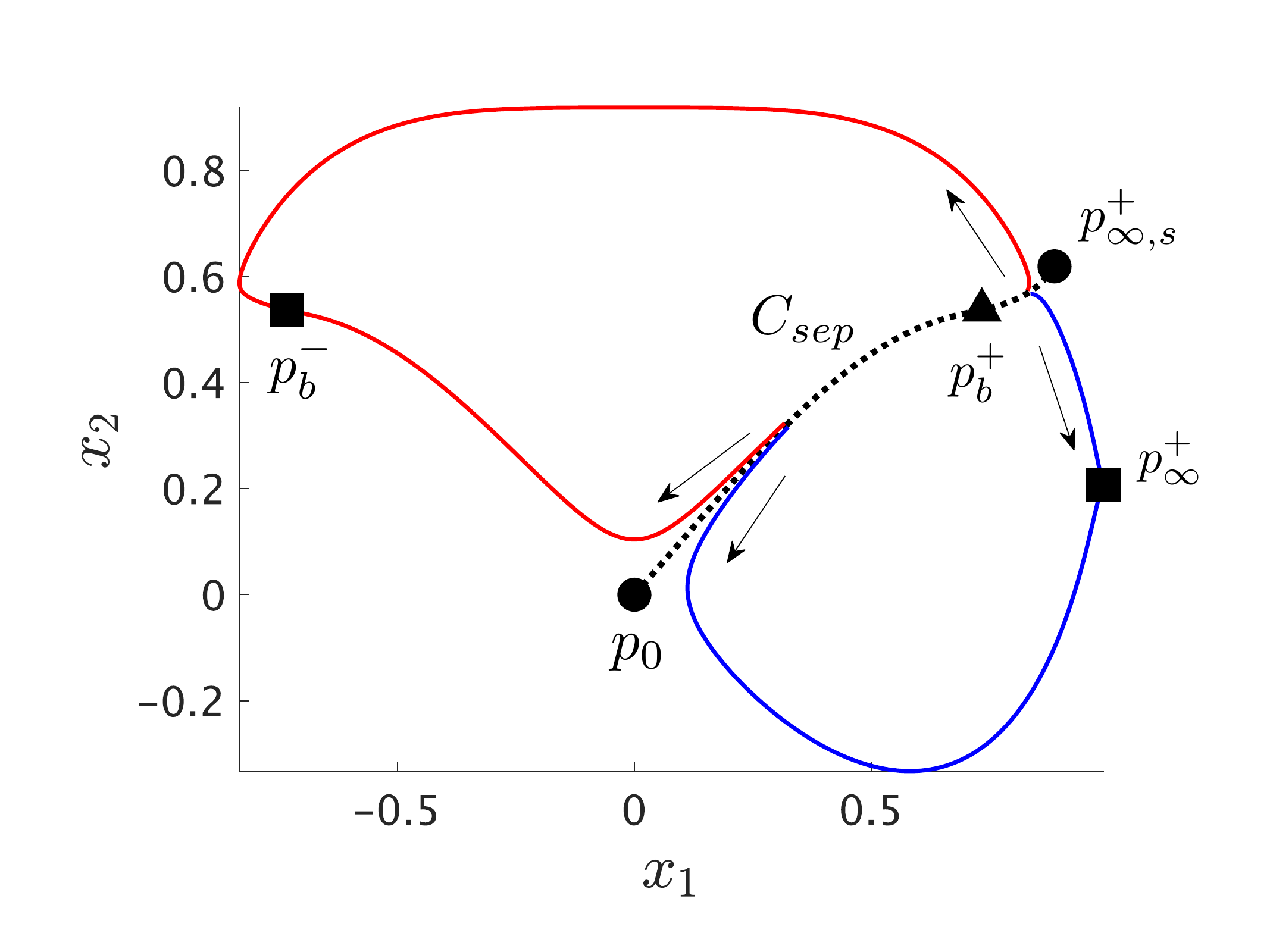}
	\end{minipage}
	\caption{The chain $C_{sep}$ and asymptotic behavior of solutions near $C_{sep}$}
	\flushleft
	A nature of the chain $C_{sep}$ defined by (\ref{Csep}), the collection of black (solid and dotted) curves, is drawn.
	Red curves correspond to \KMa{global-in-time} solutions for (\ref{KK}), while blue curves correspond to blow-up solutions for (\ref{KK}).
	Colors correspond to Figure \ref{fig-KK-global}.
	Initial points are indeed separated by $C_{sep}$, no matter how large they are.
	See also Figure \ref{fig-KK-separatrix-initial-zoom}.
	\label{fig-KK-separatrix}
\end{figure}

\begin{figure}[htbp]\em
	\begin{minipage}{0.5\hsize}
		\centering
		\includegraphics[width=6.0cm]{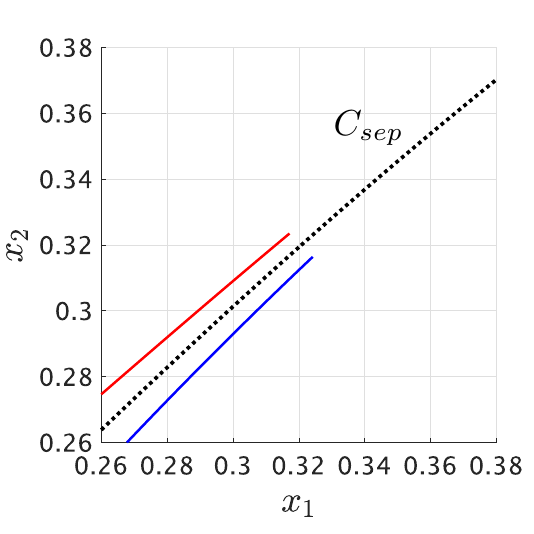}
	(a)
	\end{minipage}
	\begin{minipage}{0.5\hsize}
		\centering
		\includegraphics[width=6.0cm]{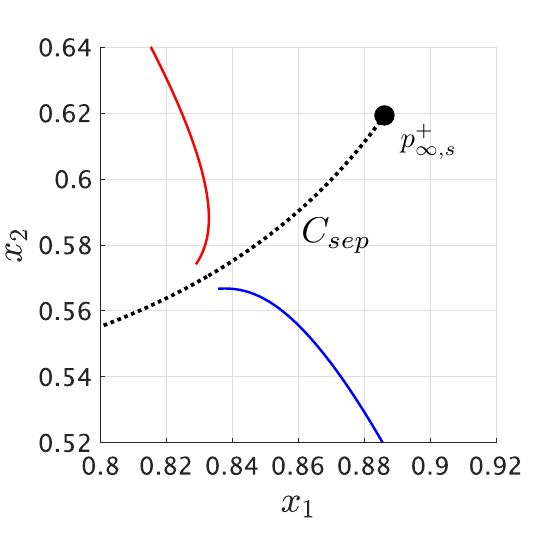}
	(b)
	\end{minipage}
\caption{\KMa{Enlarged view} of initial points in Figure \ref{fig-KK-separatrix}}
	\flushleft
	Endpoints of colored curves in (a) denote initial points of the \KMa{global-in-time} solution (red) and the blow-up solution (blue) going to the direction towards the origin, respectively, while those in (b) denote initial points of the \KMa{global-in-time} solution (red) and the blow-up solution (blue) going to the direction towards the saddle $p_{\infty, s}^+$, respectively.
\label{fig-KK-separatrix-initial-zoom}
\end{figure}

\subsubsection{Continuous dependence of $t_{\max}$ on initial points}

\KMa{Next we investigate the continuous dependence of $t_{\max}$ on initial points across $C_{sep}$ given in (\ref{Csep}).}
Here we consider a line segment $\ell$ which is transverse to $C_{sep}$.
See Figure \ref{fig-KK-zoom}.
The segment $\ell$ is chosen so that $C_{sep}$ and $\ell$ are orthogonal to each other at the boundary $p_{0,s}$ of $W_{\rm loc}^s(p_{\infty, s}^+)$ validated by the parameterization method (cf. Figure \ref{fig-1d-stable-manifold-ex3}).
The boundary $p_{0,s}$ of $W_{\rm loc}^s(p_{\infty, s}^+)$ in $\mathcal{D}$ is then uniquely determined as the intersection $C_{sep}\cap \ell \equiv \{p_{0,s}\}$.
Our problem here is then stated as follows.

\begin{prob}
Does the blow-up time vary continuously on $\ell$ ?
If not, study whether $t_{\max}$ is discontinuous only in each side of $C_{sep}$ on $\ell$, or discontinuous in both sides of $C_{sep}$.
\end{prob}

\KMa{Indeed, the concrete dependence of $t_{\max}$ cannot be unraveled unless explicit formulae (or both lower and upper bounds) for $t_{\max}$ as functions of initial points are obtained.
Our present methodology enables us to unravel this hidden nature in a reasonable way.}
\par
To study the above problem, the following steps are operated.
\begin{enumerate}
\item Set a line segment $\ell$ transverse to the chain $C_{sep}$.
\item Compute the blow-up time \KMa{$t_{\max}$} of the solution through $\{p_{0,s}\} \equiv C_{sep}\cap \ell$.
\item Choose several points on $\ell$ in the blue region, shown in Figure \ref{fig-KK-zoom}, and validate blow-up times through these points.
\item Plot all validated blow-up times and study the distribution.
\item \KMa{Investigate the distribution provides continuous dependence on initial points}.
\end{enumerate}

\begin{figure}[htbp]\em
	\begin{minipage}{1.0\hsize}
		\centering
		\includegraphics[width=8.0cm]{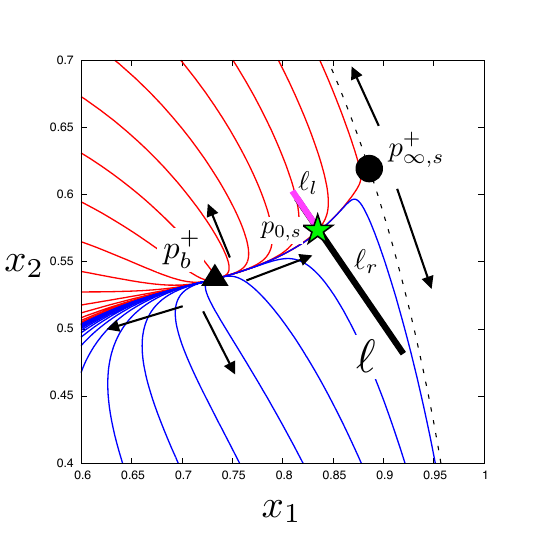}
	\end{minipage}
	\caption{\KMa{The enlarged view} of Figure \ref{fig-KK-global}: choice of the segment $\ell$.} 
	\flushleft
	The black ball is the saddle $p^+_{\infty, s}$, while the black triangle is the source $p_b^+$.
	The black dotted curve is the horizon $\mathcal{E}$.
	The red curve connecting $p^+_{\infty, s}$ and $p_b^+$ is a component of the \KMa{chain} $C_{sep}$.
	Recall that trajectories through points in the red region correspond to \KMa{global-in-time} solutions for (\ref{KK}), while trajectories through points in the blue region correspond to blow-up solutions for (\ref{KK}).
	\KMa{A line $\ell$ is chosen so that it is transverse to $C_{sep}$ and is divided into two segments $\ell_l$ (purple line) and $\ell_r$ (black line) across $C_{sep}$ and it is orthogonal to $C_{sep}$ at $p_{0,s}$ mentioned below}.
	The intersection point $\{p_{0,s}\} \equiv \ell \cap C_{sep}$ is denoted by the green star.
	\label{fig-KK-zoom}
\end{figure}

\begin{figure}[htbp]\em
\begin{minipage}{0.5\hsize}
\centering
\includegraphics[width=8.0cm]{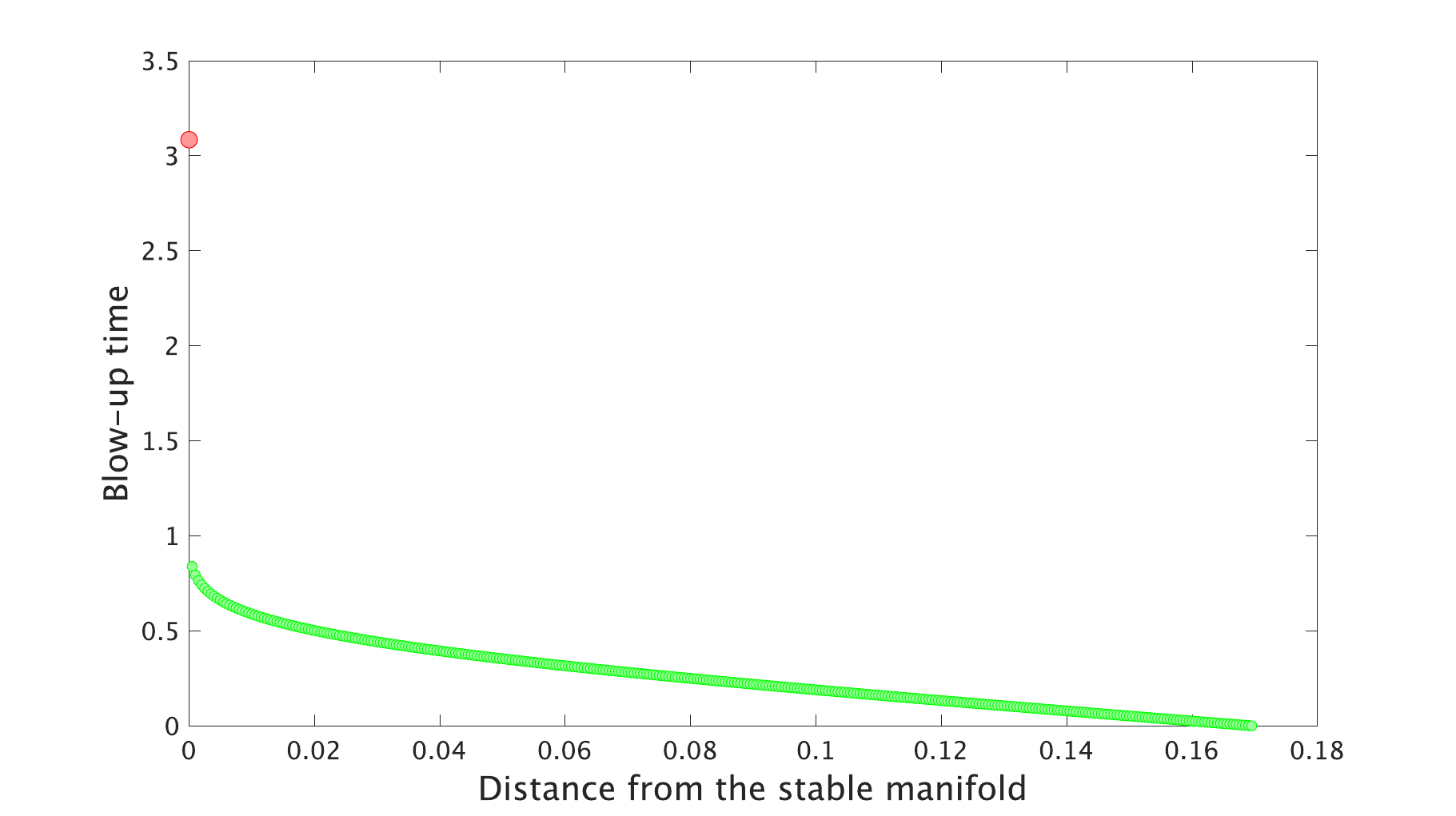}
(a)
\end{minipage}
\begin{minipage}{0.5\hsize}
\centering
\includegraphics[width=8.0cm]{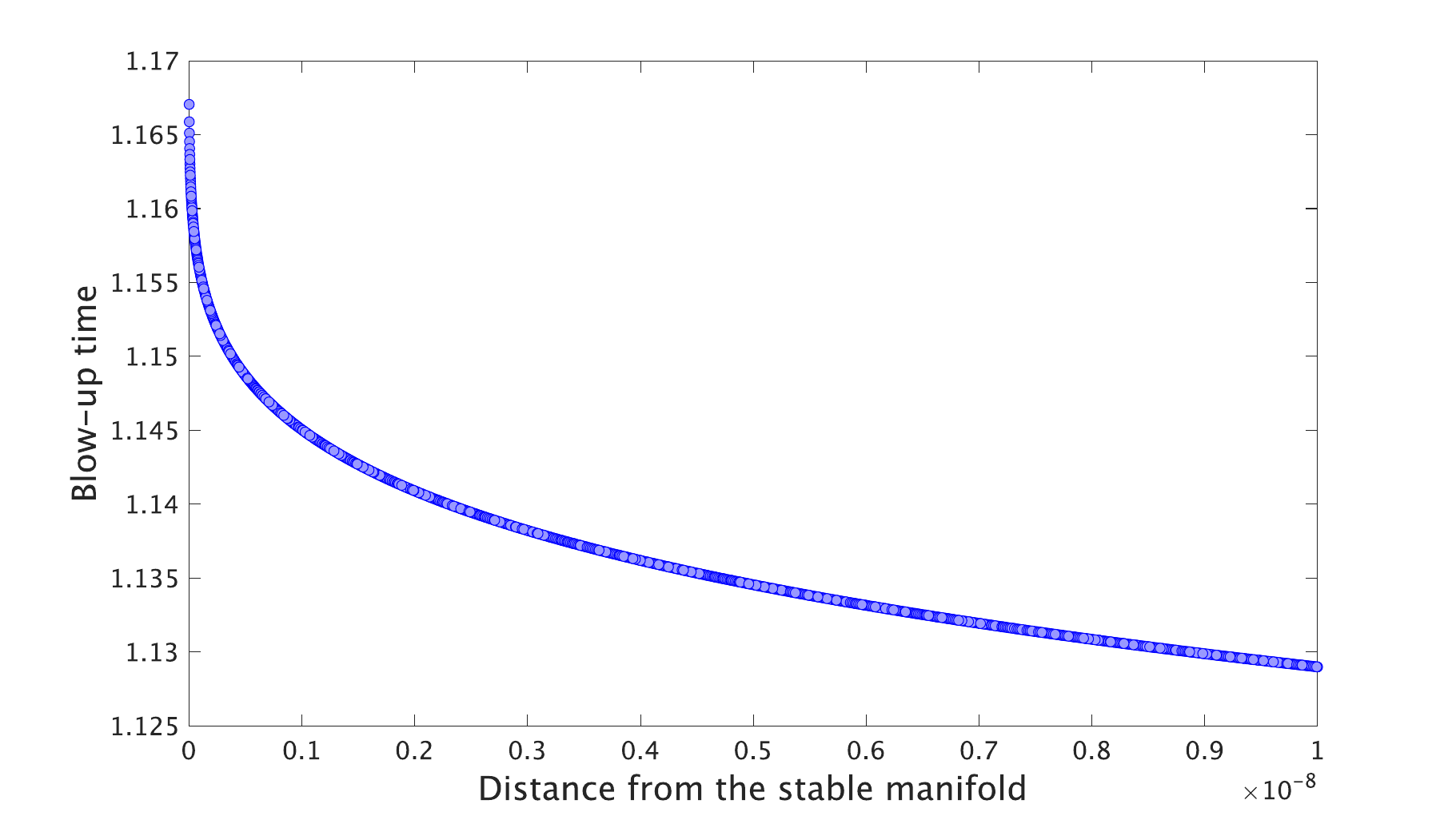}
(b)
\end{minipage}
\caption{Blow-up times of solutions with initial points on $\ell$}
\flushleft
We have totally chosen $10,000$ points on $\ell_r$ for validating $t_{\max}$.
\par
(a) Relationship of points on $\ell_r$ and the blow-up times of solutions through those points.
Horizontal: distance from $p_{0,s}$ on $\ell$. 
Vertical: blow-up time $t_{\max}$ of the corresponding solution.
The value $0$ on the horizontal axis corresponds to $p_{0,s}$. 
The blow-up time $t_{\max} = t_{\max}(p)$ looks discontinuous at $p = p_{0,s}$.
The red point denotes $t_{\max}(p_{0,s})$, while green points denote $t_{\max} = t_{\max}(p)$ at $p\in \KMa{\ell_r} \setminus \{p_{0,s}\}$. 
All plotted blow-up times here \KMa{except} $t_{\max}(p_{0,s})$ have rigorous error bounds less than $3.7235\times 10^{-5}$, while the rigorous error bound of $t_{\max}(p_{0,s})$ is $2.5175\times 10^{-11}$.
\par
(b) \KMa{Enlarged view} of the graph (a) for points within the distance $\leq 1.0\times 10^{-8}$ from $p_{0,s}$.
All plotted blow-up times here \KMa{except} $t_{\max}(p_{0,s})$ have rigorous error bounds less than $1.8288\times 10^{-2}$.
As $p \in \ell_r$ approaches to $p_{0,s}$, $t_{\max}$ significantly increases.
In the present study we do not have validations for $t_{\max}$ associated with points $p\in \ell_r$ within the distance $\leq 1.0\times 10^{-13}$.
\label{fig-blow-up-time-saddle}
\end{figure}

The point $p_{0,s}$ decomposes the line segment $\ell$ into two pieces, denoted by $\ell_l$ and $\ell_r$ consisting of points on $\ell$ in the left side \KMa{(red in Figure \ref{fig-KK-zoom})} and the right side \KMa{(blue in Figure \ref{fig-KK-zoom})} of $p_{0,s}$, respectively.
Our validations, rigorous integrations of (\ref{KK-desing}) in forward time direction, show that all sample points on $\ell_r$ converge to $p_{\infty}^+$ as $\tau \to \infty$, which correspond to a family of sink-type blow-up solutions.
Their validated blow-up times as well as the blow-up time of the solution \KMa{through $p_{0,s}$} are shown in Figure \ref{fig-blow-up-time-saddle} with their rigorous error bounds.
Looking at Figure \ref{fig-blow-up-time-saddle}, the corresponding blow-up times increase as sectional points \KMa{on $\ell_r$} become close to $W^s_{\rm loc}(p_{\infty, s})$\KMa{. }
On the other hand, all points on $\ell_l$ converge to the sink equilibrium $p_b^-$ \KMa{(Figure \ref{fig-KK-global})}. 
\KMa{Because} the preimage of $p_b^-$ under the compactification is bounded, the corresponding solution in the original time-scale exists for all $t \geq 0$.
This fact is easily confirmed by showing that $t_{\max}(p) = \infty$ \KMa{for $p\in \ell_l$}.
These observations show that $t_{\max}$ is discontinuous as a function of points on $\ell$ at $p_{0,s}$ from $\ell_l$.
\par
Next we discuss the continuity of $t_{\max}$ at $p_{0,s}$ on $\{p_{0,s}\}\cup \ell_r$.
Our validations show that 
\begin{equation*}
t_{\max}(p_{0,s}) \in 3.109637008_{391221}^{441572},
\end{equation*}
which is much higher than \KMa{$t_{\max} = t_{\max}(p)$ through $p\in \ell_r$}, according to Figure \ref{fig-blow-up-time-saddle}.
However, $t_{\max} = t_{\max}(p)$ drastically increases as $p\in \ell_r$ approaches to $p_{0,s}$.
At the point $p\in $ with $|p_{0,s} - p| = 1.0\times 10^{-13}$, validation of blow-up solutions did not succeed.
As long as we have validated, we cannot conclude the discontinuity of $t_{\max}$ at $p_{0,s}$ in both sides. 
Nevertheless, we can still conclude that $t_{\max}$ behaves in a singular manner around $p_{0,s}$ where the trajectory approaches to different invariant sets \KMa{as $\tau\to \infty$}.

\begin{rem}
Rigorous enclosures of $t_{\max}$ \KMa{on $W^s(p_{\infty}^+)$, namely sink-type blow-up solutions}, are validated by local Lyapunov functions and rigorous integrations of (\ref{KK-desing}), which are exactly machineries applied in \cite{MT2020_2} \KMa{and hence the detailed validation \KMa{methodology} is omitted}.
The difference of orders of (the worst) rigorous error bounds of $t_{\max}$ on and off $W^s(p_{\infty, s})$ shown in Figure \ref{fig-blow-up-time-saddle} comes from that of the methodology for validating rigorous bounds of $t_{\max}$.
Nevertheless, there is no significant influence on the qualitative tendency of $t_{\max}$ in the present study.
\end{rem}

\KMa{
\begin{rem}[Different choice of $\ell$ can provide different distributions of $t_{\max}$.]
If we choose a line segment $\ell$ across $W^s_{\rm loc}(p_0)$ instead of $W^s_{\rm loc}(p_{\infty, s}^+)$, then $t_{\max}$ at the unique intersection point $W^s_{\rm loc}(p_0)\cap \ell$ is $+\infty$, which provide the different distribution of $t_{\max}$ from Figure \ref{fig-blow-up-time-saddle}.
\end{rem}
}

\begin{rem}[Behavior of $t_{\max}$: a numerical experiment]
We have numerically calculated the behavior of $t_{\max}$ as a function of distance to the stable manifold in Figure \ref{fig-blow-up-time-saddle}-(b).
Let $x$ be the distance of a point $p$ from $p_{0,s}$ in $\ell_r$ and $t_{\max}(x)$ be the corresponding blow-up time.
As far as we have calculated, we could not match $t_{\max}(x)$ by functions of the form $x^a$, $e^{ax}$, $c(\ln x)^a$, and $Cx^a (\ln x)^b$ for constants $a,b,c$.
It is needless to say that this asymptotic form can be different for smaller $x$ and a different choice of $\ell$.
\end{rem}

\KMa{
\subsubsection{Short summary of our observations}

Our observations here are summarized as follows.
\begin{itemize}
\item Blow-up characterizations such as the asymptotic behavior and blow-up times {\em do not always depend continuously on initial points} in the presence of saddle-type blow-up solutions.
\item The blow-up time $t_{\max}$ varies in a singular manner near the chain of connecting orbits involving saddle-type blow-ups, like $C_{sep}$.
\end{itemize}
Note that these features cannot be unraveled only from local information around invariant objects, because local invariant manifolds themselves do not determine the asymptotic behavior of solutions through all points around the manifolds.
In other words, global information of solutions are necessary to investigate this issue.
It should be also noted that the above nature is observed not only by the presence of invariant sets like $C_{sep}$, but also by the presence of another invariant sets like $p_{\infty}^+$ and $p_b^-$, at least one of which is included in the horizon $\mathcal{E}$.
This consequence strongly supports the importance of investigations of global dynamical structure to unravel the significantly different asymptotic behavior of solutions for the original vector field.
Computer-assisted proofs provide a systematic and mathematical rigorous way to investigate such global information of solutions.
Moreover, the presence of saddle-type blow-up solutions provides an easy prediction of the existence of the above nature.}

\section{Concluding remarks}

In this paper, we have shown several characteristics of blow-up solutions for autonomous ODEs which are unstable under perturbations of initial points, referred to as {\em saddle-type blow-up solutions}, with the computer-assisted proofs of their existence and analytic characterization of blow-up times.
\KMa{
Combining compactifications, time-scale desingularizations of vector fields, parameterization of invariant manifolds and their extensions via ODE integrations with computer-assisted proofs, blow-up solutions and their extensions are validated systematically, no matter how stable equilibria on the horizon characterizing these blow-up solutions are.
It should be noted that, as seen in all examples, {\em our methodology does not require a priori information about the existence of blow-up solutions}.
This is a big advantage so that the present methodology can be applied to various dynamical systems and blow-up problems under mild assumptions.
}
\par
\KMa{
Characteristics we have unraveled in the present paper are just examples of intrinsic natures which saddle-type blow-up solutions induce. 
But it is not an easy task to predict the presence of such features theoretically, because these are observed as the composite of multiple structures.
For example, distribution of $t_{\max}$ can be investigated by the combination of an analytic expression of $t_{\max}$ and explicit distribution of local stable manifolds of equilibria on the horizon for desingularized vector fields.
\KMa{As for the separatrix nature among global-in-time solutions and blow-up solutions}, it cannot be characterized without concrete distribution of global-in-time solutions, sink-type and saddle-type blow-up solutions.
Computer-assisted proofs, on the other hand, connect features of explicitly validated objects to extract global nature as the composite of local characteristics, like the above features.
These computation techniques efficiently work to gain insights into blow-up solutions.}

\par
\bigskip
\KMa{
We end this paper by leaving comments about topics involving saddle-type blow-up solutions, which can relate to the present study towards further insights into global nature of blow-up solutions, dynamics at infinity and general finite-time singularities.
}

\subsection{Remarks on \KMa{saddle-type} blow-up solutions in science and engineering}
\label{section-remark-intro}
\KMa{Saddle-type} blow-up solutions can arise in scientific and engineering studies.
We review several preceding studies to assert the importance of \KMa{saddle-type} blow-up solutions, and believe that our present methodology will contribute to unravel the dynamical nature of finite-time singularities involving saddle-type blow-up solutions in the following kinds of problems.

\subsubsection{Singular shock waves}
In the Riemann problem of the systems of {\em conservation laws}
\begin{equation}
\label{conservation}
U_t + f(U)_x = 0
\end{equation}
for some smooth $f: \mathbb{R}^n\to \mathbb{R}^n$, namely the initial value problem of (\ref{conservation}) with 
\begin{equation*}
U(0,x) = \begin{cases}
U_L & x<0,\\
U_R & x>0,\\
\end{cases}\quad \text{ for }\quad U_L, U_R\in \mathbb{R}^n,
\end{equation*}
{\em shock waves} are characterized by locally integrable (weak) solutions with discontinuities with the constraints called {\em jump conditions} or the {\em Rankine-Hugoniot conditions}.
With the assumption of {\em viscous shock criterion}, the Riemann problem is reduced to find connecting orbits of the traveling wave ODE associated with (\ref{conservation}) connecting $U_L$ and $U_R$.
In the 1980s and 1990s, shock waves with a singular nature on the front were observed for a simple system of conservation laws, which are referred to as {\em delta-shocks} or {\em singular shocks}.
Roughly speaking, singular shocks are characterized by shocks with Dirac's delta singularity on the shock front (see e.g., \cite{KSS2003, KK1990, S2007} for precise discussions of delta-shocks and singular shocks). 
A typical feature of singular shocks with the presence of the delta-like singularity is that several constraints in jump conditions are violated\footnote{
In $n$-dimensional systems of conservation laws, jump conditions are characterized by $n$ (non)linear equations.
}, which is referred to as the presence of the {\em Rankine-Hugoniot deficit} of a shock measuring the magnitude of singularity on the shock front.
From the viewpoint of dynamical systems, there is a characterization of singular shocks (e.g., \cite{SSS1993}), showing that {\em singular shocks can consist of a collection of blow-up solutions and \lq\lq invariant sets at infinity''}.
In several concrete problems such as the Keyfitz-Kranser model \cite{KK1990} and the two-phase model \cite{KSS2003}, the geometric singular perturbation theory plays a key role in characterizing singular shocks as a singular perturbation of blow-up connections for the traveling wave problems associated with the original conservation laws with the regularization keeping the self-similarity of waves (well-known as {\em Dafermos regularization}).
Preceding studies with blowing-up (desingularization) of singularities and the geometric singular perturbation theory indicate that singular shocks are characterized by trajectories approaching to {\em normally hyperbolic invariant manifolds}, corresponding to the infinity for appropriately transformed dynamical systems \cite{H2016, S2004}\footnote{
It is also indicated that the Rankine-Hugoniot deficit is measured by trajectories at infinity connecting blow-up solutions \cite{KSS2003}.
When the Rankine-Hugoniot deficit is absent, the corresponding shock wave is characterized in the ordinary sense.
}.
\KMa{We believe that saddle-type blow-up solutions can play key roles in characterizing such singular nature both qualitatively and quantitatively (e.g., Rankine-Hugoniot deficits).}

\subsubsection{Suspension bridge}
The equation of the following form is well studied as a model expressing scientific and engineering phenomena:
\begin{equation}
\label{suspension}
w''''(t) + kw''(t) + f(w(t)) = 0\quad (t\in \mathbb{R}),
\end{equation}
where $k\in \mathbb{R}$ is a parameter and $f$ is a locally Lipschitzian.
This equation arises in the dynamical phase-space analogy of a nonlinearly supported elastic struture \cite{HBT1989} and a model characterizing pattern formations in physical, chemical and biological systems \cite{BS2006}.
See also e.g., \cite{PT2012}.
In \cite{BFGK2011}, a possible finite-time blow-up for the solution of (\ref{suspension}) is discussed with a mild assumption 
\begin{equation*}
f\in {\rm Lip}_{loc}(\mathbb{R}),\quad f(t)t > 0\quad \text{ for every}\quad t\in \mathbb{R}\setminus \{0\}.
\end{equation*}
A fundamental result involving blow-up is that the existence of a blow-up solution $w(t)$ for (\ref{suspension}) as $t\to t_{\max} < \infty$ implies that
\begin{equation}
\label{blow-up-osc}
\liminf_{t\to t_{\max}}w(t) = -\infty\quad \text{ and }\quad \limsup_{t\to t_{\max}}w(t) = +\infty,
\end{equation}
namely a blow-up with oscillation.
Moreover, the existence of the above oscillatory blow-up for (\ref{suspension}) with a specific nonlinearity $f$ is proved.
There are several reports about the relationship between the system (\ref{suspension}) to traveling waves for the the model equation of a {\em suspension bridge}
\begin{equation*}
u_{tt} + u_{xxxx} + \gamma u^+ = W(t,x),
\end{equation*}
proposed by Lazer-McKenna \cite{LM1990}.
According to many preceding works and historical sources, one of the most interesting behaviors for suspension bridges (including the Tacoma Narrow Bridge where was collapsed in November 1940) is the following:
\begin{center}
Large vertical oscillations can rapidly change, almost instantaneously, to a {\em torsional oscillation} (quotation from \cite{GP2011}).
\end{center}
Preceding works involving this catastrophic phenomenon discuss the mechanism of torsional oscillations in detail\footnote{
In \cite{GP2013}, there are several additional comments about the case of London's Millennium Bridge (April 2007) and the Assago metro Bridge in Milan (February 2011).
See the reference papers therein for details about these engineering topics.
}, one of which is considered to be the oscillatory blow-up behavior mentioned above.
It should be noted that there is another direction to the origin of such torsional oscillations.
In \cite{AG2015}, it is explained that {\em internal resonances} can trigger the torsional instability.
\par
Later successive works (e.g., \cite{GP2013}) have reported the qualitative nature of the above blow-up such as infinitely many change of signs before blow-up, vanishing intervals of oscillations several quantitative estimates.
In order to obtain the nature, several growth conditions of $f$ (but generalized under these conditions unlike \cite{GP2011}), restrictions to $k$ and an inequality for derivatives of solution $w$ at an initial time are assumed.
It should be noted that {\em norms of initial points are not essential to characterize the above behavior}.
See \cite{GP2013} for details.
Recently, the first author and collaborators \cite{DALP2015} have characterized the above blow-up nature for particular nonlinearity $f$ in (\ref{suspension}) by constructing a concrete asymptotic form of blow-up profiles and validating a periodic solution with \KMa{computer-assisted proofs}.
In \cite{DALP2015}, it is also validated that the periodic solution for an auxiliary equation is unstable, which indicates that {\em the corresponding blow-up solution is unstable under perturbations of initial points}.
It is thus expected that the blow-up nature which is unstable under perturbations of initial points plays a key role in describing rich and interesting, sometimes catastrophic, scientific and engineering nature.

\begin{rem}
\label{rem-intro-per-blowup}
In \cite{Mat2018}, it is proved that blow-up behavior with wide oscillations like (\ref{blow-up-osc}) can be characterized by periodic orbits at infinity, which is referred to as a {\em periodic blow-up}. 
More precisely, global trajectories on the stable manifold of a hyperbolic periodic orbit \KMa{on the horizon for the desingularized vector field} correspond to blow-up solutions with oscillations whose asymptotic behavior, such as the blow-up rate and the oscillatory nature, are uniquely determined by the order of the original vector field and the periodic orbit \KMa{on the horizon}.
The fundamental machinery for this characterization is the same as that shown in Section \ref{sec:pre_blow_up}.
Arguments in the present paper will also contribute to reveal universal mechanisms of this kind of  blow-up solutions \KMa{which are saddle-type both quantitatively and qualitatively, and their validations.}
\end{rem}

\KMa{\subsubsection{More comments}}
We leave several comments about the link to blow-up behavior arising in the suspension bridge problem. 
As noted\KMa{,} it is proved in \cite{DALP2015} with the computer assistance that there is an unstable hyperbolic periodic orbit $\Gamma = \{w(t)\}$ expressing an asymptotic behavior of blow-up behavior for (\ref{suspension}) with specific $k$ and $f$.
It is then conjectured in \cite{DALP2015} that, for the appropriately transformed dynamics from the problem of the form (\ref{suspension}), {\em the boundary of the basin of attraction of the origin coincides with $W^s(\Gamma)$}. 
A consequence of the conjecture is the existence of a three dimensional manifold which \lq\lq separates" the phase space and for which solutions with \KMa{initial points} taken on one side of the manifold blow-up in finite time while on the other side, solutions converge to the origin.
In the present paper, we have focused on unstable, in particular saddle-type, blow-up solutions which can extract the above nature.
We have revealed here that \KMa{saddle-type} blow-up solutions, even with the simpler asymptotic behavior than \cite{DALP2015}, can separate the phase space so that initial points on one side determine \KMa{global-in-time} solutions, while those on the other side induce blow-up solutions.
\KMa{We have mainly investigated asymptotic behavior of solutions near a chain of connecting orbits for desingularized vector fields including saddles on the horizon, like $C_{sep}$ given in (\ref{Csep}) for (\ref{KK-desing}), and shown that $C_{sep}$ triggers the above significantly different asymptotic behavior among solutions.
In particular, $C_{sep}$ have played a role as a separatrix among solutions for the original vector field.}
We believe that \KMa{such invariant objects} can characterize the \lq\lq boundary" of the basin of attraction mentioned in \cite{DALP2015}.
\par
Note that the above \KMa{object} is characterized only for stationary blow-up \KMa{(Theorems \ref{thm:blowup-dir}, \ref{thm:blowup-Poincare} and \ref{thm:blowup-parabolic})} so far.
On the other hand, a computer-assisted proof of the existence of (un)stable manifolds of hyperbolic {\em periodic orbits} is already established in e.g., \cite{CLM2018}, and the treatment of blow-up solutions involving periodic orbits at infinity is also established in \cite{Mat2018, Mat2019}\KMa{. 
In other words, the same machinery as shown in Section \ref{sec:pre_blow_up} can be applied. }
Going back to the suspension bridge problem, combination of preceding works with the arguments in the present paper can contribute to unravel the nature of blow-up behavior in (\ref{suspension}) only with a few mild assumptions.

\section*{Acknowledgements}
JPL was supported by an NSERC Discovery Grant.
KM was partially supported by Program for Promoting the reform of national universities (Kyushu University), Ministry of Education, Culture, Sports, Science and Technology (MEXT), Japan, World Premier International Research Center Initiative (WPI), MEXT, Japan, JSPS Grant-in-Aid for Young Scientists (B) (No. JP17K14235) and JSPS KAKENHI Grant Number JP21H01001.
AT was partially supported by JSPS KAKENHI Grant Numbers JP18K13453, JP20H01820, JP21H01001.

\bibliographystyle{plain}
\bibliography{blow-up_saddle}

\begin{thebibliography}{10}

\bibitem{AIU2017}
K.~Anada, T.~Ishiwata, and T.~Ushijima.
\newblock A numerical method of estimating blow-up rates for nonlinear
  evolution equations by using rescaling algorithm.
\newblock {\em Japan Journal of Industrial and Applied Mathematics}, pages
  1--15, 2017.

\bibitem{AG2015}
G.~Arioli and F.~Gazzola.
\newblock A new mathematical explanation of what triggered the catastrophic
  torsional mode of the {T}acoma {N}arrows {B}ridge.
\newblock {\em Applied Mathematical Modelling}, 39(2):901--912, 2015.

\bibitem{MR4127962}
B.~Barker, J.D. Mireles-James, and J.~Morgan.
\newblock Parameterization method for unstable manifolds of standing waves on
  the line.
\newblock {\em SIAM J. Appl. Dyn. Syst.}, 19(3):1758--1797, 2020.

\bibitem{BFGK2011}
E.~Berchio, A.~Ferrero, F.~Gazzola, and P.~Karageorgis.
\newblock Qualitative behavior of global solutions to some nonlinear fourth
  order differential equations.
\newblock {\em Journal of Differential Equations}, 251(10):2696--2727, 2011.

\bibitem{BK1988}
M.~Berger and R.V. Kohn.
\newblock A rescaling algorithm for the numerical calculation of blowing-up
  solutions.
\newblock {\em Communications on pure and applied mathematics}, 41(6):841--863,
  1988.

\bibitem{COSY}
M.~Berz and K.~Makino.
\newblock Verified integration of {ODE}s and flows using differential algebraic
  methods on high-order {T}aylor models.
\newblock {\em Reliable Computing}, 4(4):361--369, 1998.

\bibitem{BS2006}
D.~Bonheure and L.~Sanchez.
\newblock Heteroclinic orbits for some classes of second and fourth order
  differential equations.
\newblock In {\em Handbook of differential equations: ordinary differential
  equations}, volume~3, pages 103--202. Elsevier, 2006.

\bibitem{BLM2016}
M.~Breden, J-P. Lessard, and J.D. Mireles-James.
\newblock Computation of maximal local (un) stable manifold patches by the
  parameterization method.
\newblock {\em Indagationes Mathematicae}, 27(1):340--367, 2016.

\bibitem{Bunger:2020aa}
F.~B{\"u}nger.
\newblock A {T}aylor model toolbox for solving {ODE}s implemented in
  {MATLAB/INTLAB}.
\newblock {\em Journal of Computational and Applied Mathematics}, 368:112511,
  2020.

\bibitem{MR1976079}
X.~Cabr{\'e}, E.~Fontich, and R.~de~la Llave.
\newblock The parameterization method for invariant manifolds. {I}. {M}anifolds
  associated to non-resonant subspaces.
\newblock {\em Indiana Univ. Math. J.}, 52(2):283--328, 2003.

\bibitem{MR1976080}
X.~Cabr{\'e}, E.~Fontich, and R.~de~la Llave.
\newblock The parameterization method for invariant manifolds. {II}.
  {R}egularity with respect to parameters.
\newblock {\em Indiana Univ. Math. J.}, 52(2):329--360, 2003.

\bibitem{MR2177465}
X.~Cabr{\'e}, E.~Fontich, and R.~de~la Llave.
\newblock The parameterization method for invariant manifolds. {III}.
  {O}verview and applications.
\newblock {\em J. Differential Equations}, 218(2):444--515, 2005.

\bibitem{CLM2018}
R.~Castelli, J.-P. Lessard, and J.D. Mireles-James.
\newblock Parameterization of invariant manifolds for periodic orbits {(II)}:
  {A} posteriori analysis and computer assisted error bounds.
\newblock {\em Journal of Dynamics and Differential Equations},
  30(4):1525--1581, 2018.

\bibitem{C2016}
C.-H. Cho.
\newblock Numerical detection of blow-up: a new sufficient condition for
  blow-up.
\newblock {\em Japan Journal of Industrial and Applied Mathematics},
  33(1):81--98, 2016.

\bibitem{CHO2007}
C.-H. Cho, S.~Hamada, and H.~Okamoto.
\newblock On the finite difference approximation for a parabolic blow-up
  problem.
\newblock {\em Japan Journal of Industrial and Applied Mathematics},
  24(2):131--160, 2007.

\bibitem{DALP2015}
L.~D'Ambrosio, J.-P. Lessard, and A.~Pugliese.
\newblock Blow-up profile for solutions of a fourth order nonlinear equation.
\newblock {\em Nonlinear Analysis: Theory, Methods \& Applications},
  121:280--335, 2015.

\bibitem{D1960}
J.~Dieudonn{\'e}.
\newblock {\em Foundations of modern analysis}.
\newblock Academic Press, New York, 1960.

\bibitem{D1985}
J.W. Dold.
\newblock Analysis of the early stage of thermal runaway.
\newblock {\em The Quarterly Journal of Mechanics and Applied Mathematics},
  38(3):361--387, 1985.

\bibitem{D1993}
F.~Dumortier.
\newblock Techniques in the theory of local bifurcations: Blow-up, normal
  forms, nilpotent bifurcations, singular perturbations.
\newblock In {\em Bifurcations and Periodic Orbits of Vector Fields}, pages
  19--73. Springer, 1993.

\bibitem{D2006}
F.~Dumortier.
\newblock Compactification and desingularization of spaces of polynomial
  li{\'e}nard equations.
\newblock {\em Journal of Differential Equations}, 224(2):296--313, 2006.

\bibitem{DH1999}
F.~Dumortier and C.~Herssens.
\newblock Polynomial {L}i\'{e}nard equations near infinity.
\newblock {\em Journal of differential equations}, 153(1):1--29, 1999.

\bibitem{DLA2006}
F.~Dumortier, J.~Llibre, and J.C. Art{\'e}s.
\newblock {\em Qualitative theory of planar differential systems}.
\newblock Springer, 2006.

\bibitem{EG2006}
U.~Elias and H.~Gingold.
\newblock Critical points at infinity and blow up of solutions of autonomous
  polynomial differential systems via compactification.
\newblock {\em Journal of mathematical analysis and applications},
  318(1):305--322, 2006.

\bibitem{FM2002}
M.~Fila and H.~Matano.
\newblock Blow-up in nonlinear heat equations from the dynamical systems point
  of view.
\newblock {\em Handbook of dynamical systems}, 2:723--758, 2002.

\bibitem{F1969}
H.~Fujita.
\newblock On the nonlinear equations ${\Delta} u + e^u = 0$ and $\partial v /
  \partial t = {\Delta} v + e^v$.
\newblock {\em Bulletin of the American Mathematical Society}, 75(1):132--135,
  1969.

\bibitem{GV2002}
V.A. Galaktionov and J.-L. V{\'a}zquez.
\newblock The problem of blow-up in nonlinear parabolic equations.
\newblock {\em Discrete \& Continuous Dynamical Systems-A}, 8(2):399, 2002.

\bibitem{GP2011}
F.~Gazzola and R.~Pavani.
\newblock Blow up oscillating solutions to some nonlinear fourth order
  differential equations.
\newblock {\em Nonlinear Analysis: Theory, Methods \& Applications},
  74(17):6696--6711, 2011.

\bibitem{GP2013}
F.~Gazzola and R.~Pavani.
\newblock Wide oscillation finite time blow up for solutions to nonlinear
  fourth order differential equations.
\newblock {\em Archive for Rational Mechanics and Analysis}, 207(2):717--752,
  2013.

\bibitem{G2004}
H.~Gingold.
\newblock Approximation of unbounded functions via compactification.
\newblock {\em Journal of Approximation Theory}, 131(2):284--305, 2004.

\bibitem{GKO2020}
A.~Giraldo, B.~Krauskopf, and H.M. Osinga.
\newblock Computing connecting orbits to infinity associated with a homoclinic
  flip bifurcation.
\newblock {\em J. Comput. Dyn.}, 7(2):489--510, 2020.

\bibitem{GOMEZ_PDESurvey}
J.~G{\'o}mez-Serrano.
\newblock Computer-assisted proofs in {PDE}: a survey.
\newblock {\em SeMA Journal}, pages 1--26, 2018.

\bibitem{MR3706909}
J.L. Gonzalez and J.D. Mireles-James.
\newblock High-order parameterization of stable/unstable manifolds for long
  periodic orbits of maps.
\newblock {\em SIAM J. Appl. Dyn. Syst.}, 16(3):1748--1795, 2017.

\bibitem{Ha2016}
J.~Harada.
\newblock Blowup profile for a complex valued semilinear heat equation.
\newblock {\em Journal of Functional Analysis}, 270(11):4213--4255, 2016.

\bibitem{Ha2017}
J.~Harada.
\newblock Nonsimultaneous blowup for a complex valued semilinear heat equation.
\newblock {\em Journal of Differential Equations}, 263(8):4503--4516, 2017.

\bibitem{H2010}
J.~Hell.
\newblock Conley index at infinity.
\newblock {\em Ph.D. Thesis in Freie Universit\"{a}t Berlin}, 2010.

\bibitem{HV1997}
M.A. Herrero and J.J.L. Vel{\'a}zquez.
\newblock A blow-up mechanism for a chemotaxis model.
\newblock {\em Annali della Scuola Normale Superiore di Pisa-Classe di
  Scienze}, 24(4):633--683, 1997.

\bibitem{H2016}
T.-H. Hsu.
\newblock Viscous singular shock profiles for a system of conservation laws
  modeling two-phase flow.
\newblock {\em Journal of Differential Equations}, 261(4):2300--2333, 2016.

\bibitem{HBT1989}
G.W. Hunt, H.M. Bolt, and J.M.T. Thompson.
\newblock Structural localization phenomena and the dynamical phase-space
  analogy.
\newblock {\em Proceedings of the Royal Society of London. A. Mathematical and
  Physical Sciences}, 425(1869):245--267, 1989.

\bibitem{Immler:2018aa}
F.~Immler.
\newblock A verified {ODE} solver and the {L}orenz attractor.
\newblock {\em Journal of Automated Reasoning}, 61(1):73--111, 2018.

\bibitem{kv}
M.~Kashiwagi.
\newblock kv - {C++} {N}umerical {V}erification {L}ibraries.
\newblock {\tt http://verifiedby.me/kv/}.

\bibitem{Kashi1}
M.~Kashiwagi and S.~Oishi.
\newblock Numerical validation for ordinary differential equations ---
  iterative method by power series arithmetic.
\newblock {\em Proc. 1994 Symposium on Nonlinear theorem and its Applications
  (NOLTA'94 Symposium, 1994.10.7)}, pages 243--246, 1994.

\bibitem{KSS2003}
B.L. Keyfitz, R.~Sanders, and M.~Sever.
\newblock Lack of hyperbolicity in the two-fluid model for two-phase
  incompressible flow.
\newblock {\em DISCRETE AND CONTINUOUS DYNAMICAL SYSTEMS SERIES B},
  3(4):541--564, 2003.

\bibitem{KOCH_ComputeAssisted}
H.~Koch, A.~Schenkel, and P.~Wittwer.
\newblock Computer-assisted proofs in analysis and programming in logic: a case
  study.
\newblock {\em SIAM Review}, 38(4):565--604, 1996.

\bibitem{KR2004}
H.~Kokubu and R.~Roussarie.
\newblock Existence of a singularly degenerate heteroclinic cycle in the
  {L}orenz system and its dynamical consequences: {P}art {I}.
\newblock {\em Journal of Dynamics and Differential Equations}, 16(2):513--557,
  2004.

\bibitem{KK1990}
H.C. Kranzer and B.L. Keyfitz.
\newblock A strictly hyperbolic system of conservation laws admitting singular
  shocks.
\newblock In {\em Nonlinear evolution equations that change type}, pages
  107--125. Springer, 1990.

\bibitem{feigenbaum}
Oscar~E. Lanford, III.
\newblock A computer-assisted proof of the {F}eigenbaum conjectures.
\newblock {\em Bull. Amer. Math. Soc. (N.S.)}, 6(3):427--434, 1982.

\bibitem{LM1990}
A.C. Lazer and P.J. McKenna.
\newblock Large-amplitude periodic oscillations in suspension bridges: some new
  connections with nonlinear analysis.
\newblock {\em Siam Review}, 32(4):537--578, 1990.

\bibitem{code}
J.-P. Lessard, K.~Matsue, and A.~Takayasu.
\newblock Codes of \lq\lq {S}addle-{T}ype {B}low-{U}p {S}olutions with
  {C}omputer-{A}ssisted {P}roofs: {V}alidation and {E}xtraction of {G}lobal
  {N}ature".
\newblock {\tt https://github.com/taklab-org/GC-ubs-CAP}.

\bibitem{MR3148084}
J.-P. Lessard and C.~Reinhardt.
\newblock Rigorous {N}umerics for {N}onlinear {D}ifferential {E}quations
  {U}sing {C}hebyshev {S}eries.
\newblock {\em SIAM J. Numer. Anal.}, 52(1):1--22, 2014.

\bibitem{Lohner}
R.~J. Lohner.
\newblock Enclosing the solutions of ordinary initial and boundary value
  problems.
\newblock In E.~Kaucher, U.~Kulisch, and Ch. Ullrich, editors, {\em {Computer
  Arithmetic, Scientific Computation and Programming Languages}}, pages
  255--286,. B.G.Teubner, 1987.

\bibitem{Mat2018}
K.~Matsue.
\newblock On blow-up solutions of differential equations with
  {P}oincar\'{e}-type compactifications.
\newblock {\em SIAM Journal on Applied Dynamical Systems}, 17(3):2249--2288,
  2018.

\bibitem{Mat2019}
K.~Matsue.
\newblock Geometric treatments and a common mechanism in finite-time
  singularities for autonomous {ODE}s.
\newblock {\em Journal of Differential Equations}, 267(12):7313--7368, 2019.

\bibitem{MT2020_1}
K.~Matsue and A.~Takayasu.
\newblock Numerical validation of blow-up solutions with quasi-homogeneous
  compactifications.
\newblock {\em Numerische Mathematik}, 145:605--654, 2020.

\bibitem{MT2020_2}
K.~Matsue and A.~Takayasu.
\newblock Rigorous numerics of blow-up solutions for {ODE}s with exponential
  nonlinearity.
\newblock {\em Journal of Computational and Applied Mathematics}, 374:112607,
  2020.

\bibitem{MR3792792}
J.D. Mireles-James.
\newblock Validated numerics for equilibria of analytic vector fields:
  invariant manifolds and connecting orbits.
\newblock In {\em Rigorous numerics in dynamics}, volume~74 of {\em Proc.
  Sympos. Appl. Math.}, pages 27--80. Amer. Math. Soc., Providence, RI, 2018.

\bibitem{M2016}
N.~Mizoguchi.
\newblock Type {II} blowup in a doubly parabolic {K}eller-{S}egel system in two
  dimensions.
\newblock {\em Journal of Functional Analysis}, 271(11):3323--3347, 2016.

\bibitem{NAKAO_VerifiedPDE}
M.T. Nakao.
\newblock Numerical verification methods for solutions of ordinary and partial
  differential equations.
\newblock {\em Numerical Functional Analysis and Optimization},
  22(3-4):321--356, 2001.

\bibitem{MR3971222}
M.T. Nakao, M.~Plum, and Y.~Watanabe.
\newblock {\em Numerical verification methods and computer-assisted proofs for
  partial differential equations}, volume~53 of {\em Springer Series in
  Computational Mathematics}.
\newblock Springer, Singapore, [2019] \copyright 2019.

\bibitem{NZ2015}
N.~Nouaili and H.~Zaag.
\newblock Profile for a simultaneously blowing up solution to a complex valued
  semilinear heat equation.
\newblock {\em Communications in {P}artial {D}ifferential {E}quations},
  40(7):1197--1217, 2015.

\bibitem{PT2012}
L.A. Peletier and W.C. Troy.
\newblock {\em Spatial patterns: higher order models in physics and mechanics},
  volume~45.
\newblock Springer Science \& Business Media, 2012.

\bibitem{Rump2015}
S.M. Rump and M.~Kashiwagi.
\newblock Implementation and improvements of affine arithmetic.
\newblock {\em Nonlinear Theory and Its Applications, IEICE}, 6(3):341--359,
  2015.

\bibitem{SSS1993}
D.G. Schaeffer, S.~Schecter, and M.~Shearer.
\newblock Nonstrictly hyperbolic conservation laws with a parabolic line.
\newblock {\em Journal of differential equations}, 103(1):94--126, 1993.

\bibitem{S2004}
S.~Schecter.
\newblock Existence of {D}afermos profiles for singular shocks.
\newblock {\em Journal of Differential Equations}, 205(1):185--210, 2004.

\bibitem{S2007}
M.~Sever.
\newblock {\em Distribution solutions of nonlinear systems of conservation
  laws}.
\newblock American Mathematical Soc., 2007.

\bibitem{TMSTMO2017}
A.~Takayasu, K.~Matsue, T.~Sasaki, K.~Tanaka, M.~Mizuguchi, and S.~Oishi.
\newblock Numerical validation of blow-up solutions for ordinary differential
  equations.
\newblock {\em Journal of Computational and Applied Mathematics}, 314:10--29,
  2017.

\bibitem{Tucker2002}
W.~Tucker.
\newblock A rigorous ode solver and smale's 14th problem.
\newblock {\em Foundations of Computational Mathematics}, 2(1):53--117, 2002.

\bibitem{TUCKER_ValidatedIntroduction}
W.~Tucker.
\newblock {\em Validated numerics: a short introduction to rigorous
  computations}.
\newblock Princeton University Press, 2011.

\bibitem{VANDENBERG_Dynamics}
J.B. van~den Berg and J.-P. Lessard.
\newblock Rigorous numerics in dynamics.
\newblock {\em Notices of the AMS}, 62(9):1057--1061, 2015.

\bibitem{MR2821596}
J.B. van~den Berg, J.D. Mireles-James, J.-P. Lessard, and K.~Mischaikow.
\newblock Rigorous numerics for symmetric connecting orbits: even homoclinics
  of the {G}ray-{S}cott equation.
\newblock {\em SIAM J. Math. Anal.}, 43(4):1557--1594, 2011.

\bibitem{W2013}
M.~Winkler.
\newblock Finite-time blow-up in the higher-dimensional parabolic-parabolic
  {K}eller-{S}egel system.
\newblock {\em Journal de Math{\'e}matiques Pures et Appliqu{\'e}es},
  100(5):748--767, 2013.

\bibitem{Zgli}
P.~Zgliczynski.
\newblock {$C^1$ Lohner Algorithm}.
\newblock {\em Foundations of Computational Mathematics}, 2(4):429--465, 2002.

\bibitem{ZCov2009}
P.~Zgliczy\'{n}ski.
\newblock Covering relations, cone conditions and the stable manifold theorem.
\newblock {\em J. Differential Equations}, 246(5):1774--1819, 2009.

\bibitem{ZS2017}
G.~Zhou and N.~Saito.
\newblock Finite volume methods for a {K}eller-{S}egel system: discrete energy,
  error estimates and numerical blow-up analysis.
\newblock {\em Numerische Mathematik}, 135(1):265--311, 2017.

\end{thebibliography}

\end{document}